\documentclass[11pt]{amsart}

\usepackage{amsmath}
\usepackage{amssymb}

\usepackage{amsfonts}
\usepackage{amsthm}
\usepackage{enumerate}
\usepackage[mathscr]{eucal}

\usepackage{tikz}

\usepackage{bbm}
\usepackage{subfigure}

\usepackage{graphicx}
\usepackage{verbatim}
\newcommand{\rhoop}{\text{\rm Op}}

\newcommand {\Span} {\operatorname{span}}
\newcommand{\vsig}{\varsigma}

\newcommand{\sB}{\mathscr{B}}

\newcommand{\Be}{\begin{equation}}
\newcommand{\Ee}{\end{equation}}

\newcommand{\Ba}[1]{\begin{array}{#1}}
\newcommand{\Ea}{\end{array}}
\newcommand{\Bea}{\begin{eqnarray}}
\newcommand{\Eea}{\end{eqnarray}}
\newcommand{\Beas}{\begin{eqnarray*}}
\newcommand{\Eeas}{\end{eqnarray*}}
\newcommand{\Benu}{\begin{enumerate}}
\newcommand{\Eenu}{\end{enumerate}}
\newcommand{\Bi}{\begin{itemize}}
\newcommand{\Ei}{\end{itemize}}

\def\intslash{\rlap{\kern  .32em $\mspace {.5mu}\backslash$ }\int}
\def\qsl{{\rlap{\kern  .32em $\mspace {.5mu}\backslash$ }\int_{Q_x}}}

\def\emph#1{{\it #1 }}

\def\Ga{\Gamma}
\def\ga{\gamma}

\def\cf{{\it cf}}

\def\dist{{\text{\it dist}}}

\def\supp{{\text{\rm supp}}}

\def\inn#1#2{\langle#1,#2\rangle}
\def\biginn#1#2{\big\langle#1,#2\big\rangle}

\def\noi{\noindent}

\def\meas{{\text{\rm meas}}}

\def\lc{\lesssim}
\def\gc{\gtrsim}

\def\eps{\varepsilon}
\def\ep{\epsilon}

\def\ka{\kappa}
            
             \def\La{\Lambda}

\def\vphi{\varphi}
             
              \def\Om{\Omega}

\def\fM{{\mathfrak {M}}}

\def\fP{{\mathfrak {P}}}
\def\fQ{{\mathfrak {Q}}}

\def\fa{{\mathfrak {a}}}

\def\fz{{\mathfrak {z}}}

\def\bbC{{\mathbb {C}}}

\def\bbE{{\mathbb {E}}}

\def\bbN{{\mathbb {N}}}

\def\bbR{{\mathbb {R}}}

\def\bbT{{\mathbb {T}}}

\def\bbZ{{\mathbb {Z}}}

\def\sD{{\mathscr {D}}}

\def\sH{{\mathscr {H}}}

\def\cA{{\mathcal {A}}}

\def\cD{{\mathcal {D}}}

\def\cH{{\mathcal {H}}}
\def\cI{{\mathcal {I}}}

\def\cM{{\mathcal {M}}}
\def\cN{{\mathcal {N}}}

\def\cS{{\mathcal {S}}}

\def\cU{{\mathcal {U}}}

 at 10 true pt

\def\be#1{\begin{equation}\label{ #1}}
\def\endeq{\end{equation}}
\def\endal{\end{align}}
\def\bas{\begin{align*}}
\def\eas{\end{align*}}
\def\bi{\begin{itemize}}
\def\ei{\end{itemize}}

\def\eps{\varepsilon}

\def\emph#1{{\it #1}}
\def\textbf#1{{\bf #1}}

\usepackage{bbm}
\def\bbone{{\mathbbm 1}}

\theoremstyle{plain}
  \newtheorem{theorem}{Theorem}[section]

   \newtheorem{proposition}[theorem]{Proposition}
   
   \newtheorem{lemma}[theorem]{Lemma}
   \newtheorem{corollary}[theorem]{Corollary}

\theoremstyle{remark}
   \newtheorem{remark}[theorem]{Remark}
   
   \newtheorem*{Remark}{Remark}
   
\theoremstyle{definition}
   \newtheorem{definition}[theorem]{Definition}

\newcommand {\lBs} {{(B^s_{p,q})_{\ell^p}}}
\newcommand {\SE} {{\mathbb E}}
\newcommand {\SN} {{\mathbb N}}
\newcommand {\SR} {{\mathbb R}}

\newcommand {\SZ} {{\mathbb Z}}
\newcommand {\SNz} {{\mathbb \SN_0}}
\newcommand {\SRd} {{\mathbb R^d}}
\newcommand {\SZd} {{\mathbb Z^d}}

\newcommand {\e} {{\varepsilon}}

\newcommand {\hpsi} {{\widehat\psi}}
\newcommand {\hphi} {{\widehat\vphi}}

\newcommand {\tzeta} {{\tilde\zeta}}
\newcommand {\tchi} {{\tilde\chi}}

\newcommand {\dt} {{\delta}}
\newcommand{\Dt}{{\Delta}}
\newcommand{\ENp}{{\bbE_N^\perp}}

\newcommand {\Ds} {\displaystyle}

\newcommand {\mand} {{\quad\mbox{and}\quad}}
\renewcommand {\mid} {{\,\,\,\colon\,\,\,}}
\def\supp{\mathop{\rm supp}}
\def\dist{\mathop{\rm dist}}

\newcounter{reg}
\newcommand{\sline}{{\smallskip

\noindent}}

\newcommand{\bline}{{\medskip

\noindent}}

\def\Xint#1{\mathchoice
{\XXint\displaystyle\textstyle{#1}}%
{\XXint\textstyle\scriptstyle{#1}}%
{\XXint\scriptstyle\scriptscriptstyle{#1}}%
{\XXint\scriptscriptstyle\scriptscriptstyle{#1}}%
\!\int}
\def\XXint#1#2#3{{\setbox0=\hbox{$#1{#2#3}{\int}$ }
\vcenter{\hbox{$#2#3$ }}\kern-.6\wd0}}

\def\mint{\Xint-}

\begin{document}

\title 
{Basis properties of 
the Haar system  in limiting Besov spaces}

\author[G. Garrig\'os \ \ \ A. Seeger \ \ \ T. Ullrich] {Gustavo Garrig\'os   \ \ \ \   Andreas Seeger \ \ \ \ Tino Ullrich}

\address{Gustavo Garrig\'os\\ Department of Mathematics\\University of Murcia\\30100 Espinardo\\Murcia, Spain} \email{gustavo.garrigos@um.es}

\address{Andreas Seeger \\ Department of Mathematics \\ University of Wisconsin \\480 Lincoln Drive\\ Madison, WI,53706, USA} \email{seeger@math.wisc.edu}
\address{Tino Ullrich\\Faculty of Mathematics, Technical University of Chemnitz\\ 09107 Chemnitz, Germany} \email{tino.ullrich@mathematik.tu-chemnitz.de}

 %


\begin{abstract} 
We study Schauder basis properties for the Haar system in Besov spaces $B^s_{p,q}(\bbR^d)$. 
We give a complete description of the limiting cases, obtaining various positive results
for  $q\leq \min\{1,p\}$, and providing new counterexamples in other situations.  
The study is based on suitable estimates of the dyadic averaging operators $\SE_N$; 
in particular we find asymptotically optimal growth rates for the norms of these operators in global and local situations.  
\end{abstract}

 \maketitle

\section{Introduction}

The purpose of this paper is to  complete the study of the basis properties 
of the (inhomogeneous) Haar system in the scale of Besov spaces $B^s_{p,q}(\bbR^d)$. In view of previous results, only the endpoint cases are 
 of interest. This is a companion  to  the forthcoming paper \cite{gsu-endpt}, in which the authors consider the same endpoint questions for the Triebel-Lizorkin spaces. The outcomes for Besov and for  Triebel-Lizorkin spaces,  in both non-endpoint situations (\cite{triebel78, triebel-bases, su, sudet, gsu}) and endpoint situations, are markedly different.
 
To state the results, we first set some basic notation.
Consider the one variable functions
 $h^{(0)}=\bbone_{[0,1)}$ and $h^{(1)}=\bbone_{[0,1/2)}-\;\bbone_{[1/2,1)}$.
 For every $\ep=(\ep_1,\ldots,\ep_d)\in\{0,1\}^d$, for $k\in \bbN_0$ and $\mu=(\mu_1,\dots, \mu_d)\in \bbZ^d$ 
  one defines
\[
h^{\ep}_{k,\mu}(x)
= \prod_{i=1}^d h^{(\ep_i)}(2^kx_i-\mu_i),\quad x=(x_1,\ldots,x_d)\in\SR^d.
\]
Denoting $\Upsilon=\{0,1\}^d\setminus\{\vec 0\}$, the Haar system is  given as 
\[
\sH_d=\Big\{h^{\vec 0}_{0,\mu}\Big\}_{\mu \in\SZd}\cup \Big\{h^{\ep}_{k,\mu}\mid k\in\SNz,\;\mu\in\SZd,\;\ep\in\Upsilon\Big\}.
\]
We refer to $2^k$ as the Haar frequency of $h^{\ep}_{k,\mu}$.
We consider an enumeration   $\cU=\{u_n\}_{n=1}^\infty $ of the Haar system, and write  $u_n=h_{k(n),\mu(n)}^{\ep(n)}$ 
for the corresponding frequency and position parameters $k(n)$, $\mu(n)$.
 
Given $R\in \bbN$,  the partial sum operator $S_R\equiv S_R^\cU$ is defined as the projection onto $\text{span}\{u_1,\dots, u_R\}$, that is
\Be\label{SR} S_R^\cU f= \sum_{n=1}^R 
u_n^*(f) u_n \,,
\Ee
where  for $u_n= h_{k(n),\mu(n)}^{\ep(n)}$  the linear functional $u_n^*$
is defined by
\Be
u_n^* (f) =
2^{k(n)d} \langle f,h_{k(n),\mu(n)}^{\ep(n)}\rangle \,,
\label{unstar}
\Ee
at least when $f\in L^1_{loc}(\SR^d)$. Below we shall only consider Besov spaces so that
$u_n\in B^s_{p,q}$ and $u^*_n$ extends to an element of $(B^s_{p,q})^*$ for all $n\in \SN$,
so that \eqref{unstar} will actually have a meaning for all $f\in B^s_{p,q}$. 

We say that $\cU$ is a  {\it Schauder basis} of $B^s_{p,q}(\SR^d)$ if   
\Be\label{basicseqdef}\lim_{R\to\infty}\|S_R^\cU f-f\|_{B^s_{p,q}}=0
\Ee
holds for every $f\in B^s_{p,q}$. We say that $\cU$ is a {\it basic sequence} if \eqref{basicseqdef} holds for every $f$ in the $B^s_{p,q}$-closure of $ \Span\sH_d$. 
Finally, we say that $\sH_d $ is an \emph{unconditional basis} of $B^s_{p,q}$   if every  enumeration $\cU$ is a Schauder basis.

The above basis properties are related with the uniform bound \Be
\label{uniformSRbound}
 C_\cU:= \sup_{R\geq 1}\big\|S_R^\cU\big\|_{B^s_{p,q}\to B^s_{p,q}}<\infty.
\Ee 
Indeed, one has 
\[
\big\|S_R^{\cU}f-f\big\|_{B^s_{p,q}}\lesssim\,(C_\cU+1)\|f-h\|_{B^s_{p,q}}+\big\|S_R^{\cU}h-h\big\|_{B^s_{p,q}},\quad
h\in\Span\sH_d.
\]
Thus, \eqref{uniformSRbound} implies that $\cU$ is a basic sequence in $B^s_{p,q}$.
If $\text{span} (\sH_d )$ is dense in $B^s_{p,q}$, then $\cU$ is a Schauder basis if and only if 
\eqref{uniformSRbound} holds. 
If in addition  the bound in \eqref{uniformSRbound} does not depend on the enumeration $\cU$ then $\sH_d $ is an unconditional basis of $B^s_{p,q}$. By the uniform boundedness principle such a uniform estimate is also necessary for unconditionality. This is well-known for Banach spaces, and a proof for quasi-Banach spaces can be found in \cite{albiac}.

We consider the full range of indices $s\in\SR$ and $0<p,q\leq \infty$. When $p=\infty$ or $q=\infty$ the space $B^s_{p,q}$ is not separable, but in those cases one may consider the Schauder basis property in the $B^s_{p,q}$-closure of the Schwartz class $\cS$, which we will denote $b^s_{p,q}$ (as in \cite[Def 1.1.3]{RuSi96}).

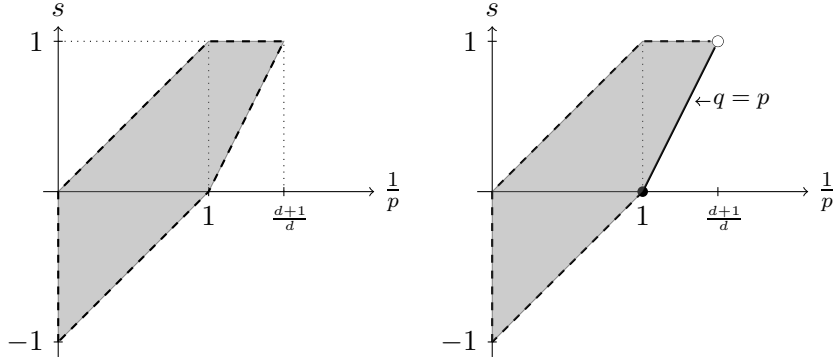
\begin{figure}[ht]
 \centering
\subfigure
{\begin{tikzpicture}[scale=2]
\draw[->] (-0.1,0.0) -- (2.1,0.0) node[right] {$\frac{1}{p}$};
\draw[->] (0.0,-0.0) -- (0.0,1.1) node[above] {$s$};
\draw (0.0,-1.1) -- (0.0,-1.0)  ;

\draw (1.0,0.03) -- (1.0,-0.03) node [below] {$1$};
\draw (1.5,0.03) -- (1.5,-0.03) node [below] {{\tiny $\;\;\frac{d+1}{d}$}};
\draw (0.03,1.0) -- (-0.03,1.00) node [left] {$1$};
\draw (0.03,-1.0) -- (-0.03,-1.00) node [left] {$-1$};

\draw[dotted] (1.0,0.0) -- (1.0,1.0);
\draw[dotted] (1.5,0.0) -- (1.5,1.0);
\draw[dotted] (0.0,1.0) -- (1.0,1.0);

\draw[fill=black!50, opacity=0.4]  (0.0,0.0) -- (1,1) -- (1.5,1)
-- (1.0,0.0) -- (0,-1) -- (0.0,0);

\draw[dashed, thick] (0.0,-1.0) -- (0.0,0.0) -- (1.0,1.0) -- (1.5,1.0) -- (1.0,0.0) --
(0.0,-1.0);
\end{tikzpicture}
}
\subfigure
{\begin{tikzpicture}[scale=2]
\draw[->] (-0.1,0.0) -- (2.1,0.0) node[right] {$\frac{1}{p}$};
\draw[->] (0.0,0.0) -- (0.0,1.1) node[above] {$s$};
\draw (0.0,-1.1) -- (0.0,-1.0)  ;
\draw (1.0,0.03) -- (1.0,-0.03) node [below] {$1$};
\draw (1.5,0.03) -- (1.5,-0.03) node [below] {{\tiny $\;\;\frac{d+1}{d}$}};
\draw (0.03,1.0) -- (-0.03,1.00) node [left] {$1$};
\draw (0.03,-1.0) -- (-0.03,-1.00) node [left] {$-1$};

\draw[dotted] (1.0,0.0) -- (1.0,1.0);
\draw[dashed, thick] (0.0,-1.0) -- (0.0,0.0) -- (1.0,1.0) -- (1.45,1.0);
\draw[dashed, thick] (1.0,0.0) -- (0.0,-1.0);
\draw[thick] (1.48,0.96) -- (1.0,0.0);
\fill (1,0) circle (1pt);
\draw (1.5,1) circle (1pt);

\draw[<-, thin] (1.35,0.6)--(1.45,0.6);
\node [right] at (1.4,0.6) {{\footnotesize $q=p$}};

\draw[fill=black!50, opacity=0.4]  (0.0,0.0) -- (1,1) -- (1.5,1)
-- (1.0,0.0) -- (0,-1) -- (0.0,0);
\fill[white] (1.5,1) circle (1pt);

\end{tikzpicture}
}
\caption{Parameter domain $\fP$  for  $\sH_d$ in $B^s_{p,q}(\SRd)$. The left figure shows the region of unconditionality, and right figure the region for the Schauder basis property.}\label{fig1}
\end{figure}

\

The pentagon $\fP$ depicted in Figure \ref{fig1} shows the natural index region for these problems. More precisely,  Triebel showed in \cite{triebel78, triebel-bases} that, for all $q<\infty$, the Haar system $\sH_d$ is an unconditional basis of $B^s_{p,q}(\SR^d)$ for the $(1/p,s)$ parameters in the interior of the pentagon $\fP$; {\it i.e.} those satisfying one of the conditions (i), (ii) in Theorem \ref{uncond-besov}. He also showed 
 that for the parameters in the complement of the closure of $\fP$ the Haar system does not form a basis. 
Except for a few trivial cases, the behavior at the points $(1/p,s)$ lying in the boundary of $\fP$ seems to be unexplored; see however the separate work  \cite{oswald} and Remark \ref{R_oswald} below.

In this paper we attempt to fill these gaps, by giving an answer, positive or negative, depending on the secondary index $q$. Moreover, in some cases, the negative answer is replaced by slightly weaker properties, such as the local Schauder basis, or basic sequence properties.


We begin by stating complete results about unconditionality, which 
contain new negative cases compared to \cite{triebel78, triebel-bases}. We remark that the corresponding results in Triebel-Lizorkin spaces are much more restrictive, see the discussion in \cite[Remark 2.2.10]{triebel-bases} and the counterexamples in \cite{su}.

\begin{theorem} \label{uncond-besov}
Let $0<p,q\leq \infty$ and $s \in \bbR$. 
Then, $\sH_d $ is an unconditional basis of $B^s_{p,q}(\SR^d)$ if and only if
one of the following two conditions is satisfied.

\sline (i) $\quad 1\le p<\infty$, $\quad -1+\frac1p< s< \frac1p$, $\;\quad 0<q< \infty$.

\sline (ii)  $\;\frac{d}{d+1} < p< 1$, $\quad d(\frac 1p-1)<s<1$, $\quad 0<q<\infty$.

\end{theorem}

 The region (i)-(ii) is shown in the left of Figure \ref{fig1}.
In the next results we shall be concerned with the endpoint behavior when we 
 drop unconditionality.  
To do so we must single out 
 specific enumerations of the Haar system, labeled `admissible', or `strongly admissible'.

 
 \smallskip
 
\begin{definition} \label{strongly-adm}
  {\em (i)} An enumeration $\cU$  is  said to be {\it admissible} if there is a constant $b\in \bbN$
with the following property: for each cube $I_{\nu} = \nu+[0,1]^d$, $\nu\in \bbZ^d$, if $u_n$ and $u_{n'}$ are both supported in $I_\nu$ and $|\supp(u_n)|\geq 2^{bd}|\supp(u_{n'})|$, then  necessarily $n<n'$\,.

{\em (ii)} 
 { An enumeration $\cU$ is {\it strongly admissible} if there is a constant $b\in \bbN$ with the following property:
 for each cube $I_\nu$, $\nu\in\SZ^d$, if $I^{**}_\nu$ denotes the five-fold dilated cube with respect to its center, and if $u_n$ and $u_{n'}$  are supported in  $I^{**}_\nu$ with $|\supp (u_n)| \ge 2^{bd}|\supp( u_{n'}) |$ then necessarily $n< n'$.}

\end{definition}
\smallskip

The notion in (i) was used in  \cite{gsu} for the case $b=1$, but the results stated in that paper continue to hold with this slightly more general definition. 

The stronger notion in (ii) turns out to be more appropriate in the endpoint cases, for which the characteristic functions of cubes may not be pointwise multipliers; cf. \cite{RuSi96}. 
Loosely speaking, in a strongly  admissible enumeration if a Haar frequency $2^k$ shows up at step $n$ (i.e. if $u_n=h^{\e}_{k,\mu}$ for some $k\in \SN_0$) then all Haar functions with Haar frequency $\le 2^{k-b}$ which are `nearby'  (in a well defined sense) have already been counted before step $n$. We refer to \S\ref{admissiblesection} for concrete examples. 

 Finally, we remark that the above distinction is void for the classical Haar system in the unit cube, $\sH([0,1)^d)$, where
admissibility is a straightforward property; for $b=1$ it means that $n< n'$ implies
$|\supp (u_n)| \ge |\supp( u_{n'}) |$  (and could be slightly weakened if $b\geq2$).  The typical example is the lexicographic ordering.

 \smallskip


We now  formulate a theorem involving {\it all} strongly admissible enumerations of the Haar system.
A positive endpoint result is obtained for $B^s_{p,p}$ when $s=d/p-d$ and $\frac{d}{d+1}<p\leq 1$. Also new negative results are obtained for suitable strongly admissible enumerations; see the right of Figure \ref{fig1}. 

   \begin{theorem} \label{schauder-besov} Let $0<p,q\leq \infty$ and $s \in \bbR$. Then, the following statements are equivalent, i.e. (a)$\iff$(b):
  
 \sline  (a) 
Every  strongly admissible  enumeration of $\sH_d $ is a  Schauder basis of $B^{s}_{p,q}(\bbR^d)$.

\bline (b) One  of the following three  conditions is satisfied:

\Benu \item[(i)\;\;] $\quad 1\le p< \infty$, $\quad\; \frac 1p -1< s< \frac1p$, $\quad 0<q< \infty$,

\item[(ii)\;]  $\quad \frac{d}{d+1} < p<  1$, $\quad \frac dp-d<s<1$, $\quad 0<q< \infty$,

\item[(iii)] $\quad \frac{d}{d+1} < p\le 1$, $\quad s=\frac dp-d$, $\quad q=p$.
\Eenu
\end{theorem} 

\

 Next we explore various weaker properties at the boundary of $\fP$. 
 We  say that an enumeration 
 $\cU$ satisfies the {\it  local Schauder basis property} for $B^s_{p,q}$  if  
 \Be \big\|(S_R^\cU f-f)\,\chi\big\|_{B^s_{p,q}}\to 0
\label{SRchi}
\Ee  for all $f\in B^s_{p,q}(\SR^d)$ and all $\chi\in C^\infty_c(\bbR^d)$.
This implies that the basis expansion holds, $g=\sum_{n=1}^\infty u^*_n(g)u_n$ in $B^s_{p,q}$,
 for  all compactly supported  $g\in B^s_{p,q}(\SR^d)$.
 Similarly we say that $\cU$ satisfies the \emph{local basic sequence property} in $B^s_{p,q}(\SR^d)$ when \eqref{SRchi} holds for all
$f\in {\overline{\Span\sH_d}}^{B^s_{p,q}}$ and all $\chi\in C^\infty_c(\SR^d)$.
The next theorem and Figure \ref{fig2} show the region of validity for the first of these properties.

\begin{theorem} \label{schauder-besov-local}
 Let $0<p,q\leq\infty$ and $s \in \bbR$. Then, the following statements are equivalent, i.e. (a)$\iff$(b):
  
\sline   (a) 
Every  strongly admissible  enumeration of the Haar system $\sH_d $ satisfies the local  Schauder basis property for  $B^{s}_{p,q}(\SR^d)$.

\sline (b) One  of the following four  conditions is satisfied:

(i) $\quad 1\le p< \infty$, $\quad -1+\frac 1p< s< \frac1p$, $\quad 0<q< \infty$,

(ii) $\quad 1\le p< \infty$, $\quad s= -1+\frac 1p$, $\quad 0<q\le 1$,


(iii)  $\quad \frac{d}{d+1} < p<1$, $\quad \frac dp-d<s<1$, $\quad 0<q< \infty$,

(iv) $\quad \frac{d}{d+1} < p<1$, $\quad s=\frac dp-d$, $\quad 0<q\le p$.
\end{theorem} 


We remark that these local properties can be given slightly  stronger statements using 
the Bourdaud definitions  
$(B^s_{p,q})_{\ell^p}$ of localized Besov spaces; see \S\ref{bourdaudsect} below.

	 
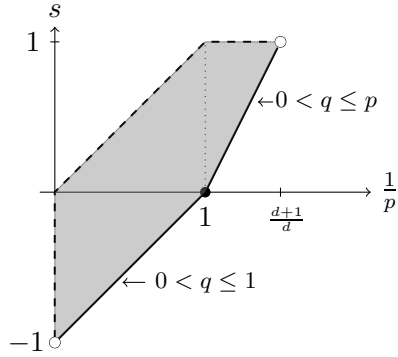
\begin{figure}[ht]
 \centering
{\begin{tikzpicture}[scale=2]
\draw[->] (-0.1,0.0) -- (2.1,0.0) node[right] {$\frac{1}{p}$};
\draw[->] (0.0,0.0) -- (0.0,1.1) node[above] {$s$};

\draw (1.0,0.03) -- (1.0,-0.03) node [below] {$1$};
\draw (1.5,0.03) -- (1.5,-0.03) node [below] {{\tiny $\;\;\frac{d+1}{d}$}};
\draw (0.03,1.0) -- (-0.03,1.00) node [left] {$1$};

\draw[dotted] (1.0,0.0) -- (1.0,1.0);

\draw[dashed, thick] (0.0,-0.97) -- (0.0,0.0) -- (1.0,1.0) -- (1.45,1.0);

\draw[thick] (1.48,0.96) -- (1.0,0.0);
\fill (1,0) circle (1pt);
\draw (1.5,1) circle (1pt);
\draw[<-, thin] (1.35,0.6)--(1.45,0.6);
\node [right] at (1.4,0.6) {{\footnotesize $0<q\leq p$}};

\draw[thick] (0.02,-0.98) -- (1.0,0.0);
\draw (0,-1) circle (1pt) node [left] {$-1$};
\draw[<-, thin] (0.45,-0.6)--(0.6,-0.6);
\node [right] at (0.6,-0.6) {{\footnotesize $0<q\leq 1$}};

\draw[fill=black!50, opacity=0.4]  (0.0,0.0) -- (1,1) -- (1.5,1)
-- (1.0,0.0) -- (0,-1) -- (0.0,0);
\fill[white] (1.5,1) circle (1pt);
\fill[white] (0,-1) circle (1pt);

\end{tikzpicture}
}

\caption{Region for the local Schauder basis property 
in the spaces $B^s_{p,q}(\SR^d)$, depending on the value of $q\in (0,\infty)$.}\label{fig2}
\end{figure}

\begin{remark}
Strong admissibility  may be replaced by admissibility (for the positive results in Theorems
 \ref{schauder-besov} and  \ref{schauder-besov-local}) in the cases where 
$\bbone_{[0,1]^d}$ is a pointwise multiplier of $B^s_{p,q}(\SR^d)$. 
By \cite[\S2.8.7]{Tr83} or \cite[\S4.6.3]{RuSi96}, the latter holds  when \Be
\max\Big\{d(\tfrac1p-1), \tfrac1p-1\Big\} < s<\tfrac1p,
\label{p_multipl}
\Ee
so it applies to the interior points of $\fP$. It also applies in other positive results (such as the local basic sequence property, which will follow from Theorem \ref{basic-sequence-besov}) in the case $s=1$, $\frac{d}{d+1} <p<1$ corresponding to the interior of the horizontal edge of $\fP$. 
\end{remark}


\begin{remark} \label{RemTd}
A similar statement to Theorem \ref{schauder-besov-local} holds for $\sH(\bbT^d)$, the Haar system  in the torus   in the  standard lexicographic enumeration. Namely, it is a Schauder basis on $B^s_{p,q}(\bbT^d)$ if and only if one of  the conditions   (i), (ii), (iii), (iv) in Theorem \ref{schauder-besov-local}  are satisfied. Moreover, in the range \eqref{p_multipl} 
the class of $C^\infty$-functions with compact support in $(0,1)^d$ are dense in 
$B^s_{p,q}((0,1)^d)$  (see \cite[\S3.2]{triebel-compl}) 
and thus it is easy to see that the Schauder basis problem for the Haar systems on $\bbT^d$ and on $(0,1)^d$  are equivalent in this range.  So far this observation does not  apply to the cases corresponding to the  non-horizontal edges of $\fP$ in higher dimensions, however see  Franke's better result \cite[\S4.6]{franke} for the interval $(0,1)$.
\end{remark}

\begin{remark} \label{R_oswald}
{\em (i)} In a classical work \cite{oswaldold}, P. 
Oswald considered, for $0<p<1$, the Schauder basis property (including some endpoint results) for a class of Besov spaces on the interval, ${\sB}^s_{p,q, (1)}(I)$, defined by first order differences. In these classes,
which in general differ from $B^s_{p,q}$, one has a positive answer in
the larger region $1/p-1< s<1/p$ (in particular, for some $s\geq1$); see \cite[Theorem 3]{oswaldold}.

{\em (ii)} In a very recent  separate study \cite{oswald},  Oswald pursued further these questions for both, the class ${\sB}^s_{p,q, (1)}(I^d)$ and the standard Besov spaces
 $B^s_{p,q}((0,1)^d)$. He obtained analogs of the positive results in 
(ii)-(iv) of Theorem \ref{schauder-besov-local}, and 
presented similar counterexamples as ours for the case $s=d/p-d$.
 Contrary to what is stated in that paper, these local results do not  transfer to the spaces on $\bbR^d$ by simply 
enumerating the Haar system, as one may see from Theorem \ref{schauder-besov} above and 
the specific example in Proposition \ref{admissiblecounterex}.

%
\end{remark}


\subsection*{\it Dyadic averaging operators.}  A crucial tool in our analysis will be the {dyadic averaging operator} $\bbE_N$, defined  as the  
conditional expectation 
with respect to the $\sigma$-algebra generated by the set $\sD_N$ of all dyadic cubes of length $2^{-N}$. 
That is, setting \[I_{N,\mu}=2^{-N}(\mu+[0,1)^d),\quad \mu\in\SZ^d,\] 
we have 
\Be\label{expect}\bbE_N f(x)=\sum_{\mu\in \bbZ^d} \bbone_{I_{N,\mu}}(x)\,2^{Nd} \,\int_{I_{N,\mu}}f(y) dy\,,\Ee
at least for $f\in L^1_{loc}(\SR^d)$.

The relation  with the Haar system is given via the  martingale difference operator $\bbE_{N+1}-\bbE_N$  which is the
orthogonal projection onto the space generated by the Haar functions with frequency $2^N$, i.e.
\Be\label{martdiff} 
\bbE_{N+1} f-\bbE_Nf
= \sum_{\ep\in\Upsilon}\sum_{\mu\in \bbZ^d} 2^{Nd} \inn{f}{ h^{\ep}_{N,\mu} }h^{\ep}_{N,\mu}.
\Ee

In addition to $\bbE_N$ we shall need another operator  which  involves Haar functions of a fixed frequency level.
For  $N\in \bbN$ and any $\fa\in \ell^\infty (\bbZ^d\times \Upsilon)$ we set
\Be\label{TNadef}
T_N[f,\fa] =\sum_{\ep\in \Upsilon} \sum_{\mu\in \bbZ^d} a_{\mu,\ep} 2^{Nd} 
\inn{f}{h^{\ep}_{N,\mu}}h^{\ep}_{N,\mu} .
\Ee

One aims for estimates of the operators $f\mapsto T_N [f,\fa]$ that  are uniform in $\|\fa\|_\infty\le 1$.
The relation between the partial sum operators $S_R^\cU$ and the operators $\bbE_N$ and $T_N[\cdot, \fa]$ 
is explained in \S\ref{localizationsect}. In particular, the uniform boundedness of these operators in $B^s_{p,q}$ 
implies the \emph{local basic sequence} property for all strongly admissible enumerations $\cU$. The region of uniform boundedness for these operators is given in the next theorems, and depicted in Figure \ref{fig3}.

\begin{theorem} \label{basic-sequence-besov} Let $0<p<\infty$, $0<q\leq \infty$ and $s\in \bbR$.

(a) The operators $\SE_N$ admit an extension from $\cS(\bbR^d)$ to $B^s_{p,q}(\SR^d)$ such that
$$\sup_{N\ge 0} \|\bbE_N\|_{B^s_{p,q}\to B^s_{p,q}} <\infty$$ 
if and only if one of the following six conditions is satisfied:

\medskip

(i) $\;\;\quad p>1$, $\quad s=\frac1p$, $\quad q=\infty$,

(ii) $\;\quad p \ge 1$, $\quad -1+\frac 1p< s< \frac1p$,  $\quad 0<q\le \infty$,

(iii) $\quad p \ge 1$, $\quad s= -1+\frac 1p$, $\quad 0<q\le 1$,

(iv)  $\!\quad \frac{d}{d+1} < p<1$, $\quad s=1$, $\quad 0<q\le p$,

(v)  $\;\quad \frac{d}{d+1} < p<1$, $\quad d(\frac 1p-1)<s<1$, $\quad 0<q\le \infty$,

(vi) $\quad \frac{d}{d+1} \le  p<1$, $\quad s=d(\frac 1p-1)$, $\quad 0<q\le p$.

\smallskip

\noi (b) If one of the conditions (i)-(vi) is satisfied and if $\|\fa\|_{\ell^\infty(\bbZ^d\times\Upsilon)}\le 1$ then the operators $T_N[\cdot,\fa]$ are uniformly bounded on $B^s_{p,q}(\bbR^d)$.

\end{theorem} 
  
  \smallskip
  
  Finally we state  a result for $p=\infty$.
 
  \begin{theorem}\label{p=inftythm}
  (i) If $-1<s<0$ then the operators  $\bbE_N$ have  uniformly bounded extensions to  $B^{s}_{\infty, q}(\SR^d)$, for all $0<q\le \infty$.
  
  (ii) If $s=0$ then $\bbE_N$ admits a bounded extension to  $B^0_{\infty,q}(\SR^d)$ if and only if $q=\infty$. Moreover, we have
    $\sup_N\|\bbE_N\|_{B^0_{\infty,\infty} \to B^0_{\infty,\infty} }<\infty$.
 
(iii) If $s=-1\;$ then $\;\sup_N\|\bbE_N\|_{B^{-1}_{\infty,q} \to B^{-1}_{\infty,q} }=\infty$, for all $0<q\leq\infty$. 
  \end{theorem}


\begin{figure}[ht]\label{uniformbdnessfigure}
 \centering
\subfigure
{\begin{tikzpicture}[scale=2]
\draw[->] (0.03,0.0) -- (2.1,0.0) node[right] {$\frac{1}{p}$};
\draw[->] (0.0,0.03) -- (0.0,1.1) node[above] {$s$};

\draw (1.0,0.03) -- (1.0,-0.03) node [below] {$1$};
\draw (1.5,0.03) -- (1.5,-0.03) node [below] {{\tiny $\;\;\frac{d+1}{d}$}};
\draw (0.03,1.0) -- (-0.03,1.00) node [left] {$1$};

\draw[dotted] (1.0,0.0) -- (1.0,1.0);

\draw[line width=1.1] (0.0,-0.03)--(0.0,-0.97);

\draw (0,0) circle (1pt);
\draw[dashed, thick] (0.02,0.02) -- (0.98,0.98) ;

\draw[thick] (1.03,1.0) -- (1.5, 1.0);
\draw[thick] (1.48,0.96) -- (1.0,0.0);
\fill (1,0) circle (1pt);
\fill (1.5,1) circle (1pt);
\draw (1,1) circle (1pt);
\draw[<-, thin] (1.35,0.6)--(1.45,0.6);
\node [right] at (1.4,0.6) {{\footnotesize $0<q\leq p$}};
\draw[<-, thin] (1.3,.95)--(1.45,0.6);

\draw[thick] (0.02,-0.97) -- (1.0,0.0);
\draw (0,-1) circle (1pt) node [left] {$-1$};
\draw[<-, thin] (0.45,-0.6)--(0.6,-0.6);
\node [right] at (0.6,-0.6) {{\footnotesize $0<q\leq 1$}};

\node [right] at (-0.04,-0.2) {{\tiny $0<q< \infty$}};
\draw[->, thin] (0.3,-0.28)--(0.3,-0.35)--(0.05,-0.35);

 \node [right] at (0.6,-0.6) {{\footnotesize $0<q\leq 1$}}; \node [right] at (0.0,-1.3) {{\footnotesize The cases $0<q<\infty$}};
\end{tikzpicture}
}
\subfigure
{\begin{tikzpicture}[scale=2]
\draw[->] (-0.1,0.0) -- (2.1,0.0) node[right] {$\frac{1}{p}$};
\draw[->] (0.0,-0.9) -- (0.0,1.1) node[above] {$s$};

\draw (1.0,0.03) -- (1.0,-0.03) node [below] {$1$};
\draw (1.5,0.03) -- (1.5,-0.03) node [below] {{\tiny $\;\;\frac{d+1}{d}$}};
\draw (0.03,1.0) -- (-0.03,1.00) node [left] {$1$};

\draw[dotted] (1.0,0.0) -- (1.0,1.0);
\draw[thick] (0.0,-0.97) -- (0.0,0.0) -- (0.98,0.98);
\draw[line width=1.1] (0.0,-0.03)--(0.0,-0.97);
\draw[dashed] (1.02,1) -- (1.5,1.0) -- (1.0,0.0) --
(0.02,-0.98);

\draw[dashed] (1.48,0.96) -- (1.0,0.0);
\draw (1,1) circle (1pt);

\draw[dashed] (0.02,-0.98) -- (1.0,0.0);
\draw (0,-1) circle (1pt) node [left] {$-1$};

\fill (0,0) circle (1pt);

\node [right] at (0.0,-1.3) {{\footnotesize The case $q=\infty$}};
\end{tikzpicture}
}
\caption{Regions for uniform boundedness of $\SE_N$  (hence for the \emph{local} basic sequence property) 
in the spaces $B^s_{p,q}(\SR^d)$.}
\label{fig3}
\end{figure}
Moreover, below we also investigate situations  when the individual operators $\bbE_N$ are bounded but not uniformly bounded, and derive precise growth conditions for the operator norms in such cases. See Theorem \ref{equiv-s=1thm} for complete  results in the case $s=1$, and Theorem \ref{locallowerboundsthm} for the case $s=d/p-d$ and $p\le 1$. A more detailed description of these and other local results is given the next subsection.

  \subsection*{\it Guide through this paper and discussion of further quantitative results.}
  The positive results in the interior of the pentagon $\fP$ in Figure \ref{fig1}, including the unconditionality property, are classical and due to Triebel \cite{triebel73, triebel78, triebel-bases}.
  Moreover, unboundedness results outside the closure of $\fP$ are discussed  in those references and \cite{gsu}.
  
  The new positive results in Theorems \ref{basic-sequence-besov} and \ref{p=inftythm} at the boundary of $\fP$  rely on $L^p$ bounds for the operators $L_k\bbE_N L_j$, where $L_k$ are suitable local means and the operators act on functions with compactly supported Fourier transforms. These bounds were already contained in our previous paper \cite{gsu} (see also \cite{gsu-wavelet} for a proof of such results using wavelets). We review these estimates  in \S\ref{preparatoryresults}, see in particular Corollary \ref{Upscor}.  For both ranges $p\ge 1$ and $p\le 1$ further  straightforward 
  estimates imply  four key propositions  with different outcomes in the four cases depending on the signs of $j-N$ and $k-N$. These propositions are stated  for $p\le 1$ in \S\ref{upbdsple1sect} 
  and for $p\ge 1$ in \S\ref{upbdspge1sect}.
  
 Concerning the negative results in Theorems \ref{basic-sequence-besov} and \ref{p=inftythm}, these are presented as follows. 
 First, when $s=1/p$, 
characteristic functions of cubes (and also  Haar functions)  do not belong to $B^{1/p}_{p,q}$ 
  when $q<\infty$ which rules out these cases. In \S\ref{necconds=1/p} we shall further show  that the space $b^{1/p}_{p,\infty}$ 
  (the closure of the Schwartz class under the $B^{1/p}_{p,\infty}$ 
  norm) intersects the algebraic span of $\sH_d$ only in  $\{0\}$. This is in contrast with the fact, shown in \S\ref{positiveresultondensity}, that $b^{1/p}_{p,\infty}$ is actually contained in ${\overline{\Span\sH_d}}^{B^{1/p}_{p,\infty}}$ if $1<p<\infty$, so some positive result will hold in this case; see Proposition \ref{P_supEN}.

In \S\ref{B1pq-bd-sect} we consider the cases $s=1$.
At the endpoint space $B^1_{1,\infty}$ we show that the operators $\bbE_N$ are individually bounded, 
  but not uniformly bounded, and for large $N$ we have
  $\|\bbE_N\|_{B^1_{1,\infty}\to B^1_{1,\infty}}\approx N$, see  Theorem  \ref{equiv-s=1thm}.
  
  When   $s=1$ and $\frac{d}{d+1}<p<1$ 
  the operators $\bbE_N$ are also individually bounded on $B^1_{p,q}$ 
  but not uniformly bounded if $q>p$. In these cases 
   Theorem  \ref{equiv-s=1thm} implies that  for  large $N$   we have 
  $\|\bbE_N\|_{B^1_{p,q}\to B^1_{p,q}}\approx N^{1/p-1/q}$. The situation is worse at the endpoint $p=d/(d+1)$, that is the vertex of $\fP$ where $s=1=d/p-d$. In  this case  Theorem \ref{locallowerboundsthm} gives an exponential lower bound even for compactly supported functions, while the $\bbE_N$ fail to be individually bounded in the whole $B^1_{d/(d+1),q}(\bbR^d)$ when $q>p=\frac d{d+1}$.

  In \S\ref{leftdottedlines-sect} we discuss 
the simpler situation on the line $s=1/p-1$ with $1<p\le \infty$.
The cases $q>1$ are easily ruled out because Haar functions do not belong to the dual space $(B^{1/p-1}_{p,q})^*=B^{1/p'}_{p', q'}$.
In the cases $0<q\leq1$, a lower bound $\|\SE_N\|_{B^{-1}_{\infty,q}\to B^{-1}_{\infty,q}}\gtrsim N$ is obtained by duality in \S\ref{Ss-1}.
  
In \S\ref{ple1slowerendpt} we gather the negative results for Theorem \ref{basic-sequence-besov} at the edge $s=d/p-d$  with  $p\leq 1$. 
Again, we rule out the cases  $q>1$, as Haar functions do not belong to the dual space $(B^{d/p-d}_{p,q})^*=B^0_{\infty, q'}$. Moreover, we show in Theorem \ref{noestimateforq>p} that the  individual operators $\bbE_N$ are unbounded on $B^{d/p-d}_{p,q}$ when $q>p$.  We shall actually prove sharp results if one restricts to compactly supported functions.
  To quantify these we use the following definition.
  
\begin{definition}\label{Opdefinition}
Let $Q$ be an open dyadic cube in $\bbR^d$ of side length $\ge 1/2$ and $X$ be a (quasi-)Banach space of tempered distributions $\mathcal{S}'(\bbR^d)$. For a 
linear operator $T$ defined on those $f\in X$ which are supported in $Q$ we set \Be \label{localnorms}
 \rhoop(T,X,Q)=   \sup\big\{ \|Tf\|_{X} 
 : \,\,
\|f\|_{X}\le 1,\, \supp(f)\subset Q\big\}  .
\Ee
\end{definition}

In Theorem \ref{locallowerboundsthm} the precise growth of 
$\rhoop(\bbE_N,B^{d/p-d}_{p,q}, Q) $ is obtained for the range
 $\frac{d}{d+1}\le p\le q\leq 1$ 
where it is shown that
  \[ \rhoop\big(\bbE_N,B^{d/p-d}_{p,q}, Q\big)\approx  \big(2^{Nd}|Q|\big)^{\frac 1p-\frac 1q}.\]
  We also show that the above-mentioned lower bounds for $s=1$ have a sharp local analogue,
  namely for every cube $Q$ of sidelength $\geq1$ one has 
  \[
	\rhoop\big(\bbE_N, B^1_{p,q},Q\big)\approx N^{\frac1p-\frac1q},\quad 
 \tfrac{d}{d+1}<p<1,\quad q\in[p,\infty].
\]

	We now turn to the uniform boundedness of the operators $S_R^\cU$, associated with strongly admissible enumerations $\cU$.
    We prove in \S\ref{SS91} that if 
	$\SE_N$ (and $T_N[\cdot,\fa]$) are uniformly bounded in $B^s_{p,q}$
		then we also 
     have, for each fixed $Q$, 
      \Be\label{localopnorm}
      \sup_{R\in \bbN} \rhoop\big(S^\cU_R, B^s_{p,q}, Q\big)<\infty.
      \Ee
      Assuming \eqref{localopnorm}, one    
     has  the local Schauder basis property if and only if 
       the span of the Haar system is dense in $B^s_{p,q}$. 
       
      The density of $\Span\sH_d$ in $B^s_{p,q}$ is studied separately in \S\ref{densityandapprsect}.
			It clearly fails when $p=\infty$ or $q=\infty$ because $B^s_{p,q}$ is not separable in those cases.
       When $s=1$, we also show that density fails in $B^1_{p,q}$ when $\frac{d}{d+1}\leq p<1$ and $q\le p$. This gives the negative results in Theorem \ref{schauder-besov-local} for those cases. We do not know, however, whether density should also fail in the remaining cases $q>p$; see our discussion in \S\ref{failureofdensitysect}.

  %

 The positive Schauder basis  results in Theorem \ref{schauder-besov} 
are obtained in \S\ref{bourdaudsect}. They follow from \S\ref{SS91} and the fact that the $B^s_{p,q}(\SR^d)$-norms can be `localized' if and only if $p=q$. Moreover, we actually prove the Schauder basis (resp. basic sequence) property in the Bourdaud spaces $(B^s_{p,q})_{\ell^p}$, in the range of Theorem \ref{schauder-besov-local} (resp. \ref{basic-sequence-besov} and \ref{p=inftythm}). We remark that these spaces coincide with $B^s_{p,q}$ if and only if $p=q$.  Alternatively, the positive statement in (iii) of Theorem \ref{schauder-besov} is also a special case of a more general result for  
     Triebel-Lizorkin spaces $F^s_{p,q}(\SR^d)$ in \cite{gsu-endpt}.

 In \S\ref{admissiblesection} we construct an explicit strongly admissible enumeration $\cU$ for which 
\[
S^\cU_{R(m)} f=\bbE_m f \quad \text{ if } \supp (f) \subset (-5,5)^d,
\]
for a suitable sequence  $R(m)$. One can then apply the examples
     on unboundedness of $\bbE_N$ when restricted to functions on cubes, alluded to above, to see that
     \eqref{localopnorm} fails. This gives the negative results in Theorems \ref{schauder-besov-local}  and  \ref{schauder-besov}
	at the edge $s=d/p-d$, for $\frac{d}{d+1}<p\le 1$ and all $q>p$.

     This same enumeration $\cU$ is used in  \S\ref{failurestronglyadmsect} to show 
     that, when $q\in(0,p)$, the operators $S_R^\cU$ are not uniformly bounded in the whole spaces 
		$B^{d/p-d}_{p,q} (\bbR^d)$ if $\frac{d}{d+1}<p\le 1$, or $B^{1/p-1}_{p,q} (\bbR^d)$ if $1<p<\infty$. 
		Hence $\cU$ is not a Schauder basis in these cases.      
  
  Finally, regarding the negative results in Theorem \ref{uncond-besov},  examples showing 
the failure of unconditionality  for  parameters $(1/p,s)$ on the boundary  of $\fP$ are given
in  \S\ref{uncondfailurep<1} 
for the case $B^{d/p-d}_{p,q}$ with $\frac{d}{d+1}<p\le 1$, and in
 \S\ref{uncondfailurep>1}  for the case $B^{1/p-1}_{p,q}$ with $1<p<\infty$. Since the argument 
 in \S\ref{uncondfailurep<1}  
 also applies to a similar result for Triebel-Lizorkin spaces, we include lower bounds for those as well.
  
\subsubsection*{Acknowledgements}

The authors would like to thank the Isaac Newton Institute for Mathematical Sciences, Cambridge, for support and hospitality during the program 
 \emph{Approximation, Sampling and Compression in Data Science} where some work on this paper was undertaken. This work was supported by EPSRC grant no EP/K032208/1. G.G. was supported in part by grants MTM2016-76566-P, MTM2017-83262-C2-2-P and Programa Salvador de Madariaga PRX18/451 from Micinn (Spain), and 
grant 20906/PI/18 from Fundaci\'on S\'eneca (Regi\'on de Murcia, Spain). A.S. was supported in part by National Science Foundation grants DMS 1500162 and 1764295. 
T.U. was supported in part by 
 Deutsche Forschungsgemeinschaft (DFG), grant 403/2-1.
 	
	\newpage
  
\section{Preparatory results}\label{preparatoryresults}

\subsection{\it Besov quasi-norms}
\label{S_qnorm}
Let $s\in\SR$ and $0<p\leq \infty$ be given. Throughout the paper we fix a number $A>d/p$ and an integer 
\Be \label{Mcondition} M >A+|s|+2 . 
\Ee 
Consider two functions $\beta_0, \beta \in C^\infty_c(\SRd)$, 
supported in $(-1/2,1/2)^d$, with the properties
 $|\widehat{\beta}_0(\xi)|>0$ if  $|\xi|\leq1$, 
 $|\widehat{\beta}(\xi)|>0$ if $1/8\leq|\xi|\leq1$ and $\beta$ has vanishing
moments up to order $M$, that is
\Be\label{moments-M}
\int_{\SR^d} \beta(x) \;x_1^{m_1}\cdots x_d^{m_d}\,dx =0, \quad\forall\;m_i\in \bbN_0\;\mbox{ with  $\;m_1+\ldots +m_d\leq M$.}
\Ee
The optimal value of $M$ is irrelevant for the purposes of this paper, and \eqref{Mcondition} suffices for our results.
We let  $\beta_k:=2^{kd}\beta(2^k\cdot)$ for each $k\geq1$, and  denote 
\[L_k(f)=\beta_k*f\]
whenever $f\in \cS'(\SR^d)$. It is then known, see e.g. \cite[2.5.3]{triebel2}, that an equivalent quasi-norm in the Besov spaces $B^s_{p,q}(\SR^d)$, $0<q\leq\infty$, is given by
\Be
\big\|g\big\|_{B^s_{p,q}}\,\approx\,\Big\|\big\{2^{ks}L_k g\big\}_{k=0}^\infty\Big\|_{\ell^q(L^p)}\, .
\label{Bspq_localmeans}\Ee
Recall also that  $b^{s}_{p,q}$ denotes the closure of $\cS(\SR^d)$ in the $B^s_{p,q}$ norm.
When $p<\infty$ and $q=\infty$, it not difficult to see that $b^{s}_{p,\infty}$ coincides with the set of all $g\in\cS'(\SR^d)$ such that
\Be
\lim_{k\to\infty}2^{ks}\|L_kg\|_p=0.
\label{bsinfty}
\Ee
The space $b^{s}_{\infty,\infty}$ coincides with the set of all $g\in\cS'(\SR^d)$ such that $L_k g\in C_0$ for all $k\in \bbN$ and such that 
$\lim_{k\to\infty}2^{ks}\|L_kg\|_\infty=0$.

Next, let $\eta_0\in C^\infty_c(\SRd)$ be supported on $\{\xi: |\xi|<3/8\}$ and such that $\eta_0(\xi)=1$ if $|\xi|\leq1/4$. We consider the following frequency localization operators 
\begin{subequations} 
\begin{align}\label{La0def}
\widehat{\La_0 f}(\xi) &=\frac{\eta_0(\xi)}{\widehat \beta_0(\xi)}\widehat f(\xi)\,,
\\
\label{Lakdef}
\widehat{\La_k f}(\xi) &=\frac{\eta_0(2^{-k}\xi) -\eta_0(2^{-k+1}\xi)}{\widehat \beta(2^{-k}\xi)}\widehat f(\xi), \quad k\ge 1,
\end{align}
\end{subequations}
so that  $f=\sum_{j=0}^\infty L_j \La_j f $  with convergence in $\cS'$. 
It is also well-known that
\Be  \big\|f\big\|_{B^s_{p,q}}\,\approx\,\Big\|\big\{2^{ks}\La_k f\big\}_{k=0}^\infty\Big\|_{\ell^q(L^p)}\, .\label{Bspq_freq}\Ee
In particular, if we let $\Pi_N=\sum_{j=0}^NL_j\La_j$, then 
\Be\sup_N\|\Pi_N f\|_{B^s_{p,q}}\lesssim \|f\|_{B^s_{p,q}}.\label{PiN}\Ee

Below we shall be interested in uniformly bounded extensions of the dyadic averaging operators $\SE_N$ defined in \eqref{expect}.
Observe that
\[
\SE_N-\Pi_N=\SE_N\,(I-\Pi_N)\,-\,(I-\SE_N)\,\Pi_N,
\]
so if we denote 
\[ \bbE_N^\perp= I-\bbE_N,\]
then, using \eqref{Bspq_localmeans}, we have
\Bea
\big\|\SE_N f-\Pi_N f\big\|_{B^s_{p,q}} & \lesssim &  
\;\Big\|\big\{2^{ks}\sum_{j=N+1}^\infty L_k\SE_NL_j\La_jf\big\}_{k=0}^\infty\Big\|_{\ell^q(L^p)}
\;+ \label{SENPiN}\\
& & +
\;\;\Big\|\big\{2^{ks}\sum_{j=0}^N L_k\SE^\perp_NL_j\La_jf\big\}_{k=0}^\infty\Big\|_{\ell^q(L^p)}.\nonumber
\Eea
Thus, as in \cite{gsu}, the uniform bounds of $\SE_N$ will be reduced to suitable estimates for the compositions $L_k\bbE_N L_j$ and $L_k\bbE_N^\perp L_j$, for each $j,k\geq0$.

\subsection{\it Local estimates}
We consider the following Peetre maximal operators: if $j\geq0$ and $g$ is continuous, then
\begin{align*}
\cM_j g(x)&= \sup_{|h|_\infty \le 2^{-j+5}} |g(x+h)|,
\\
\cM_{A,j}^{*} g(x)&
= \sup_{|h|_\infty \le 2^5} \frac{|g(x+h)|}{(1+2^j |h|)^A},
\\
\fM_{A,j}^{**} g(x)&= \sup_{h\in \bbR^d} \frac{|g(x+h)|}{(1+2^j |h|)^A}.
\end{align*}
Clearly, we have the pointwise relations 
$\cM_j g\lc \cM_{A,j}^*g\le\fM^{**}_{A,j} g$.

The following lemma was proved in \cite[Lemma 2.2]{gsu} using the cancellation properties of $L_{\max\{j,k\}}$.

\begin{lemma}\label{LjLk}
For $j,k\ge 0$ we have
\Be\label{LjLkpt} |L_k L_j g(x) |\lc 2^{-|k-j|(M-A)} \cM_{A,\max\{j,k\}}^{*} g(x).\Ee
\end{lemma}

 We remark that
 the larger maximal function
$\fM_{A,\max\{j,k\}}^{**}f $ was used in \cite[Lemma 2.2]{gsu}, 
in place of $\cM_{A,\max\{j,k\}}^{*}f$.
However, since the convolution kernel  of $L_jL_k$ 
is supported on a cube of sidelength $2^{-j}+2^{-k}$, it is clear that \eqref{LjLkpt}  holds as well. 

From  our previous work \cite{gsu} we  have the following crucial estimates.

 \begin{proposition} \label{localizedproposition}  Let $0<p\leq \infty$ and
 \Be \label{BjkN}
 B_p(j,k,N)= \begin{cases} 
 2^{N-j}\,2^{\frac{j-k}p}\,
 2^{(j-N)(d-1)(\frac1p-1)_+}  &\text{ if } j,k>N,
 \\
 2^{\frac{N-k}p} 2^{j-N}  &\text{ if } j\le N<k,
 \\
  2^{k-N} 2^{j-N}2^{(N-k)d(\frac1p-1)_+} &\text{ if } 0\le j,k\le N,
  \\
  2^{k-j+\frac{j-N}p+[N-k+(j-k)(d-1)](\frac1p-1)_+} &\text{ if }  k\le N<j.
 \end{cases} 
 \Ee
 Then the following inequalities hold for all continuous functions $g$\,:
\sline  (i) For $j\ge N+1$,
  \begin{multline} \label{LkENLjgeN}
  \|L_k\bbE_N[L_j  g]\|_p 
  \\ \lc
 \begin{cases}
   B_p(j,k,N) 
   \|   \cM_{j}g   \|_p 
   &\text{ if } k\ge N+1,
   \\
  B_p(j,k,N)   \|   \cM_{j} g   \|_p + 2^{-|j-k|(M-A)}\|\cM^{*}_{A,j} g
  \|_p &\text{ if } 0\le k\le N.
  \end{cases}
  \end{multline}
  
  \sline (ii) For $0\le j\le N$,
 \begin{multline}
\label{LkENLjleN}
  \|L_k\ENp [L_jg]\|_p \\
  \lc \begin{cases}  B_p(j,k,N)\|\cM_j g\|_p+ 2^{-|j-k|(M-A)} \|\cM_{A,j}^{*} g\|_p
  &\text{ if } k\ge N+1,
  \\ B_p(j,k,N) \|\cM_{j} g\|_p
  &\text{ if } 0\le k\le N.
  \end{cases}
\end{multline}

\sline (iii) The same bounds hold if the operators $\bbE_N$ in (i) and 
$\bbE_N^\perp$ in (ii) are replaced by $T_N[\cdot,\fa]$
(as defined in \eqref{TNadef}), uniformly in $\|\fa\|_\infty\le 1$.
\end{proposition}

{\it Remark.} These bounds are contained in \cite[\S2]{gsu}, although
the statement of \cite[Proposition 2.1]{gsu} is slightly less general.
Namely, applying these bounds to $g\in \cS'(\bbR^d)$ such that $\supp \widehat g\subset\{|\xi|\leq 2^{j+1}\}$,
and using additionally the Peetre inequality
$ \|\fM_{A, j}^{**}  g\|_p\lc \|g\|_p$, for $A>d/p$, one obtains \cite[Proposition 2.1]{gsu}. 
The formulation here will be applied later to functions of the form $g=\zeta\,\La_j f$ with $f\in \cS'(\bbR^d)$ and
 $\zeta\,\in C^\infty_c$. \qed

\

The statement of Proposition \ref{localizedproposition} can be put into a more convenient form.
First, when $g=\Lambda_j f$, the Peetre maximal inequality gives
$\|\fM_{A,j}^{**}(\La_j f)\|_p\lc\|\La_j f\|_p$ provided that $A>d/p$. Next, if $M\geq A+1$ then in 
the cases 
$k\le N< j$ 
and
$j\le N< k$  the term $2^{-|j-k|(M-A)}$ is dominated by $B_p(j,k,N)$ and thus the statement can be simplified. 
Finally, we shall use the quantities
\Be
U_{p,s}(j,k,N):= 2^{(k-j) s} B_p(j,k,N)\;;
\label{Ups_def}
\Ee
see also \eqref{BjkN-ple1} and \eqref{BjkN-pge1} below.
We then obtain

\begin{corollary} \label{Upscor}
Let $U_{p,s}(j,k,N)$ be as in \eqref{Ups_def}.
Then for all $f\in \cS'(\bbR^d)$
\Be\label{Upscorrfirst}
2^{ks} \|L_k\bbE_N L_j \La_j f \|_p \lc
U_{p,s}(j,k,N)\, 2^{js} \,\|\La_j f\|_p,\;\; \text{ if } j\ge N+1,
\Ee
and
\Be\label{Upscorrsecond}
2^{ks} \|L_k\bbE_N^\perp  L_j \La_j f \|_p \lc
U_{p,s}(j,k,N) \,2^{js}\, \|\La_j f\|_p,\;\; \text{ if } j\le N.
\Ee
The same holds with $\SE_N$ and $\SE^\perp_N$ replaced by $T_N[\cdot,\fa]$
if $\|\fa\|_\infty\le 1$.
\end{corollary}


  \section{Upper bounds  for  $p\leq 1$}\label{upbdsple1sect}
	
In the range $p\leq 1$, the constants in \eqref{Ups_def}
take the following explicit form 
  \Be \label{BjkN-ple1}
 U_{p,s}(j,k,N)
= \begin{cases}   2^{k(s-\frac 1p)}  2^{j(\frac dp-d-s)}  2^{N(d-\frac {d-1}{p})}
  &\text{ if } j,k>N 
 \\ 2^{k(s-\frac 1p)} 2^{j(1-s)} 2^{N(\frac 1p-1)}
  &\text{ if } j\le N<k 
 \\
 2^{k(s+d+1-\frac dp)}  2^{j(1-s)} 2^{N( \frac dp-d-2)}   &\text{ if } 0\le j,k\le N 
  \\2^{k(s+d+1-\frac dp) } 2^{j(\frac dp-d-s)} 2^{-N}&\text{ if } 
	\,\, k\le N <j.
 \end{cases}  \Ee
We now state four propositions corresponding to the 
four cases of \eqref{BjkN-ple1}. 
We then sketch the straightforward proofs.

\begin{proposition}\label{jk>N}
For $\frac{d-1}{d}<p\le 1$ and $r>0$,
\begin{subequations}
\begin{multline}\label{26a}
\Big(\sum_{k>N} 2^{ks r} \Big\|\sum_{j>N}L_k\bbE_N L_j\Lambda_j f\Big\|_p^r\Big)^{1/r}\\
\lc \begin{cases} \sup_{j>N} 2^{js} \|\La_j f\|_p &\text{ if } d(\frac 1p-1)<s<\frac 1p,
\\
\Big(\sum_{j>N}2^{jsp} \|\La_jf\|_p^p\Big)^{1/p}   &\text{ if } s=d(\frac 1p-1)<\frac1p.
\end{cases}
\end{multline}
For $p=1$ and $s=1$ we have
\Be
\sup_{k>N} 2^{k} \Big\|\sum_{j>N}L_k\bbE_N L_j\Lambda_j f\Big\|_1
\lc  \sup_{j>N} 2^{j} \|\La_j f\|_1 .
\Ee
\end{subequations}
The same inequalities hold  when $\bbE_N$ is replaced with 
 $T_N[\cdot,\fa]$ if $\|\fa\|_\infty\le 1$.
 \end{proposition}
 
 \begin{proposition}\label{jleNk>N}
For $0<p<1$ and $r>0$,
\begin{subequations}
\begin{multline}\label{jleNk>Nineq}
\Big(\sum_{k>N} 2^{ks r} \Big\|\sum_{j\le N} L_k\bbE_N ^\perp L_j\Lambda_j f\Big\|_p^r\Big)^{1/r}\\
\lc \begin{cases} \sup_{j\le N} 2^{js} \|\La_j f\|_p &\text{ if } s<1,
\\
\Big(\sum_{j=0}^N 2^{jsp} \|\La_jf\|_p^p\Big)^{1/p}   &\text{ if } s=1
\\
 2^{(s-1)N}\Ds\sup_{j\le N} 2^{js} \|\La_j f\|_p   & \text{ if } 1<s<1/p.
\end{cases}
\end{multline}
Inequality \eqref{jleNk>Nineq} also holds for $p=1$ and $s<1$.
When $p=s=1$ we have
\Be \label{jleNk>Nineqpartb} 
\sup_{k>N} 2^{k} \Big\|\sum_{j\le N} L_k\bbE_N ^\perp L_j\Lambda_j f\Big\|_1 
\lc \sum_{j=0}^N 2^{j} \|\La_jf\|_1
\Ee
\end{subequations} 
The same statements hold  with $\bbE_N^\perp$ replaced by
 $T_N[\cdot,\fa]$ if $\|\fa\|_\infty\le 1$.
 \end{proposition}

 \begin{proposition}\label{jleNkleN}
For $\frac d{d+2}<p\le 1$ and $r>0$,
\begin{multline}\label{jleNkleNineq}
\Big(\sum_{k\le N} 2^{ks r} \Big\|\sum_{j\le N} L_k\bbE_N^\perp L_j\Lambda_j f\Big\|_p^r\Big)^{1/r}\\
\lc \begin{cases} \sup_{j\le N} 2^{js} \|\La_j f\|_p &\text{ if } \frac dp-d-1<s<1,
\\
\Big(\sum_{j=0}^N 2^{jsp} \|\La_jf\|_p^p\Big)^{1/p}   &\text{ if } s=1
\\
 2^{(s-1)N}\sup_{j\le N} 2^{js} \|\La_j f\|_p  & \text{ if } 1<s<1/p.
\end{cases}
\end{multline}
The same inequality holds  with $\bbE_N^\perp$ replaced by 
 $T_N[\cdot,\fa]$ if $\|\fa\|_\infty\le 1$.
 \end{proposition}

 \begin{proposition}\label{j>NkleN}
For $0<p\le 1$ and $r>0$,
\begin{multline}\label{j>NkleNineq}
\Big(\sum_{k\le N} 2^{ks r} \Big\|\sum_{j>N}L_k \bbE_N L_j\Lambda_j f\Big\|_p^r\Big)^{1/r}\\
\lc \begin{cases} \sup_{j>N} 2^{js} \|\La_j f\|_p &\text{ if } s>\frac dp-d, \\
\Big(\sum_{j>N} 2^{jsp} \|\La_jf\|_p^p\Big)^{1/p}   &\text{ if } s=\frac dp-d .
\end{cases}
\end{multline}
The same inequality holds with $\bbE_N$ replaced by 
 $T_N[\cdot,\fa]$ if $\|\fa\|_\infty\le 1$.
 \end{proposition}
 
 \subsection*{\it Proofs} The proofs of the four propositions involve Corollary \ref{Upscor} and an application of the $p$-triangle inequality for $p\le 1$.

 \begin{proof}[Proof of Proposition \ref{jk>N}]
First observe that the range of $s$ in \eqref{26a} is nontrivial if and only if $p>d/(d-1)$.
Let $\Sigma_r$ denote the left hand side of \eqref{26a}. Then the $p$-triangle inequality and Corollary \ref{Upscor} give
\Beas
\Sigma_r&\leq &\Big(\sum_{k>N}2^{ksr}\Big[\sum_{j>N}\|L_k\bbE_N L_j\Lambda_j f\|_p^p\Big]^\frac rp\Big)^{1/r}\\
& \lesssim &
\Big(\sum_{k>N}\Big[\sum_{j>N}U_{p,s}(j,k,N)^p\;2^{jsp}\|\La_j f\|_p^p\Big]^\frac rp\Big)^{1/r}.
\Eeas
When $d(\frac1p-1)<s<\frac1p$ this implies
\[
\Sigma_r\lesssim \Big(\sum_{k>N}\big[\sum_{j>N}
2^{k(s-\frac 1p)p}  2^{j(\frac dp-d-s)p}  2^{N(d-\frac {d-1}{p})p}\big]^\frac rp\Big)^{1/r}\,\sup_{\ell> N} 2^{\ell s}\|\La_\ell f\|_p,
\]
which gives the asserted expression because the series above 
is bounded (uniformly in $N$).
At the endpoint $s=d(\frac1p-1)<\frac1p$ we have 
\[
\Sigma_r\lesssim \Big(\sum_{k>N}2^{(k-N)(s-\frac1p)r}\Big)^{1/r}\,\Big(\sum_{j> N} 2^{j sp}\|\La_j f\|_p^p\Big)^\frac1p,
\]
which also leads to the asserted expression in \eqref{26a}. Finally, if $s=p=1$, using that $U_{p,s}(j,k,N)=2^{N-j}$ we obtain
\[
\Sigma_\infty\lesssim \sum_{j>N}2^{N-j}\,2^j\|\La_j f\|_1\leq \sup_{\ell> N} 2^{\ell }\|\La_\ell f\|_1.\qedhere
\]\end{proof}

\begin{proof}[Proof of Proposition \ref{jleNk>N}]
The left  hand side of \eqref{jleNk>Nineq} 
is controlled by
\begin{align*} 
&\Big(\sum_{k>N}\Big[\sum_{j\le N}U_{p,s}(j,k,N)^p\;2^{jsp}\|\La_j f\|_p^p\Big]^{r/p}\Big)^{1/r}
\\
&\lc\Big(\sum_{k\ge N} 2^{k(s-\frac 1p)r}2^{N(\frac 1p-1)r}\Big)^{\frac 1r}
\Big(\sum_{j\le N} 2^{j(1-s)p} [2^{js}\|\La_j f\|_p]^p
\Big)^{\frac 1p}.
\end{align*}
If $s<1/p$ the first sum can be evaluated as 
$C_2(p,s,r) 2^{N(s-1)}$ and the above expression is dominated by a constant times
\[\Big(\sum_{j\le N} 2^{(j-N)(1-s)p} [2^{js}\|\La_j f\|_p]^p\Big)^{\frac 1p}.
\] \eqref{jleNk>Nineq} follows immediately. The proof of 
\eqref{jleNk>Nineqpartb} 
is similar.
\end{proof}
  
 \begin{proof}[Proof of Proposition \ref{jleNkleN}]
The left  hand side of \eqref{jleNkleNineq} 
is controlled by
\begin{align*} 
&\Big(\sum_{k\le N}\Big[\sum_{j\le N}U_{p,s}(j,k,N)^p\;2^{jsp}\|\La_j f\|_p^p\Big]^{r/p}\Big)^{1/r}
\\
&\lc\Big(\sum_{k\le N} 2^{k(s+d+1-\frac dp)r}2^{N(\frac dp-d-2)r}\Big)^{\frac 1r}
\Big(\sum_{j\le N} 2^{j(1-s)p} [2^{js}\|\La_j f\|_p]^p
\Big)^{\frac 1p}.
\end{align*}
If $s>\frac dp-d-1$
 the first factor can be evaluated to be 
$C_3(p,s,r)2^{N(s-1)}$
and the above expression is again dominated by a constant times
\[\Big(\sum_{j\le N} 2^{(j-N)(1-s)p} [2^{js}\|\La_j f\|_p]^p\Big)^{\frac 1p}.
\] 
Note that for the $s$-range in 
\eqref{jleNkleNineq} to be nontrivial we want 
 $\frac dp-d-1<1$, i.e. $p>\frac{d}{d+2}$. Now  \eqref{jleNkleNineq} follows easily.
\end{proof}

 \begin{proof}[Proof of Proposition \ref{j>NkleN}]
The left  hand side of \eqref{j>NkleNineq} 
is controlled by
\begin{align*} 
&\Big(\sum_{k\le N}\Big[\sum_{j> N}U_{p,s}(j,k,N)^p\;2^{jsp}\|\La_j f\|_p^p\Big]^{r/p}\Big)^{1/r}
\\
&\lc\Big(\sum_{k\le N} 2^{k(s+d+1-\frac dp)r}2^{-Nr}\Big)^{\frac 1r}
\Big(\sum_{j> N} 2^{j(\frac dp -d-s)p} 
[2^{js}\|\La_j f\|_p]^p
\Big)^{\frac 1p}.
\end{align*}
In the range $s\ge \frac dp-d$ under consideration
 the first factor can be evaluated to be $C_4(p,s,r)2^{N(s+d-\frac dp)}$
and the above expression is dominated by a constant times
\[\Big(\sum_{j> N}  2^{-(j-N)(s-\frac dp +d)p}  [2^{js}\|\La_j f\|_p]^p\Big)^{\frac 1p}.
\] 
This yields \eqref{j>NkleNineq}.
\end{proof}

\begin{remark}
\label{R_ENp<1}
 The proofs of Propositions \ref{jk>N} and \ref{j>NkleN} show that each operator $\SE_N$ admits an extension to $B^s_{p,q}(\SR^d)$ in
the ranges of indices (iv), (v) and (vi) of Theorem \ref{basic-sequence-besov}, namely
\Be\label{extEN}
\SE_N(f):=\sum_{j=0}^\infty\SE_N[L_j\La_j f], \quad \mbox{in }\;B^s_{p,q}.
\Ee
Indeed, for all $r>0$ and for $J_2>J_1>N$ one has, in cases (iv) and (v),
\[
\big\|\SE_N(\sum_{j=J_1}^{J_2}L_j\La_j f)\big\|_{B^s_{p,r}}\lesssim_N 2^{-J_1\e}\|f\|_{B^s_{p,\infty}},
\] 
with $\e=s-d(1/p-1)>0$,  and in case (vi)
\[
\big\|\SE_N(\sum_{j=J_1}^{J_2}L_j\La_j f)\big\|_{B^s_{p,r}}\lesssim_N \,\Big(\sum_{j=J_1}^{J_2}2^{jsp}\|\La_j f\big\|_p^p\Big)^\frac1p.
\]
\end{remark}

  \begin{proof}[Proof of Theorem \ref{basic-sequence-besov}: Sufficiency for $\frac{d}{d+1}\le p\le 1$] 
	In view of \eqref{PiN}, \eqref{SENPiN} and trivial embeddings of Besov spaces, the uniform boundedness of $\SE_N$ follows immediately from the above four propositions.
  \end{proof}

\section{Upper bounds for  $1\leq p\leq \infty$}\label{upbdspge1sect}
When $p\geq1$  the constants in \eqref{Ups_def}
take the form   
   \Be \label{BjkN-pge1}
   U_{p,s}(j,k,N)
= \begin{cases} 
 2^{k(s-\frac 1p)} 
 2^{j(\frac 1p-1-s)} 2^N
  &\text{ if } j,k>N,
 \\
 2^{k(s-\frac 1p)}  2^{j(1-s)} 2^{N(\frac 1p-1)}
 &\text{ if } j\le N<k 
 \\
 2^{k(1+s)} 2^{j(1-s)} 2^{-2N}&\text{ if } 0\le j,k\le N,
  
 
 \\
 2^{k(1+s)} 2^{j(\frac 1p-1-s)} 2^{-\frac Np}&\text{ if } k\le N<j.
 \end{cases} 
 \Ee
Again we state four propositions corresponding to the four cases of \eqref{BjkN-pge1}.

\begin{proposition}\label{jk>Np>1}
Suppose  $1\le p\le \infty$. Then
\begin{subequations}
\Be\label{jk>Np>1ineqa}
\sup_{k>N} 2^{ks}\Big\|\sum_{j>N} L_k\bbE_N L_j\La_jf\Big\|_p \lc \sup_{j>N} 2^{js} \|\Lambda_j f\|_p,
\;\;\text{ if } \tfrac 1p-1<s\le \tfrac 1p.
\Ee
Moreover, for all $r>0$
\begin{multline}\label{jk>Np>1ineqb}
\Big(\sum_{k>N} 2^{ks r} \Big\|\sum_{j>N}L_k\bbE_N L_j\Lambda_j f\Big\|_p^r\Big)^{1/r}\\
\lc \begin{cases} \sup_{j>N} 2^{js} \|\La_j f\|_p &\text{ if } \frac 1p-1<s<\frac 1p,
\\
\sum_{j>N} 2^{js} \|\La_jf\|_p   &\text{ if } s=\frac 1p-1.
\end{cases}
\end{multline}
\end{subequations}
The same inequalities hold with $\bbE_N$ replaced by 
 $T_N[\cdot,\fa]$ if $\|\fa\|_\infty\le 1$.
 \end{proposition}

\begin{proposition}\label{jleNk>Np>1}
Suppose  $1\le p\le \infty$. Then
 for all $r>0$ 
\begin{subequations}
\Be \label{jleNk>Np>1ineqa}
\Big(\sum_{k>N} 2^{ks r} \Big\|\sum_{j\le N}L_k\bbE_N^\perp L_j\Lambda_j f\Big\|_p^r\Big)^{1/r}
\lc \sup_{j\le N}  2^{js} \|\La_j f\|_p ,\quad\text{if }\; s<\tfrac1p.\Ee
Moreover, if $s=\frac1p<1$ then
\Be\label{jleNk>Np>1ineqb}
\sup_{k>N} 2^{ks}\Big\|\sum_{j\le N} L_k\bbE_N^\perp L_j\La_jf\Big\|_p \lc \sup_{j\le N} 2^{js} \|\Lambda_j f\|_p,
\Ee
and if $s=p=1$ then
\Be\label{jleNk>Np>1ineqc}
\sup_{k>N} 2^{k}\Big\|\sum_{j\le N} L_k\bbE_N^\perp L_j\La_jf\Big\|_1 \lc \sum_{j\le N} 2^{j} \|\Lambda_j f\|_1.
\Ee
\end{subequations}
The same inequalities hold with $\bbE^\perp_N$ replaced by 
 $T_N[\cdot,\fa]$ if $\|\fa\|_\infty\le 1$.
 \end{proposition}

 \begin{proposition}\label{jleNkleNp>1}
Let  $1\le p\le\infty$ and $r>0$. Then
\begin{subequations}
\Be \label{jleNkleNp>1ineqa}
\Big(\sum_{k\le N} 2^{ks r} \Big\|\sum_{j\le N} L_k\bbE_N^\perp L_j\Lambda_j f\Big\|_p^r\Big)^{1/r}\\
\lc  \sup_{j\le N} 2^{js} \|\La_j f\|_p \, \text{ if } -1<s<1.
\Ee
Moreover, for the case $s=-1$ we have  
\Be \label{jleNkleNp>1ineqb}\sup_{k\le N} 2^{-k} \Big\|\sum_{j\le N} L_k\bbE_N^\perp L_j\Lambda_j f\Big\|_p
\lc  \sup_{j\le N} 2^{-j} \|\La_j f\|_p,
\Ee
and for the case $s=1$ we have
\Be \label{jleNkleNp>1ineqc}
\Big(\sum_{k\le N} 2^{k r} \Big\|\sum_{j\le N} L_k\bbE_N^\perp L_j\Lambda_j f\Big\|_p^r\Big)^{1/r}
\lc \sum_{j=0}^N 2^{j}\|\La_jf\|_p.
\Ee\end{subequations}
The same inequalities  hold  with $\bbE_N^\perp$ replaced by 
 $T_N[\cdot,\fa]$ if $\|\fa\|_\infty\le 1$.
 \end{proposition}

 \begin{proposition}\label{j>NkleNp>1}
Let  $1\le p\le\infty$. Then 
\begin{subequations}
for all $r>0$,
\Be\label{j>NkleNp>1ineqa}
\Big(\sum_{k\le N} 2^{ks r} \Big\|\sum_{j>N} L_k\bbE_N L_j\Lambda_j f\Big\|_p^r\Big)^{1/r}\lc \sup_{j>N} 2^{js}\|\La_j f\|_p\quad \text{ if }\,s>\tfrac 1p-1
\Ee
Moreover, for the case $s=\frac 1p-1$ and $1\le p<\infty$,
\Be\label{j>NkleNp>1ineqb}
\Big(\sum_{k\le N} 2^{k(\frac 1p-1) r} \Big\|\sum_{j>N} L_k\bbE_N L_j\Lambda_j f\Big\|_p^r\Big)^{1/r}\lc \sum_{j=N+1}^\infty  2^{j(\frac 1p-1)}\|\La_j f\|_p.
\Ee
Finally, for the case $s=-1$ and $p=\infty$
\Be \label{j>NkleNp>1ineqc}
\sup_{k\le N} 2^{-k} \Big\|\sum_{j>N} L_k\bbE_N L_j\Lambda_j f\Big\|_\infty \lc \sum_{j=N+1}^\infty 2^{-j}\|\La_j f\|_\infty.
\Ee
\end{subequations}
The same inequalities  hold  when $\bbE_N$ is replaced by 
 $T_N[\cdot,\fa]$, with $\|\fa\|_\infty\le 1$.
 \end{proposition}

\subsection*{\it Proofs}
The proofs of the four propositions involve Corollary \ref{Upscor} and an application of the triangle inequality for $L^p$ when $p\ge 1$.

\begin{proof}[Proof of Proposition \ref{jk>Np>1}]
Assume $s<1/p$. By the triangle inequality and Corollary \ref{Upscor}
the left hand side of 
\eqref{jk>Np>1ineqb}
 is estimated by
\begin{align*}
&\Big(\sum_{k>N} 2^{ksr} \Big[\sum_{j>N}\|L_k \bbE_N L_j \La_j f\|_p \Big]^r\Big)^{\frac 1r}
\\
&\lc \Big(\sum_{k>N} \Big[ \sum_{j>N} U_{p,s}(j,k,N) 2^{js} \|\La_j f\|_p\Big]^r\Big)^{\frac 1r}
\\
&\lc \Big(\sum_{k>N} 2^{k(s-\frac 1p) r}\Big)^{\frac 1r}
\sum_{j>N} 2^{j(\frac 1p-1-s)} 2^{N} 2^{js}\|\La_j f\|_p.
\end{align*}
When $s<1/p$ the first factor is $c(p,s,r) 
2^{N(s-1/p)}$ and we see that the entire expression is dominated by a constant times 
\[\sum_{j>N} 2^{(N-j) (s+1-\frac 1p)} 2^{js}
\|\La_j f\|_p
\]
which proves \eqref{jk>Np>1ineqb} 
and of course also
\eqref{jk>Np>1ineqa} 
 when $s<1/p$. Replacing the $\ell^r$ norm by a supremum in the above proof we see that
 \eqref{jk>Np>1ineqa}   is valid even for $s=1/p$. 
\end{proof} 

\begin{proof}[Proof of Proposition \ref{jleNk>Np>1}]
Let $s<1/p$. The left hand side of 
\eqref{jleNk>Np>1ineqa}
 is estimated by a constant times
\begin{align*}
&\Big(\sum_{k>N} \Big[ \sum_{j\le N} U_{p,s}(j,k,N) 2^{js} \|\La_j f\|_p\Big]^r\Big)^{\frac 1r}
\\
&\lc \Big(\sum_{k>N} 2^{k(s-\frac 1p) r}\Big)^{\frac 1r}
\sum_{j\le N} 2^{j(1-s)} 2^{N(\frac 1p-1)} 2^{js}\|\La_j f\|_p
\\
&\lc\sum_{j\le N} 2^{(j-N)(1-s)}  2^{js}
\|\La_j f\|_p.
\end{align*}
This  easily yields  \eqref{jleNk>Np>1ineqa}.
The proofs of \eqref{jleNk>Np>1ineqb}, \eqref{jleNk>Np>1ineqc} are similar. 
\end{proof} 
\begin{proof}[Proof of Proposition \ref{jleNkleNp>1}]
Assume $s>-1$. The left hand side of 
\eqref{jleNkleNp>1ineqa}
 is estimated by a constant times
\begin{align*}
&\Big(\sum_{k\le N} \Big[ \sum_{j\le N} U_{p,s}(j,k,N) 2^{js} \|\La_j f\|_p\Big]^r\Big)^{\frac 1r}
\\
&\lc \Big(\sum_{k\le N} 2^{k(1+s)r}\Big)^{\frac 1r}
\sum_{j\le N} 2^{j(1-s)} 2^{-2N} 2^{js}\|\La_j f\|_p\,,
\end{align*}
and since the first factor is $\tilde c(p,q,r) 2^{N(1+s)}$ we estimate the expression by
a constant times \[\sum_{j\le N} 2^{(j-N)(1-s)}  2^{js}
\|\La_j f\|_p.\]
This  easily yields  \eqref{jleNkleNp>1ineqa} and also 
\eqref{jleNkleNp>1ineqc}. 
The proof of \eqref{jleNkleNp>1ineqb} which has  a supremum in $k$ for the case $s=-1$ is similar.
\end{proof}
\begin{proof}[Proof of Proposition \ref{j>NkleNp>1}]
Let  $s>-1$. The left hand side of 
\eqref{j>NkleNp>1ineqa}
 is estimated by a constant times
\begin{align*}
&\Big(\sum_{k\le N} \Big[ \sum_{j> N} U_{p,s}(j,k,N) 2^{js} \|\La_j f\|_p\Big]^r\Big)^{\frac 1r}
\\
&\lc \Big(\sum_{k\le N} 2^{k(1+s)r}\Big)^{\frac 1r}
\sum_{j> N} 2^{j(\frac 1p-1-s)} 2^{-\frac Np} 2^{js}\|\La_j f\|_p
\\
&\lc 
\sum_{j> N} 2^{(j-N)(\frac 1p-1-s)} 2^{js}\|\La_j f\|_p
\end{align*}
which yields 
\eqref{j>NkleNp>1ineqa} and also 
\eqref{j>NkleNp>1ineqb}. 
The proof of \eqref{j>NkleNp>1ineqc} for the case  $s=-1$ is similar.
\end{proof}

\begin{remark}
\label{R_ENp>1}
Similar reasonings as in Remark \ref{R_ENp<1} justify the meaning of the extension formula for $\SE_N$ in \eqref{extEN}, for the ranges of indices in (i), (ii), (iii) in Theorem \ref{basic-sequence-besov},
and the cases (i), (ii) in Theorem \ref{p=inftythm}. In the special case $s=1/p$, for $1<p\leq\infty$, one has
\[
\big\|\SE_N({\textstyle\sum_{j=J_1}^{J_2}L_j\La_j f})\big\|_{B^{1/p}_{p,\infty}}\lesssim_N 2^{-J_1}\|f\|_{B^s_{p,\infty}},
\]
so the series $\sum_{j=0}^\infty\SE_N(L_j\La_jf)$ always converges in $B^{1/p}_{p,\infty}$, even though the series $\sum_{j=0}^\infty L_j\La_j f$ only does if $f\in b^{1/p}_{p,\infty}$. 
\end{remark}

\begin{proof}[Proof of Theorems \ref{basic-sequence-besov} and \ref{p=inftythm}: Sufficiency for $1\leq p\leq\infty$] 
As before, one uses the previous four propositions combined with \eqref{PiN}, \eqref{SENPiN} and trivial embeddings of Besov spaces.
  \end{proof}
  
  \section{ Necessary conditions for  boundedness when $s=1/p$} \label{necconds=1/p}
  It is well known that the characteristic function of a cube (and also the Haar functions) do not belong to $B^{1/p}_{p,q}$ for any $q<\infty$; see \cite[2.6.3 (18)]{triebel_teubner}. In this section we elaborate a bit more on this result. 
	
Recall that 
$b^{s}_{p,\infty}$ denotes the closure of the Schwartz space in the $B^s_{p,\infty}$ norm.
Note also that $B^s_{p,q}\subset b^s_{p,\infty}$ for all $q<\infty$; see \eqref{bsinfty} above.
Finally, $\Span \sH_d$ denotes the vector space of all 
finite linear combinations of Haar functions.

\begin{proposition}\label{intersectzero}
Let $0<p\leq \infty$. Then
\[
b^{1/p}_{p,\infty}\cap\Span \sH_d
=\{0\}.
\]
\end{proposition}

Before  proving  this proposition we define, given $M\in \bbN$,  certain test functions $\Psi$ with vanishing moments of order up to $2M$ (which, along with their dilates $\Psi_N= 2^{Nd}\Psi(2^N\cdot)$,  will be also be used in subsequent sections).

\subsection{\it Tensorized test functions}\label{testfctsubsection}


Given $M\in\SN$, consider a non-negative even function $\phi_0\in C^\infty_c(-\tfrac18, \tfrac 18)$ such that  
$\phi_0^{(2M)}(t)>0$ for all $t$ in some interval $[-2\e,2\eps]$ (with $\eps<1/16$).
Since $\widehat\phi_0(0)=\int\phi_0>0$, dilating if necessary we may also assume that $\widehat \phi_0\neq 0$ on $(-1,1)$.
Let $\varphi_0 \in C^\infty_c((-\tfrac 18, \tfrac 18)^{d-1})$ be 
such that $\widehat \varphi_0\neq 0$ on $(-1,1)^{d-1}$ and $\widehat \varphi_0(0)=1$.
For $M\ge 1$, let \[ \theta(t)= (\tfrac{d}{dt})^{2M} \phi_0(t), \qquad
\vartheta (x_2,\dots, x_d)= \big(
\tfrac{\partial^2}{\partial x_2^2} +\dots+
\tfrac{\partial^2}{\partial x_d^2} \big)^M \varphi_0(x').
\] 
In one dimension the function $\vartheta $ is obsolete and we just define $\Psi=\theta$. If $d\ge 2$ we define  \Be\Psi(x)= 
\Delta^M[ \phi_0\otimes \varphi_0] (x)= \theta(x_1) \varphi_0(x') + \phi_0(x_1) \vartheta(x').
\label{Psitensor}
\Ee
Clearly, $$\int_{\SR^d} \Psi (y) y_1^{m_1}\cdots y_d^{m_d} 
dy=0, \quad\mbox{when }m_1+\ldots+m_d< 2M.$$
If we choose $2M\gg |s|+  d/p-d$ then for all $f\in 
B^{s}_{p,q}(\SR^d)$, 
 \Be\label{besovlowerbdPsi}
 \|f\|_{B^{s}_{p,q}}
 \,\gtrsim\,
 \Big\| \big\{ 2^{ks}\Psi_k*f\big\}_{k\in \bbN}\Big\|_{\ell^q(L^p)}.
 \Ee

    


\subsection{
\it Proof of Proposition \ref{intersectzero}}
We argue by contradiction and assume that there is a nontrivial  $f\in b^{1/p}_{p,\infty}\cap\Span \sH_d$. Then for some $N\in \bbN$ we can write $f$  as
\[
f=\sum_{\nu\in\Gamma} a_\nu\bbone_{I_{N,\nu}},
\]
where 
$I_{N,\nu}=\prod_{i=1}^d [\nu_i2^{-N}, (\nu_i+1)2^{-N})$, 
$\Gamma$ is a finite  nonempty subset of $\bbZ^d$, and $a_\nu\in \bbC$ with $a_\nu\neq 0$ for $\nu\in \Gamma$.

Consider the usual partial order in $\SZ^d$, that is for $\mu,\nu\in\SZ^d$ we say that
\[
\mu\leq \nu\quad \mbox{if}\quad \mu_i\leq \nu_i\quad\forall\,i=1,\ldots,d.
\]
Pick a minimal element $\mu\in\Ga$, meaning that that if $\nu\in\Ga$ and $\nu\leq\mu$ then necessarily $\nu=\mu$.
Now consider the function
\[
g(x)=f(2^{-N}(x+\mu))/a_\mu,\]
which also belongs to $b^{1/p}_{p,\infty}\cap\Span \sH_d$. Note that $g$ is now a linear combination of disjoint unit cubes 
and satisfies 
\Be
g(x)=\left\{\Ba{lll} 1 & & {\rm if}\quad  x\in [0,1)^d\\
0 && {\rm if}\quad x\in(-1,1)^d\setminus[0,1)^d.\Ea\right.
\label{pR1}
\Ee
This last property is a consequence of the minimality of $\mu$.

Consider now the function $\Psi\in C^\infty_c(\SR^d)$ as in \eqref{Psitensor}, with 
the pairs of functions  $\phi_0,\theta$ and $\varphi_0,\vartheta$  as in the paragraph preceding that definition.
So, in particular,
\[
  \int_{\SR^{d-1}}\varphi_0(x')\,dx'=1\mand  \int_{\SR^{d-1}}\vartheta(x')\,dx'=0.
\]
Observe further that for $t\in [-2\eps, -\eps]$
\begin{align*}\int_0^{\infty}\theta(t-s)ds=
\int_{-\infty}^t \theta(u) du= - \int_{t}^0 \theta(u) du
\le - \int_{-\eps}^{0} \theta(s) ds
\end{align*}
since  $\int_{-\infty}^0\theta(s) ds=\phi_0^{(2M-1)}(0)=0$ (because $\phi_0$ is even) and $\theta(u)
> 0$ for  $u\in (-2\eps,0)$. Thus,   if we set  
\[ c:= \int_{-\eps}^{0} \theta(s) ds 
>0\]
we obtain
\Be \label{pR2}
\int_0^{\infty}\theta(t-s)ds
 \leq -  c,\quad \forall\;t\in[-2\eps,-\eps].
 \Ee

Next consider $\Psi_k(x)=2^{kd}\Psi(2^kx)$, $k\geq1$, 
and note that $\Psi$ has enough vanishing moments so that
\[
\|h\|_{B^{1/p}_{p,\infty}}\gtrsim \sup_{k\geq1}2^{k/p}\|\Psi_k*h\|_p,\quad h\in B^{1/p}_{p,\infty}.
\]
Moreover, since $g\in b^{1/p}_{p,\infty}$ we also have
\Be
2^{k/p}\|\Psi_k*g\|_p\to 0, \quad {\rm as}\quad k\to\infty.
\label{pR3}
\Ee
We show that this leads to a contradiction.
Indeed, if $x_1\in[-\eps 2^{1-k},-\eps 2^{-k}]$  and $x'\in [1/4,3/4]^{d-1}$ then, using 
that $\supp\Psi_k(x-\cdot)\subset(-1,1)\times(1/8,7/8)^{d-1}$, we may apply 
\eqref{pR1} and \eqref{pR2} to obtain
\Beas
g*\Psi_k(x) & = & \int_{[0,1)^d}\Psi_k(x-y)\,dy\\
& = & \int_0^1\theta_k(x_1-y_1)\,dy_1\int_{(0,1)^{d-1}}\varphi_{0,k}(x'-y')\,dy'+\\
& & \quad\quad +\int_0^1\phi_{0,k}(x_1-y_1)\,dy_1\int_{(0,1)^{d-1}}\vartheta_{k}(x'-y')\,dy'\\
& = & 2^k\int_0^1\theta(2^k(x_1-y_1))\,dy_1= \int_0^{2^k}\theta(2^kx_1-u)\,du
\\
& = &\int_0^{\infty}\theta(2^kx_1-u)\,du\leq -c/2.
\Eeas
Thus we must have
\[
\|g*\Psi_k\|_p\geq (c/2)\,(\e2^{-k})^{1/p}(1/2)^{\frac{d-1}p},
\]
which contradicts \eqref{pR3}. \qed

\begin{Remark}
In the  recent work \cite{ysy} by Yuan, Sickel and  Yang, the 
authors study regularity properties of the Haar system in other Besov-type 
spaces $B^{s,\tau}_{p,q}(\mathbb{R}^d)$ which  serves as a first step to investigate its basis properties  in these spaces.
\end{Remark}

   \section{ Necessary conditions for boundedness when $s=1$ 
	}\label{B1pq-bd-sect}
We now consider the necessity of the condition $q\le p$ in part (iv) of  Theorem \ref{basic-sequence-besov}. This restriction was also noticed in \cite{oswald}. 
For $q>p$ we show that the  operators $\bbE_N$ are bounded, but not uniformly bounded and determine the precise behavior of the operator norms as $N\to\infty$.
The lower bounds will be obtained by testing with suitable functions with compact support; we refer to \eqref{localnorms} for the notation in the next theorem.

  \begin{theorem} 
  \label{equiv-s=1thm}
  Suppose that  either 
  
  (i) $\frac{d}{d+1} <p< 1\;$ and $\;p\le q\leq \infty$, \;\; or 
  
  (ii) $p=1$ and $q=\infty$.
  
  \noindent
  Then   for large $N$
$$\|\bbE_N\|_{B^1_{p,q}\to B^1_{p,q}} \approx  N^{\frac 1p-\frac 1q}.$$
Moreover, for  cubes $Q$ of side length  $\ge  1/2$,
$$\rhoop\big(\bbE_N, B_{p,q}^1, Q\big) 
\approx  N^{\frac 1p-\frac 1q}.$$


\end{theorem}
\subsection*{\it Proof of the upper bounds in Theorem \ref{equiv-s=1thm} }
Letting $s=1$ in Propositions \ref{jk>N} and \ref{j>NkleN} 
(and noticing that $d(\frac 1p-1)<1<\frac1p$ when $\frac{d}{d+1}<p<1$), we see that
\[
\Big\|\big\{2^{k}\sum_{j>N} L_k\SE_NL_j\La_jf\big\}_{k=0}^\infty\Big\|_{\ell^q(L^p)} \lesssim \|f\|_{B^1_{p,\infty}}\leq \|f\|_{B^1_{p,q}}.
\]
On the other hand, letting $s=1$ in
 Propositions \ref{jleNk>N} and \ref{jleNkleN}, 
and using H\"older's inequality we obtain
\[
\Big\|\big\{2^{k}\sum_{j\leq N} L_k\SE^\perp_NL_j\La_jf\big\}_{k=0}^\infty\Big\|_{\ell^q(L^p)}   \lesssim  \Big(\sum_{j=0}^N2^{jp} \|\La_jf\|_p^p\Big)^{1/p}  \lc   N^{1/p-1/q} \,\|f\|_{B^1_{p,q}}.
\]
Combining this with \eqref{PiN} and \eqref{SENPiN} we obtain $\|\bbE_N\|_{B^1_{p,q}\to B^1_{p,q}} \lesssim  N^{\frac 1p-\frac 1q}$. The above arguments also apply if $s=p=1$, provided we let $q=\infty$.
\qed

\subsection*{\it Proof of the lower bounds in Theorem \ref{equiv-s=1thm}}
We shall actually prove a stronger result which  gives a lower bound even for a  $B^1_{p,q}\to B^1_{p,\infty} $ estimate and  for functions supported in the  open  unit cube $Q_0=(0,1)^d$. 
\begin{theorem} \label{lowerbounds-s=1thm}
If $0<p\le 1$ and $p\le q\le \infty$,  
then there is $c_{p,q}>0$ such that, for each $N\geq1$,
\Be \label{lowerbounds-s=1}
\sup\Big \{\| \bbE_N f\|_{B^1_{p,\infty}}:\,\, \|f\|_{B^1_{p,q}}\le 1, \,\, \supp(f)\subset Q_0\Big\} \,\ge\, c_{p,q}\, N^{\frac1p-\frac1q}.
\Ee
\end{theorem}

Fix $u\in C^\infty_c(\bbR)$  supported in $(1/8, 7/8)$  with  $u(t)=1$ for $t\in [1/4,3/4]$, and  $\chi\in C^\infty_c (\bbR^{d-1})$ supported in $(1/8,7/8)^{d-1}$ with $\chi(x')=1$ for $x'\in 
(1/4,3/4)^{d-1}$; here $x'=(x_2,\dots, x_d)$.
Define, for large $N$, 
functions of one variable
\Be g_{N,j} (t)= e^{2\pi i 2^j t}  u(Nt-2j),\quad j\in\SN, \label{defgNj}\Ee
and let
\Be \label{fNSchwartz}
f_N(x)=
\chi(x_2,\dots, x_d) \sum_{N/8<j<N/4} 2^{-j} g_{N,j} (x_1) .
\Ee 

\begin{lemma}
For $p\le q\le \infty$ we have
\Be
\|f_N\|_{B^1_{p,q}} \lc N^{1/q-1/p}.
\label{L5.3}
\Ee
\end{lemma}
\begin{proof}
We estimate $L_k f_N=\beta_k*f_N$. 
If $2^k\leq N$, since $\beta_k*f_{N}$ is compactly supported and  $\|\beta_k*f_{N}\|_\infty\lesssim \|f_N\|_\infty\lesssim 2^{-N/8}$, then
\[
\Big(\sum_{k=0}^{\log_2N}2^{kq}\|\beta_k*f_N\|_p^q\Big)^{1/q}\lesssim N\,2^{-N/8}\ll N^{1/q-1/p}.\]

Assume now that $2^k> N$.
First notice that the sets 
\Be
\label{suppgNj}
\supp \beta_k*g_{N,j} \subset \big\{\tfrac{2j}N+(-\tfrac2{N},\tfrac2{N})\big\}\times (0, 1)^{d-1},\quad \tfrac N8< j< \tfrac N4,\Ee are pairwise disjoint,
and thus 
\Be
\|\beta_k*f_N\|_p=\Big(\sum_{\frac N8< j< \frac N4}2^{-jp}\|\beta_k*(g_{N,j}  \otimes \chi)
\|^p_p\Big)^{1/p}.
\Ee
Next, we distinguish the cases $j\ge k$ and $j\le k$.
When $j\ge k$, if we integrate by parts $M$-times with respect to $y_1$ in the convolution we obtain 
\begin{align*}
\beta_k*(g_{N,j}  \otimes \chi)(x)
=
\int \frac{\partial^M}{\partial y_1} \big[
\beta_k(x-y) u(Ny_1-2j)\chi(y') \big] \frac{e^{2\pi i 2^j y_1}}
{(-2\pi i 2^j)^{M}}
 dy_1\,dy'
\end{align*} and thus, using that $N<2^k$, 
\[\|\beta_k*(g_{N,j}  \otimes \chi)\|_p \lc 2^{-(j-k)M} N^{-1/p}, \quad j\geq k.\]
For $N/8< j\le k$  we use the cancellation of the $\beta_k$ (with $M$ vanishing moments) to obtain
\[
\|\beta_k*(g_{N,j}  \otimes \chi)\|_p \lc 
2^{-kM}\,\|\partial^Mg_{N,j}\|_\infty \,N^{-\frac1p}\lesssim
2^{-(k-j)M} N^{-1/p}, \quad j\le k.\]

Thus 
\begin{align*}
&\Big(\sum_{2^k>N}2^{kq}\|\beta_k*f_N\|_p^q\Big)^{\frac1q}
= 
\Big(\sum_{2^k>N}2^{kq}\Big[\sum_{\frac N8<j<\frac N4}2^{-jp}
\|\beta_k*(g_{N,j}\otimes\chi)\|^p_p\Big]^\frac qp\Big)^{\frac1q}
\\
&\quad \lc 
\Big(\sum_{2^k>N}\Big[\sum_{\frac N8<j<\frac N4}2^{(k-j)p}  2^{-|j-k|Mp} 
\Big]^\frac qp\Big)^{\frac1q}\,N^{-\frac1p} \,\lc\, N^{\frac1q-\frac1p}. \qedhere
\end{align*}
\end{proof}

We now take $\Psi_N= 2^{Nd}\Psi(2^N\cdot)$ with $\Psi$ as in \S\ref{testfctsubsection}, and we shall prove that 
\Be 
\|\bbE_N f_N\|_{B^1_{p,\infty}} \gc 2^N\|\Psi_N* \bbE_N f_N\|_p\gtrsim 1.
\label{ENfNgtr1}
\Ee

Define  $\Theta:\bbR\to \bbR$ by 
\Be\Theta(t)= \int_{-\infty}^t \theta(s) ds \label{Theta}\Ee  with $\theta=\phi_0^{(2M)}$ as in 
\S\ref{testfctsubsection} 
and observe that $\Theta$ 
 is odd, supported in $(-\tfrac 18, \tfrac 18)$ and   $\int_{-\infty}^\infty \Theta(t) dt=0$.
In particular, \Be \label{Thetaneq0}\int_0^{1/8}|\Theta(t)|^p dt \neq 0.\Ee

Let $\bbE_N^{(1)} $ and $\bbE_N ^{(d-1)}$ be the dyadic averaging operators on $\bbR$ and $\bbR^{d-1}$, respectively. If we denote  $\theta_N= 2^N\theta(2^N\cdot)$, $N\geq1$, then we claim that
\Be\label{1Dredux}
\Psi_N* (\bbE_N [g_{N,j}\otimes\chi])(x_1,x') =  \theta_N * (\bbE_N^{(1)} g_{N,j}) (x_1),
\text{ for $x'\in(\tfrac13,\tfrac23)^{d-1}$.}
\Ee 
Indeed, for $x'\in(\tfrac13,\tfrac23)^{d-1}$ it is easily seen that
\begin{align*}
&2^{N(d-1)} \varphi_0(2^N\cdot)* \bbE_N^{(d-1)} \chi(x')= \int \varphi_0 (y')dy'=1,
\\
&2^{N(d-1)} \vartheta(2^N\cdot)* \bbE_N^{(d-1)} \chi(x')= \int \vartheta (y')dy'=0.
\end{align*}

The proof of the lower bound in \eqref{ENfNgtr1} will rely on the following lemma.

\begin{lemma} \label{auxiliarylem} Let $\nu \in \bbZ$  and $\widetilde I_{N,\nu}=  \big[\tfrac{\nu}{2^N},\tfrac{\nu+1/8}{2^N}\big]$. 
Then,  for every  $t\in \widetilde I_{N,\nu}$ and $N/8<j<N/4$ we have\Be
\theta_N* (\bbE_N^{(1)} g_{N,j})(t)=2^{-N}\, g_{N,j}'\big(\tfrac{\nu}{2^N}\big)
 \,\Theta(2^Nt-\nu)+ O( 2^{-2(N-j)}). \label{gNjp}
\Ee Moreover if $ \tfrac{\nu}{2^N}\in [\tfrac{2j}{N}+\tfrac 1{4N}, \tfrac{2j}{N}+\tfrac 3{4N}]$ then 
$g_{N,j}'(2^{-N}\nu)=2\pi i 2^j   e^{2\pi i 2^{j-N}\nu}$.
\end{lemma}
\begin{proof}
The last assertion is immediate by the definition of $g_{N,j}$ in \eqref{defgNj}. So we focus in proving \eqref{gNjp}.
We split \[
\theta_N*\bbE_N^{(1)} g_{N,j}  =  \theta_N*(\bbE_N^{(1)} -I) g_{N,j}  + \theta_N*g_{N,j}
\]
and observe that from the cancellation properties of $\theta_N$ we have
\[ \|\theta_N* g_{N,j}\|_\infty  \lc 2^{-2N}\|g''_{N,j}\|_\infty \lc 2^{-2N} (2^{2j}+N^2)\]
which for $N/8\le j\le N/4$ implies
$\|\theta_N* g_{N,j}\|_\infty  \lc 2^{-2(N-j)} .$
 Let $I_{N,\nu}=[2^{-N}\nu, 2^{-N}(\nu+1))$. 
For $t\in \widetilde I_{N,\nu}$ we have $\supp \theta_N(t-\cdot)\subset I_{N,\nu-1}\cup I_{N,\nu}$, so recalling \eqref{Theta} we obtain
 \begin{align*} 
 &\theta_N*(\bbE_N^{(1)} -I) g_{N,j}(t)\,= \, 
\\& \int 2^{N}\Theta'(2^N(t-s)) \times\Big[
 \bbone_{I_{N,\nu}}(s) \Big( \mint_{I_{N,\nu}} g_{N,j}(w) dw -g_{N,j}(s)\Big)
 \\ &\qquad\qquad\qquad\qquad\qquad
 +
 \bbone_{I_{N,\nu-1}}(s) \Big( \mint_{I_{N,\nu-1}} g_{N,j}(w) dw -g_{N,j}(s)\Big)\Big]\, ds
 \end{align*}
For $s\in I_{N,\nu}$, using Taylor expansions  one sees that
\begin{align*} & \mint_{I_{N,\nu}} g_{N,j}(w) dw -g_{N,j}(s)\\&=
\mint_{I_{N,\nu}}
\!\! g_{N,j}'(s)(w-s) dw + \mint_{I_{N,\nu}} \int_0^1 (1-\sigma) g''_{N,j}(s+\sigma(w-s))d\sigma \, (w-s)^2 dw
 \\&= g_{N,j}'(\tfrac\nu{2^N}) \frac1{2|I_{N,\nu}|}{\Big[(\tfrac{\nu+1}{2^N}-s)^2-(\tfrac\nu{2^N}-s)^2\Big]} + O(2^{2j-2N})
 \\
 &= g_{N,j}'(\tfrac\nu{2^N}) \,\Big[\tfrac{\nu+1/2}{2^N} -s\Big] + O(2^{2j-2N}).
 \end{align*}
Similarly, for $s\in I_{N,\nu-1}$, 
\[
 \mint_{I_{N,\nu-1}} g_{N,j}(w) dw -g_{N,j}(s)=
g_{N,j}'(\tfrac\nu{2^N}) \Big[\tfrac{\nu-1/2}{2^N} -s\Big]  + O(2^{2j-2N}).\]

 Hence 
 for $t\in \widetilde I_{N,\nu}$ we have
 \Be
2^N \theta_N*(\bbE_N^{(1)}-I) g_{N,j}(t)\,= \, A_{1,j}(t) +  A_{2,j}(t)  +O(2^{-2(N-j)})\label{A1j}\Ee
 where
 \[A_{1,j} (t)= g_{N,j}'(\tfrac\nu{2^N}) \int 2^N \Theta'(2^N(t-s)) \frac{1}{2^{N+1}}\big( \bbone_{I_{N,\nu}}(s)- \bbone_{I_{N,\nu-1}}(s) \big) \, ds\]
 and 
 \begin{align*}A_{2,j} (t)&= g_{N,j}'(\tfrac\nu{2^N}) \int  
 2^{N} \Theta'(2^N(t-s)) (\tfrac\nu{2^N}-s) \, ds\,.
 \end{align*}
Integration by parts yields (for $t\in \widetilde I_{N,\nu}$)
 \[A_{2,j} (t)= g_{N,j}'(2^{-N}\nu) \int 2^{N} \Theta(2^N(s-t))
\, ds=0\]
 since $\int \Theta(s) ds=0$. 
 To  compute  $A_{1,j}(t)$ we observe that 
 \begin{align*}
 &\int 2^N \Theta'(2^N(t-s)) \big( \bbone_{I_{N,\nu}}(s)- \bbone_{I_{N,\nu-1}}(s) \big) \, ds\\
 &= \Big[\int_{\frac\nu{2^N}}^{\frac{\nu+1}{2^N}}-\int^{\frac\nu{2^N}}_{\frac{\nu-1}{2^N}}\;\Big] \;
 \frac{d}{ds}\big(- \Theta(2^N(t-s)\big)\, ds
 \\
 &= - \Theta\big(2^N(t-\tfrac{\nu+1}{2^N})\big)
 + 2\Theta\big(2^N(t-\tfrac\nu{2^N})\big)
 -  \Theta\big(2^N(t-\tfrac{\nu-1}{2^N})\big).
\end{align*}
For $t\in \widetilde I_{N,\nu}$ we have $\Theta(2^N(t-2^{-N}(\nu\pm1)))=0$ and thus
\[ A_{1,j}(t) =\, 2^{-N}\,g_j'(\tfrac\nu{2^N})\Theta(2^Nt-\nu),\quad t\in \widetilde I_{N,\nu}.\]
Inserting these expressions into \eqref{A1j} we are led to \eqref{gNjp}.
 \end{proof}

We may now complete the proof of \eqref{ENfNgtr1}.
 Using  \eqref{1Dredux}, and the fact that, by \eqref{suppgNj}, the functions $\theta_N * (\bbE_N^{(1)} g_{N,j})$ are supported in the disjoint intervals $J_{N,j}:=\frac{2j}N+(-\frac2N,\frac2N)$, we have
 \begin{align*} 
&2^N\|\Psi_N* (\bbE_N f_N)\|_p \gc  2^N \Big(\sum_{\frac N8<j<\frac N4} 
\big\| \theta_N*g_{N,j}\big\|_{L^p(\bbR)}^p 2^{-jp} \Big)^{1/p}
\\
&\gc \Big( \!\!\sum_{\frac N8<j<\frac N4} 
 \sum_{\nu\colon\frac{\nu}{2^N}\in J_{N,j}}
 \Big[|2^{-j}g_{N,j}'(\tfrac\nu{2^N})|^p \int_{\widetilde I_{N,\nu}}\!\!
|\Theta(2^Nt-\nu)|^p dt-\frac{c2^{(j-N)p}}{2^{N}}\Big]\,\Big)^{1/p}
 \end{align*} 
using the previous lemma in the last step. Since by \eqref{Thetaneq0} $$\int_{\widetilde I_{N,\nu}}|\Theta(2^N t-\nu)|^p dt =  2^{-N} \int_0^{1/8} |\Theta(t)|^p dt \gc 2^{-N},$$
 we obtain, for sufficiently large $N$, 
 \begin{equation*}
 2^N\|\Psi_N* (\bbE_N f_N)\|_p 
 \gc \Big( \sum_{\frac N8<j<\frac N4} 
 \sum_{\nu\colon\frac{\nu}{2^N}\in J_{N,j}}
 2^{-N}(1-c'2^{-pN/2})\Big)^{1/p} \gc 1. 
 \end{equation*} 
This completes the proof of \eqref{ENfNgtr1}, which together with \eqref{L5.3}
establishes Theorem \ref{lowerbounds-s=1thm}, and therefore also Theorem \ref{equiv-s=1thm}.

\section{Necessary conditions for boundedness when $s\le 0$}
\label{leftdottedlines-sect}

\subsection{\it The case 
$1<p\le  \infty$, $s=1/p-1$, $q>1$}
In these cases the operator $\bbE_N$ is not bounded in $B^{1/p-1}_{p,q}$  
because characteristic functions of cubes do not belong to the dual space $(B^{1/p-1}_{p,q})^*=B^{1/p'}_{p', q'}$; see \S\ref{necconds=1/p}.
This also applies when $p=\infty$, since 
$(b^{-1}_{\infty,q})^*=B^{1}_{1, q'}$; see \cite[\S1.1.5]{RuSi96}.

\subsection{\it The case $p=\infty$, $s=-1$, $q\le 1$}
\label{Ss-1}
We shall show 
\Be\label{b-1inftyq}
\|\bbE_N\|_{b^{-1}_{\infty,q}\to B^{-1}_{\infty,q}}\gc N.
\Ee
To prove this  we  argue by duality and first note that
\Be\label{B11inftylowerbd} 
\|\bbE_N\|_{b^1_{1,\infty}\to B^1_{1,\infty}}\approx N.
 \Ee
Indeed by Theorem \ref{equiv-s=1thm} we have 
$\|\bbE_N\|_{B^1_{1,\infty}\to B^1_{1,\infty}}\approx N$
and the lower bound is obtained by testing $\bbE_N$ on the Schwartz-function $f_N$ as in \eqref{fNSchwartz} satisfying
$\|f_N\|_{b^1_{1,\infty}} =\|f_N\|_{B^1_{1,\infty}} \lc N^{-1}$ and $\|\bbE_N f_N\|_{B^1_{1,\infty}}\gc 1$, \cf.
\eqref{L5.3}, \eqref{ENfNgtr1}.

To establish \eqref{b-1inftyq}, since $\|\bbE_N\|_{b^{-1}_{\infty,q}\to B^{-1}_{\infty,q}}\geq \|\bbE_N\|_{b^{-1}_{\infty,q}\to B^{-1}_{\infty,1}}$, by  \eqref{B11inftylowerbd} it suffices 
to prove that 
\Be\label{q<1redux}
 \|\bbE_N\|_{b^{-1}_{\infty,q}\to B^{-1}_{\infty,1}}\gc \|\bbE_N\|_{b^1_{1,\infty}\to B^1_{1,\infty}}.
\Ee  
We use that $(b^{-1}_{\infty,q})^*= B^1_{1,\infty}$ for 
$q\le 1$; see 
\cite[2.5.1/Remark 7]{triebel_teubner}.
Then for  $f\in \cS$
\Be
\|\SE_N f\|_{B^1_{1,\infty}}
=\|\SE_N f\|_{(b^{-1}_{\infty,q})^*}=
\sup_{{g\in\cS}\atop{\|g\|_{b^{-1}_{\infty,q}}\leq1}}\big|\langle \SE_N f , g\rangle\big|.
\label{bstar}
\Ee
Now, for each $g\in \cS$, since $f=\sum_{j=0}^\infty L_j\La_j f$ in $\mathcal{S}$, we have
\begin{align*} 
&|\langle \SE_N f , g\rangle|=|\langle f , \SE_Ng\rangle|
\,\leq \sum_{j=0}^\infty\|\La_jf\|_{1}
\|L_j(\SE_Ng)\|_{\infty}
\\  &\,\lc \|f\|_{b^{1}_{1,\infty}}\|\SE_Ng\|_{B^{-1}_{\infty,1}}
\leq  
\|f\|_{b^1_{1,\infty}}\,\|\SE_N\|_{b^{-1}_{\infty,q}\to B^{-1}_{\infty,1}}\,
\|g\|_{b^{-1}_{\infty,q}}.
\end{align*}
Inserting this into \eqref{bstar} we arrive at 
\[\|\SE_N f\|_{B^1_{1,\infty}}\le 
\|\SE_N\|_{b^{-1}_{\infty,q}\to B^{-1}_{\infty,1}}\,
\|f\|_{b^1_{1,\infty}}\]
and hence \eqref{q<1redux}.

  \section{Density and approximation}\label{densityandapprsect}
  In this section we show two results regarding  approximation by linear combinations of Haar functions.
  The main results  in \S\ref{failureofdensitysect}
   are relevant to the formulation of Theorems 
   \ref{schauder-besov} and \ref{schauder-besov-local} which rule out the case  $s=1$. We shall also obtain a positive result about approximation for the spaces 
  $b^{1/p}_{p,\infty}$. 
  
\subsection{\it The case $s=1$}\label{failureofdensitysect}

We shall show that no strongly admissible enumeration of the Haar system can form a Schauder basis on $B^1_{p,q}(\bbR^d)$ if $\frac{d}{d+1} \leq p < 1$ and $q>0$. Moreover if in addition  $0<q \leq p$ we shall show that the Haar system is not dense in  $B^1_{p,q}(\bbR^d)$.
It seems plausible that this last assertion would continue to hold for all $0<q\leq\infty$, but we do not have a proof in this generality.

Let us start with an auxiliary result. It is well-known that 
\[
 \|f\|_{{B^s_{p,\infty}(\mathbb{R}^d)}}\approx \|f\|_p+\sum_{j=1}^d\sup_{|h|\le 1}\frac{\|\Delta^2_{he_j}f\|_p}{|h|^{s}}\,,\]
 for all $p\leq 1$ and $d(1/p-1) < s <2$, see \cite[2.6.1]{triebel2}. Below we show that a partial lower bound 
actually holds  for all
$0<s<2$, which allows to incorporate the endpoint $s=d(1/p-1)=1$ (i.e., $p=d/(d+1)$) to our later results.
\begin{proposition} Let $0<s<2$ and $0<p\leq 1$. Then 
\begin{equation}\label{one-side}
\|g\|_p+\sum_{j=1}^d\sup_{|h|\le 1}\frac{\|\Delta^2_{he_j}g\|_p}{|h|^{s}} \lesssim \|g\|_{B^s_{p,\infty}}
\end{equation}
holds for any function $g \in L^1(\bbR^d)$\,.
\end{proposition}

\begin{proof} Let $\hpsi_0\in C^\infty_c(\SR^d)$ supported in $\{|\xi|<3/8\}$ and with $\hpsi_0(\xi)=1$ if $|\xi|\leq 1/4$, and let $\hpsi_k(\xi)=\hpsi_0(2^{-k}\xi)-\hpsi_0(2^{-k+1}\xi)$ if $k\geq1$. Consider a  standard dyadic frequency decomposition
  $g=\sum_{k=0}^\infty g_k$,  with $g_k=\psi_k*g$, which converges in $L^1$ and also a.e.
 Since \[(\sum_{k=0}^{\infty} \|g_k\|_p^p)^{1/p} \lesssim \|g\|_{B^s_{p,\infty}}\]
  we also have $\|g\|_p \lesssim \|g\|_{B^s_{p,\infty}}$. In addition, 
  for each $0<|h|\leq1$, using the trivial estimate
  $\|\Delta_{he_j}  ^2 g_k\|_p^p\le 4 \|g_k\|_p^p$, we see that
  \begin{align*} 
  \frac{\|\Delta^2_{he_j}\sum_{2^k\geq |h|^{-1}}
  g_k\|_p}{|h|^{s}}\,
  &\leq
  \Big(4\!\!\sum_{2^k\ge |h|^{-1}} (2^k |h|)^{-sp} \|g_k\|_p^p 2^{ksp}\Big)^{1/p}
  \\&\lc \sup_{2^k |h|\ge 1} 2^{ks} \|g_k\|_p.
  \end{align*}
 Let $\vphi\in\cS$ be such that $\hphi(\xi)=1$ if $|\xi|\leq 1$.
For every  $0<v<1$, we let $K_{ve_j}=\Delta^2_{ve_j}\vphi$, so that
  $$\widehat K_{ve_j}(\xi) =(e^{2\pi i\inn{v e_j}{\xi}}-1)^2 \hat\varphi(\xi).$$
  Then $K_{ve_j}$ is  a Schwartz function and we have the estimate 
  $$|K_{ve_j}(x)|\le C_N v^2 (1+|x|)^{-2N},$$
for a large $N>d$.  Hence,  for each $k\geq0$ such that $2^k |h|<1$, we have
  \begin{align*}& |\Delta_{he_j} ^2 g_k (x) |= 2^{kd} |K_{2^khe_j} (2^k\cdot)* g_k(x)|
  \\&\le C_N (2^k|h|)^2 \int 2^{kd} (1+2^k|y|)^{-2N} |g_k(x-y)| dy
  \\& \lc C_N (2^k|h|)^2 \sup_{y\in \bbR^d}  (1+2^k|y|)^{-N} |g_k(x-y)|. 
  \end{align*}
 Choosing $N> d/p$ we can apply the Peetre maximal function estimate
 to obtain 
 \begin{align*}
  &\frac{\|\Delta^2_{he_j}\sum_{2^k<|h|^{-1}}
  g_k\|_p}{|h|^{s}}\,
  \leq 
  \Big(\sum_{2^k< |h|^{-1}} (2^k |h|)^{-sp} \|\Delta_{he_j}^2  g_k\|_p^p 2^{ksp}\Big)^{1/p}
\\
  &\lc 
  \Big(\sum_{2^k< |h|^{-1}} (2^k |h|)^{(2-s)p} \| g_k\|_p^p 2^{ksp}\Big)^{1/p}
  \lc \sup_{k\ge 0} 2^{ks} \|g_k\|_p.
  \end{align*}
 Combining the two estimates yields the result.
\end{proof}

\begin{remark} The appropriate analogue for $B^s_{p,q}(\bbR^d)$-quasinorms 
for  $q<\infty$,
i.e.
\[\|g\|_p +\sum_{j=1}^d \Big(\int_{-1}^1 \frac{\|\Delta^2_{he_j}g\|_p^q}{|h|^{sq} }\frac{dh}{|h|} \Big)^{1/q} \lc \|g\|_{B^s_{p,q}}
\]
remains valid (when $0<s<2$) but is not relevant in this section.
\end{remark}

The following proposition is a modification of 
our argument in  \cite[Proposition 4.2]{gsu}. It shows the necessity of the condition $s<1$ in part (ii) of  Theorem \ref{schauder-besov}  and part (iii) of Theorem \ref{schauder-besov-local}.

\begin{proposition} 
\label{besovdensity}
There exists a Schwartz function $f$ supported in $(\frac 1{16}, \frac{15}{16})^d$ such that
for all $0<p\leq 1$  it holds
\Be
\label{liminf}\liminf_{N\to \infty} \|\bbE_N f -f\|_{B^1_{p,\infty}}>0.\Ee
\end{proposition}

\begin{proof}  
 Let $\eta\in C^\infty_c(\SR^d)$ be supported in $(\frac 1{16}, \frac{15}{16})^d$ with $\eta(x)=1$ if $x\in (1/8, 7/8)^d$.
Let $f(x)=x_1\,\eta(x)$. 
From \eqref{one-side}  we have 
\[\big\| \bbE_N f-f
 \big\|_{B^1_{p,\infty}} \gc 
\sup_{0<h\le 1}\frac{\big \|\Delta^2_{he_1}
\bbE_N f- \Delta^2_{he_1}f
\big \|_p}{h}\,,\]
Clearly, since $f$ is a Schwartz function, 
\[
 \|\Delta^2_{he_1}f
\big \|_p \lc |h|^2, \quad {0<h\le 1}.
\]
So, by an appropriate  triangle inequality,  it suffices to show that 
\Be \label{lowerbdlargeN}
\frac{\big \|\Delta^2_{2^{-N-2}e_1}
\bbE_N f
\big \|_p}{2^{-N-2}}\, \ge c>0,
\Ee 
for large $N$. We now recall a calculation  in \cite[Prop 4.2]{gsu}. 
Let $N>10$ and let $h\in (0,1/4)$. 
An explicit calculation shows that for $x\in (1/4,3/4)^d$
 \[
\SE_Nf(x)=\sum_{2^{N-2}\leq k<3\cdot 2^{N-2}}\tfrac{k+1/2}{2^N}\bbone_{[\frac{k}{2^N},\frac{k+1}{2^N})\times [0,1)^{d-1}}(x).\]
Then, under the additional assumption $0<h<2^{-N-1}$,  \[
\Delta_{he_1}\SE_Nf(x)=2^{-N-1}\sum_{2^{N-2}\le k\le 3\cdot 2^{N-2}}\bbone_{[\frac{k}{2^N}-h,\frac{k}{2^N})\times [0,1)^{d-1}}(x),
\] 
and
\begin{multline*}
\Delta^2_{he_1}\SE_Nf(x)=\\2^{-N-1}\!\!\sum_{2^{N-2}<k<3\cdot 2^{N-2}}\big(\bbone_{[\frac{k}{2^N}-2h,\frac{k}{2^N}-h)\times [0,1)^{d-1}}(x)-\bbone_{[\frac{k}{2^N}-h,\frac{k}{2^N})\times [0,1)^{d-1}}(x)\big).
\end{multline*}
Therefore, 
\[
\|\Delta^2_{he_1}\SE_Nf\|_{L^p([0,1]^d)}\gc 2^{N(1/p-1)}\,h^{1/p},
\]
and in particular
\[
\|\Delta^2_{2^{-N-2} e_1}\SE_Nf\|_{L^p([0,1]^d)}\gc 2^{-N},
\]
which implies  \eqref{lowerbdlargeN}.
\end{proof}

Finally, we conclude with the non-density result mentioned above.
\begin{corollary} \label{failureofdensitycor} Let $\frac{d}{d+1}\le p<1$, $0<q\le p$. Then $\Span \sH_d $ is not dense 
in $B^1_{p,q}(\SR^d)$.
\end{corollary} 
\begin{proof} By Proposition \ref{besovdensity} and  $B^1_{p,q}\hookrightarrow B^1_{p,\infty}$ we have, for some $f\in C^\infty_c$, 
\Be\label{failureofconvpq}
\liminf\limits_{N \to \infty} \|\bbE_N f-f\|_{B^1_{p,q}}=c>0,\Ee 
By Theorem \ref{basic-sequence-besov} the operators $\bbE_N$ are uniformly  bounded on  $B^1_{p,q}$. For $h\in \Span\sH_d $ we have $\SE_Nh=h$ for  $N\ge N_0(h)$, with sufficiently large $N_0(h)$.
Hence 
\[ 
\|\bbE_N f-f\|_{B^1_{p,q}} \lesssim 
\|\bbE_N [f-h] \|_{B^1_{p,q}} +
\|f-h \|_{B^1_{p,q}}  \lesssim \|f-h\|_{B^1_{p,q}} , \text{ for   $N\ge N_0(h)$},
\]
 and the density of $\Span\sH_d $  
in $B^s_{p,q}$ would yield a contradiction to \eqref{failureofconvpq}.
\end{proof}
\begin{remark}
When $d/(d+1)\leq p<1$, it follows from Theorem \ref{basic-sequence-besov}.iv (or vi), and from the results in \S\ref{localizationsect} below, that each strongly admissible enumeration $\cU$ of $\sH_d$ is a \emph{basic sequence} for $B^1_{p,p}(\SR^d)$, that is, $\cU$ is a Schauder basis for the subspace \[
{\overline{\Span 
\sH_d}}^{B^1_{p,p}}.\] 
It may be of interest to identify this subspace.
By  Oswald's  result in \cite{oswald}, it contains the class  ${\sB}^1_{p,p, (1)}(\SR^d) $
defined by first order differences.
\end{remark}

\subsection{\it An approximation result for $b^{1/p}_{p,\infty}$ when  $1\!<\!p\!<\!\infty$}
\label{positiveresultondensity}

In the limiting case $s=1/p$, recall that $\sH_d$ is contained in $B^{1/p}_{p,q}$ only if $q=\infty$.
We show an approximation result in this case when $1<p<\infty$. Recall that $B^{s}_{p,\infty}$ is not separable, and that $b^{s}_{p,\infty}$ denotes the closure of $\cS$ in $B^{s}_{p,\infty}$. Recall also, from Proposition 
\ref{intersectzero}, that  $b^{1/p}_{p,\infty}\cap\Span\sH_d=\{0\}$. However taking  closures one obtains

\begin{proposition}\label{prop_dens}
Let $1<p\le \infty$. Then  
\Be\label{inclusionproper}
b^{1/p}_{p,\infty}(\SR^d) \subsetneq 
\overline{\text{\rm span}(\sH_d)}^{B^{1/p}_{p,\infty}}\,.\Ee
\end{proposition}


\begin{proof} Let $1<p<\infty$.
In view of \cite[2.5.12]{Tr83}, we may use the equivalent norm
\[
\|f\|_{B^{1/p}_{p,\infty}}=\|f\|_p+\sup_{h\not=0}\frac{\|\Dt_h f\|_p}{|h|^{1/p}}.
\]

By dimensional considerations it is clear that the characteristic function of any dyadic cube $I$  of side length $2^{-k}$ can be written as a unique linear combination of Haar functions of frequency at most $2^{k-1}$ supported in the dyadic unit cube containing $I$.
It  therefore suffices to show 
that for every $f\in C^1_c(\SR^d)$ we have $\|f-f_N\|_{B^{1/p}_{p,\infty}}\to 0$,
where we choose\[
f_N=\sum_{I\in\sD_N} f(c_I)\bbone_I,
\]
and $c_I$ denotes the center of $I$. Let $\cI=\cI(f)$ be the family of $I\in \cD_N$ which intersect the support of $f$.
Clearly for $f\in C^1_c$ we have 
\Be\label{f-fNLp}
\|f-f_N\|_p\le  \sqrt{d} 2^{-N} 2^{-Nd/p}\|f'\|_\infty (\#\cI(f))^{1/p}
\lc_f 2^{-N}
\Ee
so that
$\|f-f_N\|_p\to 0$ for $N\to \infty$. For the main term it suffices to show that 
\Be
\sup_{h\not=0}\frac{\|\Dt_h (f-f_N)\|_p}{|h|^{1/p}} \lc 2^{-N(1-\frac 1p)},
\label{recall}
\Ee
and recall that we are assuming $p>1$.

For $j> N$ we define the sets
\Be
\cU_{N,j} =\Big\{(y_1,\ldots,y_d)\in\SRd\mid \min_{1\leq i\leq d}\dist (y_i, 2^{-N}\bbZ)\le 2^{-j-1}\Big\}.\label{UNj}
\Ee
Assume that $2^{-j-2}\leq|h|_\infty < 2^{-j-1}$, for some $j>N$.
If $I\in\sD_N$ then
\[
x\in I\setminus \cU_{N,j} \quad\mbox{implies} \quad x+h\in I,\]
and thus $\Dt_h f_N(x)=0$. So we have
\[
\|\Dt_h(f-f_N)\|^p_p  = A_{N}(h)+B_{N}(h)\]
where
\begin{align*}
A_N(h)&=\int_{\cU_{N,j}}|\Delta_h[f_N-f](x)|^p dx,
\\
B_N(h)&=\int_{\cU_{N,j}^\complement}|f(x+h)-f(x)|^p\,dx\;.
\end{align*}

In the second term we use $|\Dt_hf(x)|\leq|h|\int_0^1|\nabla f(x+sh)|ds$ to obtain
$B_N(h)\;\leq\; \|\nabla f\|_p^p\,|h|^p$ and thus 
\[
\sup_{|h|< 2^{-N-2}} B_N(h)/|h|\lc_f 2^{-N(p-1)}.
\] 

For the term $A_N(h)$  we use that  $\|f-f_N\|_\infty\leq C_f 2^{-N}$, and also that $f$ is compactly supported, and obtain the estimate
\Beas
A_N(h) &  \lesssim & \sum_{I\in\sD_N}\int_{I\cap\cU_{N,j}}|f(x+h)-f_N(x+h)|^p+|f(x)-f_N(x)|^p\,dx\,
\\
& \lesssim_f &  2^{-Np} 2^{N}2^{-j},
\label{prop_a1}
\Eeas
since $|I\cap\cU_{N,j}|\approx 2^{-j}2^{-(d-1)N}$.
 Hence $\sup_{|h|\approx 2^{-j}}  A_N(h)/|h| \lc_f 2^{-N(p-1)}.$
Putting the two estimates together we get 
\[
\sup_{|h|_\infty\le 2^{-N-2} } \frac{\|\Dt_h(f-f_N)\|_p}{|h|^{1/p}}
\lc_f 2^{-N(1-\frac 1p)}.
\]

Finally, if $|h|\gtrsim 2^{-N}$ we use \eqref{f-fNLp} to have
\[\sup_{|h|_\infty\ge 2^{-N}}
\frac{\|\Dt_h(f-f_N)\|_p}{|h|^{1/p}}\lesssim \frac{2\|f-f_N\|_p}{2^{-N/p}}
\lesssim_f 2^{-N(1-\frac1p)}.
\]
This shows \eqref{recall} and therefore $\|f-f_N\|_{B^{1/p}_{p,\infty}}\to 0$ for $p>1$, completing the proof of the inclusion
\eqref{inclusionproper} when $p<\infty$.
The case $p=\infty$ is immediate since for $f\in C^1_c$
$$\|f-f_N\|_{B^0_{\infty,\infty}}\lc \|f-f_N\|_\infty\lc 2^{-N}$$
by an elementary consideration. Finally, since 
$b^{1/p}_{p,\infty}(\SR^d) $ is closed in
$B^{1/p}_{p,\infty}(\SR^d) $ Proposition \ref{intersectzero}
tells us that the inclusion \eqref{inclusionproper} is proper.
\end{proof}

\begin{remark}
\label{dense_bs}
When $0<p\leq 1$, the same proof gives a version of \eqref{recall}, namely
\[
\sup_{h\not=0}\frac{\|\Dt_h (f-f_N)\|_p}{|h|^s} \lc 2^{-N(1-s)},\quad \mbox{if }\;s<1.
\]
This can be used similarly to show that $\sH_d$ is dense in the space $b^s_{p,\infty}$ when $d(1/p-1)<s<1$ and $d/(d+1)<p\leq 1$.
\end{remark}
  
  \begin{remark} Since $B^{1/p}_{p,\infty}$ is not separable, not every function $f\in B^{1/p}_{p,\infty}$ can be approximated by Haar expansions in the norm topology.  However, (local) weak$^*$ convergence does hold, with norm-uniformly bounded partial sums. More precisely, if $1<p\leq\infty$ and $\chi\in C^\infty_c(\SR^d)$, then 
\[
\big(f-S^\cU_R f\big)\chi\stackrel{{\rm w}^*}{\longrightarrow} 0 ,\mand \sup_{R\geq1}\big\|\chi\,S^\cU_R f\big\|_{B^{1/p}_{p,\infty}}\lesssim \|f\|_{B^{1/p}_{p,\infty}} 
\]
for all $f\in B^{1/p}_{p,\infty}$ and any strongly admissible enumeration.
This is a consequence of the duality relation $B^{1/p}_{p,\infty}= (B^{-1/p'}_{p',1})^*$ and the (local) norm convergence 
of  $S^\cU_R g\to g$ in the $B^{-1/p'}_{p',1}$ norm, when $g\in B^{-1/p'}_{p',1}$,  see Theorem \ref{schauder-besov-local}. 
We thank the referee for raising the question of weak* convergence.  
  \end{remark}
\section{Partial sums and localization}\label{localizationsect}

\subsection{\it Partial sums and strongly admissible enumerations}\label{SS91} We shall use a partition of unity to make statements on the structure of the partial sum operators $S_R^\cU$ associated with a strongly admissible enumeration $\cU$.

Let $\varsigma\in C^\infty_c$ be supported in a 
 $10^{-2}$ neighborhood of $[0,1)^d$ and so that 
 \Be\label{spacelocalization}
 \sum_{\nu\in \bbZ^d}\varsigma(\cdot-\nu)\equiv 1.\Ee
We shall denote 
 $\varsigma_\nu=\varsigma(\cdot-\nu)$, $\nu\in \bbZ^d$. 

In the sequel we will use the notation from Definition \ref{strongly-adm} and below. It is convenient to denote $\SE_{-1}(g)\equiv0$ and $T_{-1}[g,\fa]=\sum_{\mu\in\SZ^d}a_\mu\langle g,h^{\vec0}_{0,\mu}\rangle h^{\vec0}_{0,\mu}$.
 
\begin{lemma}  \label{basicdec} 
Let $\cU$ be a strongly admissible enumeration of $\sH_d $. Then, for every $R\in\SN$ and $\nu\in \bbZ^d$ there is an integer
$N_\nu=N_\nu(R)\ge -1$ and sequences $\fa^{\ka,\nu}$, $0\leq \kappa\leq b$,
whose terms belong to $\{0,1\}$,
such that for all locally integrable functions $g$ we have
\Be
S_R^\cU [g\varsigma_\nu]  = \bbE_{N_\nu}[g\varsigma_\nu] +
\sum_{\kappa=0}^{b} T_{N_\nu+\ka}[
g\varsigma_\nu,\fa^{\ka,\nu}].
\label{SET}
\Ee
\end{lemma}

\begin{proof} 
We write \Be\label{SRformula}S_R^\cU[g\varsigma_\nu]= 
\sum_{n=1}^R u_n^*(g\varsigma_\nu) u_n=
\sum_{n=1}^R 2^{k(n) d} 
\inn {g\varsigma_\nu}{h^{\ep(n)}_{k(n),\mu(n)} }h^{\ep(n)}_{k(n),\mu(n)}.\Ee
Note that if $u_n^*(g\varsigma_\nu)\neq 0$ then necessarily $u_n$ is supported in $I_\nu^{**}$.
Let 
$$ K_\nu = \max \big\{k(n):  \supp (h^{\ep(n)}_{k(n),\mu(n)}) \subset I^{**}_\nu,\;n=1,\dots, R\big\}.$$
If $K_\nu\le b$  the asserted formula holds with $N_\nu=-1$. We therefore may assume $K_\nu>b$.

We let $n_\nu^*\in  [1,R]$  such that  $k(n_\nu^*)=K_\nu$.
Now if   $h_{k',\mu'}^{\ep'}$
 is any other Haar function supported in $I^{**}_\nu$ there is a unique 
$n'\in \bbN$ such that 
 $h_{k',\mu'}^{\ep'}= h_{k(n'),\mu(n')}^{\ep(n')}
 $. If in addition 
 $k'\le K_\nu-b$  (in other words if for  $u_{n'}= h_{k(n'),\mu(n')}^{\ep(n')}$ 
 we have that $|\supp(u_{n'})|\ge |\supp(u_{n_\nu^*})| 2^b$)  then by the admissibility condition
 we must have $n'\le n^*_\nu$, in particular $n'\le R$. That means that all Haar functions with frequency $2^k$ and $k\le K_\nu-b$  which are supported in $I_\nu^{**}$ arise in the expansion \eqref{SRformula}.  All other Haar functions that arise in this expansion have frequencies $2^k$ with
 $K_\nu-b+1\le k\le K_\nu$. This establishes the assertion with $N_\nu= K_{\nu}-b+1$. The  functions $\fa^{\ka,\nu}$ defined on $\bbZ^d\times \Upsilon$ take values in $\{0,1\}$.
   \end{proof}
	\emph{Remark.}
	Formula \eqref{SET} can be extended to all $g\in B^s_{p,q}$, when the indices $(s,p,q)$ are as in Theorems
	\ref{basic-sequence-besov} and \ref{p=inftythm}. In that case, one must interpret
	\[
	S_R^{\cU}(g)=\sum_{j=0}^\infty S_R^\cU\big(L_j\La_j g);
	\]
	see Remarks \ref{R_ENp<1} and \ref{R_ENp>1}.

	\begin{proposition}\label{P_supEN}
	Suppose that
	\Be
	\sup_{N\geq0}\,\|\SE_N\|_{B^s_{p,q}\to B^s_{p,q}}\,+\,\sup_{{N\geq-1}\atop{\|\fa\|_\infty\leq1}}\|T_N[\cdot,\fa]\|_{B^s_{p,q}\to B^s_{p,q}}<\infty.
	\label{supEN}
	\Ee
	Then, for every strongly admissible enumeration $\cU$ and every cube $Q$ it holds
	\Be
	\sup_{R\in \bbN} \;\rhoop\big(S^\cU_R, B^s_{p,q}, Q\big)<\infty.
	\label{OpSQ}
	\Ee
	Moreover, $\cU$ is a local basic sequence of $B^s_{p,q}(\SR^d)$, that is
	\Be
	\lim_{R\to\infty}\big\|\chi\cdot(S_R^\cU f-f)\big\|_{B^s_{p,q}}=0,
	\label{limchiSR}
	\Ee
	for all $\chi\in C^\infty_c(\SR^d)$ and all $f\in{\overline{\Span\sH_d}}^{B^s_{p,q}}$.
\end{proposition}	
  \begin{proof}
	Using Lemma \ref{basicdec}, the bound in \eqref{OpSQ} follows from \eqref{supEN}.
	We now show the last assertion. Let $\chi\in C^\infty_c(\SR^d)$ and $f\in{\overline{\Span\sH_d}}^{B^s_{p,q}}$.
	Suppose that $\supp\chi\subset(-N,N)^d$, and pick any $\tchi\in C_c^\infty$ such that $\tchi\equiv 1$ in $[-N,N]^d$ and $\supp\tchi$ contained in $Q:=(-2N,2N)^d$. Observe that 
\Be
u^*_n(g)\,=\,u^*_n(\tchi\,g),\quad \mbox{if $g\in B^s_{p,q}$ and $\supp u_n\subset[-N,N]^d$},
\label{aux_unchi}
\Ee
so we also have 
\Be
\chi\cdot S_R^\cU [g] 
=\chi\cdot S_R^\cU[\tchi g],\quad\forall\;g\in B^s_{p,q}.
\label{aux_SRchi}
\Ee
Given $\e>0$, let $h\in\Span\sH_d$ be such that $\|f-h\|_{B^s_{p,q}}<\e/(1+A)$, with $A$ the constant in 
\eqref{OpSQ}. Let $R_0=R_0(h)$ be such that $S^\cU_R[h]=h$ for $R\geq R_0$. Then, for all such $R$ we have 
\Beas
\big\|\chi\cdot(S_R[f]-f)\big\|_{B^s_{p,q}} & = & \
\big\|\chi\cdot\big(S_R[f-h] +h- f\big)\big\|_{B^s_{p,q}}\\
& \lesssim & \big\|\chi\cdot S_R[\tchi (f-h)]\big\|_{B^s_{p,q}} +\big\|\chi\cdot(h- f)\big\|_{B^s_{p,q}}\\
& \lesssim & (A+1)\,
\big\|f- h\big\|_{B^s_{p,q}}<\e,\Eeas
where in the second line we have used \eqref{aux_SRchi} with $g=f-h$.
	\end{proof}

   \subsection{\it Bourdaud localizations of Besov spaces}\label{bourdaudsect}

In the unbounded setting of $\SR^d$,  the $B^s_{p,q}$-norms do not satisfy ``localization properties'' when $p\not=q$; see  e.g. the discussion in \cite[p. 66]{RuSi96}. 
At the endpoint cases considered here, this creates a difficulty when trying to derive `global' Schauder basis properties from the local ones in the previous subsection. This difficulty 
is not present in the case of $F^s_{p,q}$ spaces; see \cite{gsu, gsu-endpt}.

To handle this problem one may consider the class of $\ell^p$-local Besov spaces introduced by G. Bourdaud \cite{bourdaud}
\Be
\label{loc-B}
\lBs=\Big\{f\in S'\mid \|f\|_{\lBs}=\Big[\sum_{\nu\in\SZd}\|\vsig(\cdot-\nu)\cdot f\|_{B^s_{p,q}}^p\Big]^{1/p}<\infty\Big\}
\Ee
where $\vsig\in C^\infty_c(\SR^d)$ with $\sum_{\nu\in\SZd}\vsig(\cdot-\nu)\equiv1$ as in 
\eqref{spacelocalization}. In \cite{bourdaud} (see also \cite[2.4.7]{triebel2}) it is shown that this definition does not depend on the particular choice of $\varsigma$, and that $\lBs=B^s_{p,q}$ if and only if $p=q$. Moreover  one has the embeddings
\begin{align}
&B^s_{p,q}\hookrightarrow\lBs \,  \mbox{ if } 0<q\leq p,\label{Bloc1}
\\
&\lBs\hookrightarrow B^s_{p,q} \,\mbox{ if } p\leq q\leq \infty.\label{Bloc2}
\end{align}
Using this notation we can prove the following.

\begin{theorem} \label{schauderB-bourdaud:d/p-d}
Let $s\in\SR$ and $0<p,q\leq \infty$. Suppose that \eqref{supEN} holds.
	Then,  every strongly admissible enumeration $\cU$ of $\sH_d$ is a basic sequence of $\big(B^s_{p,q}\big)_{\ell^p}$.
Moreover, $\cU$ is a Schauder basis of $\big(B^s_{p,q}\big)_{\ell^p}$ in each of the cases (i) to (iv) in Theorem \ref{schauder-besov-local}.
\end{theorem}

%
%
%
%

 \begin{proof}
For the first assertion it suffices to show that the operator norms of $S_R\equiv S_R^\cU$ in $\lBs$ are uniformly bounded in $R$. To do so we use the
 assumption \eqref{supEN}, together with Lemma \ref{basicdec}.

 Observe first that $\varsigma_{\nu'} S_R (f\varsigma_\nu)=0$ whenever $|\nu-\nu'|_\infty\geq3$. 
 Hence
 \begin{align*}
 \big\|S_R f\big\|_{\lBs} & =  \Big(\sum_{\nu'}\big\|\varsigma_{\nu'} S_R\big(\sum_{\nu}\varsigma_\nu f\big)\big\|_{B^s_{p,q}}^p\Big)^\frac1p
\\&\lesssim \Big( \sum_{\nu'} \sum_{\nu\,:\,|\nu-\nu'|_\infty\leq2}\Big\| 
  \varsigma_{\nu'}S_R(f\varsigma_\nu)\Big\|_{B^s_{p,q}}^p \Big)^{1/p}
 \\
 &\lesssim \Big( \sum_{\nu} \big\| S_R(f\varsigma_\nu)\big\|_{B^s_{p,q}}^p \Big)^{1/p},
 \end{align*}
using in the last step that $\varsigma_{\nu'}$ is a uniform multiplier in $B^s_{p,q}$; see \cite[4.2.2]{triebel2}.
Then Lemma  \ref{basicdec}  and \eqref{supEN} 
give
 \Beas
 \big\|S_R f\big\|_{\lBs} &\lesssim & \Big( \sum_{\nu} \big\| \bbE_{N_\nu}(f\varsigma_\nu)\big\|_{B^s_{p,q}}^p 
+
\big\| \sum_{\kappa=0}^b T_{N_\nu}[f\varsigma_\nu, \fa^{\ka,\nu}]\big\|_{B^s_{p,q}}^p \Big)^{1/p} 
 \\
 &\lc_b &
 \Big( \sum_{\nu} \big\|f\varsigma_\nu\big\|_{B^s_{p,q}}^p \Big)^{1/p}  =\, \|f\|_{\lBs}.
 \Eeas
This shows the first part. 
Also, the Schauder basis property will hold if and only if $\Span\sH_d$ is dense in $\lBs$.

We now show that density holds in the range of Theorem \ref{schauder-besov-local}. Since $p<\infty$, for each $f\in\lBs$
and $\e>0$ there is some $g\in B^s_{p,q}$ with compact support such that $\|f-g\|_{\lBs}<\e$. Moreover, in the asserted range $\Span\sH_d$ is dense in $B^s_{p,q}$, so 
if $\supp g\subset(-N,N)^d=Q$, then by Proposition \ref{P_supEN} we may find a sufficiently large $R$ such that $\|g-S_Rg\|_{B^s_{p,q}}<\e/|Q|^{1/p}$. Since also $\supp (S_Rg)\subset Q$ we deduce that
\[
\|g-S_Rg\|_{\lBs}\lesssim |Q|^{1/p}\| g-S_Rg\|_{B^s_{p,q}}<\e,
\]
which completes the proof.
  \end{proof}

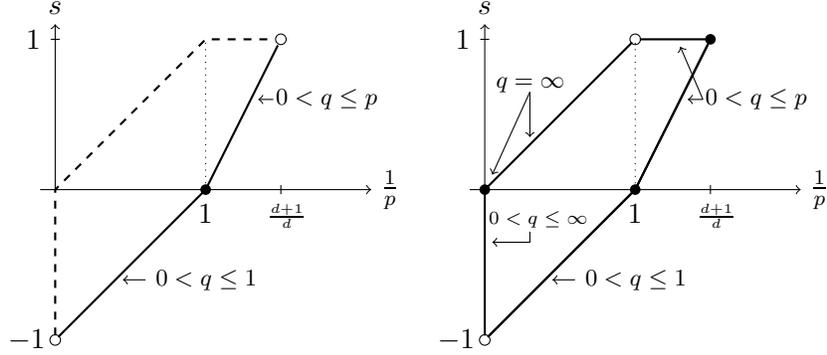
\begin{figure}[ht]
 \centering
\subfigure
{\begin{tikzpicture}[scale=2]
\draw[->] (-0.1,0.0) -- (2.1,0.0) node[right] {$\frac{1}{p}$};
\draw[->] (0.0,0.0) -- (0.0,1.1) node[above] {$s$};

\draw (1.0,0.03) -- (1.0,-0.03) node [below] {$1$};
\draw (1.5,0.03) -- (1.5,-0.03) node [below] {{\tiny $\;\;\frac{d+1}{d}$}};
\draw (0.03,1.0) -- (-0.03,1.00) node [left] {$1$};

\draw[dotted] (1.0,0.0) -- (1.0,1.0);

\draw[dashed, thick] (0.0,-0.97) -- (0.0,0.0) -- (1.0,1.0) -- (1.45,1.0);

\draw[thick] (1.48,0.96) -- (1.0,0.0);
\fill (1,0) circle (1pt);
\draw (1.5,1) circle (1pt);
\draw[<-, thin] (1.35,0.6)--(1.45,0.6);
\node [right] at (1.4,0.6) {{\footnotesize $0<q\leq p$}};

\draw[thick] (0.02,-0.98) -- (1.0,0.0);
\draw (0,-1) circle (1pt) node [left] {$-1$};
\draw[<-, thin] (0.45,-0.6)--(0.6,-0.6);
\node [right] at (0.6,-0.6) {{\footnotesize $0<q\leq 1$}};

\end{tikzpicture}
}
\subfigure
{\begin{tikzpicture}[scale=2]
\draw[->] (-0.1,0.0) -- (2.1,0.0) node[right] {$\frac{1}{p}$};
\draw[->] (0.0,-0.97) -- (0.0,1.1) node[above] {$s$};

\draw (1.0,0.03) -- (1.0,-0.03) node [below] {$1$};
\draw (1.5,0.03) -- (1.5,-0.03) node [below] {{\tiny $\;\;\frac{d+1}{d}$}};
\draw (0.03,1.0) -- (-0.03,1.00) node [left] {$1$};

\draw[dotted] (1.0,0.0) -- (1.0,1.0);
\draw[thick] (0.0,-0.97) -- (0.0,0.0) -- (0.98,0.98);
\draw[thick] (1.03,1) -- (1.5,1.0) -- (1.0,0.0) --
(0.02,-0.98);

\draw[thick] (1.48,0.96) -- (1.0,0.0);
\fill (1,0) circle (1pt);
\fill (1.5,1) circle (1pt);
\draw (1,1) circle (1pt);
\draw[<-, thin] (1.35,0.6)--(1.45,0.6);
\node [right] at (1.4,0.6) {{\footnotesize $0<q\leq p$}};
\draw[<-, thin] (1.3,.95)--(1.45,0.6);

\draw[thick] (0.02,-0.98) -- (1.0,0.0);
\draw (0,-1) circle (1pt) node [left] {$-1$};
\draw[<-, thin] (0.45,-0.6)--(0.6,-0.6);
\node [right] at (0.6,-0.6) {{\footnotesize $0<q\leq 1$}};

\node [left] at (0.6,0.7) {{\footnotesize $q=\infty$}};
\draw[->, thin] (0.3,0.65)--(0.3,0.35);
\draw[->, thin] (0.3,0.65)--(0.05,0.1);
\fill (0,0) circle (1pt);

\node [right] at (-0.04,-0.2) {{\tiny $0<q\leq \infty$}};
\draw[->, thin] (0.3,-0.28)--(0.3,-0.35)--(0.05,-0.35);
\end{tikzpicture}
}
\caption{The left  caption shows the region in which strongly admissible enumerations form a Schauder basis of the Bourdaud localization $(B^s_{p,q})_{\ell^p}\;$ ; here always $q<\infty$. The right caption shows the corresponding region for the basic sequence property.
}
\label{fig4}
\end{figure}

Finally, we gather as a corollary the positive Schauder results in the original scale of $B^s_{p,q}$ spaces. 

\begin{corollary}
Every strongly admissible enumeration $\cU$ of $\sH_d$ is a Schauder basis of $B^s_{p,q}(\SR^d)$ in each of the cases (i), (ii), (iii) 
in Theorem \ref{schauder-besov}.
\end{corollary}
\begin{proof}
When $q=p$ the result is a consequence of the identity $B^s_{p,p}=(B^s_{p,p})_{\ell^p}$ and the previous theorem. 
This covers the case (iii) in 
Theorem \ref{schauder-besov}. For the other cases, in which  $(1/p,s)$ lies in the interior of the pentagon $\fP$,
one proceeds by real interpolation as follows. Pick two numbers $s_0,s_1$ such that $s_0<s<s_1$ and $(1/p,s_i)\in\fP$, $i=0,1$.
Then, for some $\theta\in(0,1)$ we have 
\[
B^s_{p,q}=\big(B^{s_0}_{p,p},B^{s_1}_{p,p}\big)_{\theta,q}, \quad 0<q\leq \infty.
\]
Then the uniform boundedness of $S^\cU_R$ on $B^s_{p,q}$ follows by interpolation from the diagonal cases. 
\end{proof}

   \subsection{\it Error estimates for compactly supported functions}
   Here we include a technical result related to localization which will be used in the proof of
  Theorem \ref{locallowerboundsthm} 
   below.

Let $f$ be supported in a dyadic cube $Q$ with sidelength $\ell(Q)\ge1$. Since the function $\Lambda_j f$ does not have compact support, the terms $L_k \bbE_N L_j\La_j f(x)$
   will contribute  for $x$ far away from the cube. We give a crude estimate
	which will suffice for our later application.
 
\

 Let  $\zeta\in C^\infty_c(\bbR^d)$ 
be supported on $(-2,2)^d$ and such that $\zeta\equiv1$ on $[-\frac 32,\frac 32]^d$.
If $y_Q$ is  the center of $Q$, we define \[
\zeta_Q(y)= \zeta\big((y-y_Q)/ \ell(Q)\big).
\] 
Clearly $\zeta_Q f=f$ for every distribution $f$ supported in $Q$.
Moreover, this property continues to hold with $\zeta_Q$ replaced by $\tzeta_Q$, where $\tzeta(x)=\zeta(2x)$.
For $n\ge 1$ we let 
$$\zeta_{Q,n} (y)= \zeta\big(2^{-n}(y-y_Q)/ \ell(Q)\big)
-\zeta\big(2^{-n+1}(y-y_Q)/ \ell(Q)\big).$$
Note that $\zeta_{Q,n}$ has support in $\big\{\frac34\cdot 2^n\ell(Q)<|y-y_Q|_\infty<2^{n+1}\ell(Q)\big\}$, and that $\sum_{n\geq1}\zeta_{Q,n}\equiv1$.

\begin{lemma} \label{faraway} Let $s\leq1$, $0<p\leq1$ and $0<q\leq\infty$.  Then, for every $M_1>1$ there exists a constant $C_{M_1}>0$ such that, if $f\in B^s_{p,q}(\bbR^d)$ is supported in a cube $Q$ with size $|Q|\geq1$, then
   \Be
\label{kNglobal}
   \|{L_k} \bbE_N  L_j[\zeta_{Q,n} \La_j f]
   \|_p 
   \leq C_{M_1}\,
   2^{-k/p} 2^{-jM_1} 2^{-nM_1}  \|f\|_{B^s_{p,q}} ,
   \Ee  
	for all $n\geq1$, $k\geq0$, $j\geq N$ and $N\geq1$.
	
	The same holds if $\SE_N$ is replaced by 
	$ T_N  [\cdot, \fa]$ with $\|\fa\|_\infty\leq1$.
	
         \end{lemma}
      \begin{proof} 
      Let $\phi_j (x)= 2^{jd} \phi(2^j x)$ be the convolution kernel of $\La_j$,
			with $\phi\in\cS$.
      Let $$F_{j,n}(x):= \zeta_{Q,n}(x) \Lambda_j f(x)=      
      \zeta_{Q,n}(x) 
      \biginn{\phi_j(x-\cdot) \tzeta_Q(\cdot)}{f}$$
      where we have used $f=f\tzeta_Q$ for the second equation and  the pairing $\inn{\cdot}{\cdot}$ is  in the sense of tempered distributions.
			
Pick a large $\gamma\in2\SN$ such that 
$B^s_{p,q}\subset B^{-\gamma}_{2,2}$ (e.g., $\ga > d(\frac1p-\frac12)-s$). Then by duality 
     \Bea |F_{j,n}(x)| &\lc &|\zeta_{Q,n}(x) |\,\big\| (I-\Delta)^{\gamma/2} 
      \big(\phi_j(x-\cdot) \tzeta_Q(\cdot) \big)\big\|_2 \;
      \|f\|_{B^{-\gamma}_{2,2}} \nonumber
      \\
      & \lc_{M_2}& |Q|^{1/2}\, 2^{j(d+\gamma)}\, (1+2^{j+n}  \ell(Q))^{-M_2}\, 
 \|f\|_{B^s_{p,q}}.\label{M2}
 \Eea
Observe that $F_{j,n}$, and hence $L_k\bbE_N L_j  [F_{j,n}]$, are all supported in a set of diameter $C 2^{n} \ell(Q)$.
Then, if $k\leq N$ we have
      \Beas
      \|{L_k} \bbE_N L_j (F_{j,n})
   \|_p
   &  \lc &  (2^{n d}|Q|)^{1/p} 
     \|{L_k} \bbE_N L_j (F_{j,n})
   \|_\infty\\
	&  \lc &  (2^{n d}|Q|)^{1/p} 
     \|F_{j,n}
   \|_\infty.
   \Eeas
	Inserting the bound \eqref{M2} into this expression, with a sufficiently large $M_2$,
	and using that $k\leq N\leq j$, one easily obtains \eqref{kNglobal}.
  
	Assume now that $k>N$. We may use Proposition \ref{localizedproposition}.i to obtain
	\[
	\big\|L_k\SE_N L_j(F_{j,n})\big\|_p\lesssim\,2^{-\frac kp}\,2^{j(\frac dp-d)}\,2^{N(d-\frac{d-1}p)}\,\|\cM_jF_{j,n}\|_p.
	\]
	By the support properties of $F_{j,n}$ we have
	\[
	\|\cM_jF_{j,n}\|_p\lesssim \,(2^{nd}|Q|)^{1/p}\,\|F_{j,n}\|_\infty,
	\]
   so again, using \eqref{M2} with a sufficiently large $M_2$, and the assumption $N\leq j$,  one easily derives \eqref{kNglobal}.  
       \end{proof}


   \section{The case  $s=d(\frac 1p-1)$ when $q>p$.}
   \label{ple1slowerendpt}

In this section we restrict to the cases $q>p$ in the line $s=d/p-d$.
We shall see that the individual operators $\SE_N$ are not bounded, and hence 
 positive results are not expected in this range.

\begin{theorem}\label{noestimateforq>p}
Let $0<p\leq 1$. If $q>p$ then the operators $\bbE_N$ are unbounded on $B^{d/p-d}_{p,q}(\bbR^d)$.
\end{theorem}

We shall actually prove something stronger, namely
optimal estimates for the local version of the operator norms 
$\rhoop(\SE_N,B^s_{p,q},Q)$ defined in \eqref{localnorms}.
This may be of interest in the context of Besov spaces in bounded domains; see Remark \ref{RemTd}. We remark that Oswald \cite{oswald} also proved some lower bounds in a local setting which grow with $N$.  The following theorem provides optimal growth rates.

 \begin{theorem}\label{locallowerboundsthm}
 (i) If $0<p\le 1$ and $p\le q\le \infty$, then 
 there is a constant $c_1=c_1(p,q)>0$  so that
\[ \rhoop\big(\bbE_N,B^{d/p-d}_{p,q}, Q\big)\,\ge \, c_1  \; (2^{Nd}|Q|)^{\frac 1p-\frac 1q}.\]
 
(ii)  If in addition
  $\frac{d}{d+1}\leq p\le q\le  1$,  then there is a constant
    $c_2=c_2(p,q)$, such that for any dyadic cube  $Q$ with side length $\ge 1$ and  any $N>10$
\Be\label{normequivalencelocal} c_1    \,\le\, \frac{
 \rhoop\big(\bbE_N,B^{d/p-d}_{p,q}, Q\big)}{ (2^{Nd}|Q|)^{\frac 1p-\frac 1q}} 
 \le \, c_2  \,
 .\Ee
\end{theorem}

\begin{remark}
From \cite[2.11.3]{Tr83} it is known that, when $0<p\leq 1$ and $1<q<\infty$, it holds
\[
\big(B^{d/p-d}_{p,q}\big)^*=B^0_{\infty, q'}.\]
As $q'<\infty$, this space does not contain the dual functionals $u^*_n$.  
In particular, the restriction in $q$  in part (ii) of Theorem \ref{locallowerboundsthm} is natural,
since for $q>1$ and $Q_0=(0,1)^d$ we have
 \[
\rhoop\big(\bbE_0,B^{d/p-d}_{p,q}, Q_0\big)=\|\bbone_{Q_0}\|_{B^{d/p-d}_{p,q}}\,\|\bbone_{Q_0}\|_{(B^{d/p-d}_{p,q})^*}=\infty;
\] see also \cite[Thm 2.ii.a]{oswald}. 
\end{remark}

  \subsection{\it Proof of lower bounds in Theorem \ref{locallowerboundsthm}}\label{lowbdslocaldp-1sect}
  We fix $0<p\leq1$ and choose an positive  integer $M>d/p-d$.
  
   Let $\eta\in C^\infty_c(\bbR)$ be an odd function, supported on $(-1/2,1/2)$, and such that $\int_0^{1/2} \eta(t) dt=1$ and 
   $\int_0^{1/2} t^n\eta(t) dt=0$ for $n=1,\dots, M$.
     Let further
     \Be\label{gldef} g_l(x_1,\ldots, x_d)= 2^{ld}\prod_{i=1}^d \eta(2^l x_i),\Ee
     so that $\int g_l(x) P_M(x) dx=0$ whenever $P_M$ is a polynomial of degree $\le M$.
By the properties of $\eta$, if $l\geq N$ we have
   \Be\label{hNdef}  
	\bbE_N (g_l)(x)=2^{Nd}\prod_{i=1}^d \big(\bbone_{[0, 2^{-N})}(x_i)- \bbone_{[-2^{-N}, 0)}(x_i)\big)=:h_N(x).\Ee
Notice that $h_N$ is not itself a Haar function, but up to a factor $(-2)^d$, it is a translate of a Haar function with Haar frequency $2^{N-1}$.
Moreover, we also have
   \[ \bbE_N [ g_l(\cdot-\nu)] =h_N (\cdot-\nu),\quad \text{ if } \nu\in 2^{-N} \bbZ^d, \quad l\geq N. \]
   
   Let $\{\fz_m\}_{m=1}^\infty$ be an enumeration of $\bbZ^d$, and define   
   \Be\label{defof-f} f_N(x)= \sum_{m=1}^\infty a_m g_{N+m} (x-2^{-N+5}\fz_m). \Ee
Observe that the summands have disjoint supports. Also
   \Be\label{defof-Ef} \bbE_N f_N= \sum_{m=1}^\infty a_m h_N (\cdot-2^{-N+5}\fz_m). \Ee
	 We claim that 
   \Be\label{upperbdforf}
   \|f_N\|_{B^{d/p-d}_{p,q} } 
   \lc \Big(\sum_{m=1}^\infty |a_m|^q\Big)^{1/q}
   \Ee
   and
   \Be\label{lowerbdforENf}
	\|\SE_N f_N\|_{B^{d/p-d}_{p,q}}
   \gc  \Big(\sum_{m=1}^\infty |a_m|^p\Big)^{1/p}\,.
   \Ee
 This clearly implies that $\bbE_N$ cannot be a bounded operator  on $B^{d/p-d}_{p,q}(\bbR^d)$ unless
   $q\le p$.

We first show   \eqref{lowerbdforENf}. To do so we construct specific functions $\Psi_n$ such that 
\Be
 \|g\|_{B^{d/p-d}_{p,q} }\, \geq\, \|g\|_{B^{d/p-d}_{p,\infty} } \,\gtrsim \,\sup_{n\geq1}\, 2^{n(\frac dp-d)} \|\Psi_n * f\|_p.
\label{Psi2}\Ee 
Let $\psi\in C^\infty_c(\bbR)$ be supported in $(-1/2,1/2)$, with 
   \[\int\psi(t) t^l dt=0, \quad l=0,\dots, M\,,\] and such that,  for some $\eps>0$, 
   \Be\label{lowerbdforpsiconv}
   \psi* (\bbone_{[0,\frac 12)}- \bbone_{[-\frac 12,0)})(t) \ge c>0 \quad \text{ when }  t\in [\tfrac 12,\tfrac 12+\eps].
   \Ee
 We then define \Be\label{PsiNdef}
   \Psi_n(x)= 2^{nd} \prod_{i=1}^d\psi(2^nx_i),
   \Ee
   which has enough vanishing moments to guarantee the validity of \eqref{Psi2}; see \cite[2.5.3]{triebel2}. In particular,
		\[
		\|\SE_Nf_N\|_{B^{d(\frac 1p-1)}_{p,q} }\,\gtrsim  2^{N(\frac dp-d)} \|\Psi_{N+1} * (\SE_Nf_N)\|_p.
		\]
		Next, using  \eqref{lowerbdforpsiconv} one shows that, for 
		$x\in 2^{-N+5} \fz_m+2^{-N-1}[\tfrac 12,\tfrac 12 +\eps]^d$,
   \[
	\Psi_{N+1}*h_N(x-2^{-N+5}\fz_m)
   \ge 2^{Nd} c^d,
	\]
	and therefore
	\Be
	\big\|\Psi_{N+1}*h_N(\cdot-2^{-N+5}\fz_m)\big\|_p\,\gtrsim\, 2^{N(d-\frac dp)}.
	\label{Psicd}
	\Ee
   Also the functions 
$   \Psi_{N+1}*h_N(\cdot-2^{-N+5}\fz_m)
$ have disjoint supports so that
   \Beas
   \big\|\Psi_{N+1}*\SE_Nf_N\big\|_p &  = &  
   \Big(\sum_{m=1}^\infty |a_m|^p \big\|\Psi_{N+1}*h_N(\cdot-2^{-N+5}\fz_m)\big\|_p^p\Big)^{1/p}
   \\
   & \gc & \Big(\sum_m|a_m|^p\Big)^{1/p} 2^{-Nd(\frac 1p-1)}
   \Eeas
   and \eqref{lowerbdforENf} follows.
   
   To prove \eqref{upperbdforf} we examine $L_j g_l$ with $l=N+m$  and use the cancellation of the convolution kernel $\beta_j$ of $L_j$ when $j\ge l$, and the cancellation of $g_l$ for $j<l$. Here cancellation refers to $M$  vanishing moments. As a consequence we obtain the estimate
   \Be
   |L_j g_l(x)| \lc \,
	\begin{cases} 2^{ld} \bbone_{[-1,1]^d}(2^l x) 2^{-M|l-j|} &\text{ for } j\ge l, \\
   2^{jd} \bbone_{[-1,1]^d}(2^j x) 2^{-M|l-j|} &\text{ for } j\le l;
   \end{cases}
   \label{moment_gl}
	\Ee
	see a similar argument in the proof of \cite[Lemma 2.2]{gsu}.
From here one easily obtains   
   \Be
2^{jd(\frac 1p-1)}\|L_jg_l\|_p  \lc \left\{\Ba{ll}
2^{-(M-d(\frac 1p-1))|l-j|} &\text{ if } j\ge l
\\
2^{-M|l-j|} &\text{ if } j\le l
\Ea\right\}  \;\leq\,  2^{-\dt|l-j|},
\label{Ljgjdt}
   \Ee
	if we set $\dt=M-d(\frac 1p-1)>0$.
This leads to 
\begin{align*}
2^{j d(\frac 1p-1)}
 \|L_j f_N\|_p &\leq
2^{j d(\frac 1p-1)}
 \Big(\sum_{m=1}^\infty|a_m|^p\,\|L_j g_{N+m} (\cdot-2^{-N+5} \fz_m)\|_p^p\Big)^{1/p}
\\&\lc 
\Big(\sum_{m=1}^\infty |a_m|^p 2^{-|N+m-j|\,\dt p}\Big)^{1/p},
\end{align*}
and consequently, 
\begin{align*}
&\Big(\sum_{j\ge 0}\big [2^{j d(\frac 1p-1)} \|L_j f_N\|_p\big ]^q\Big)^{1/q}\lc \Big(\sum_{j\ge 0} \Big(\sum_{m=1}^\infty |a_m|^p\, 2^{-|N+m-j|\,\delta p }\Big)^{q/p}\Big)^{1/q}.
 \end{align*}
 Since  $q\ge p$ we can apply the triangle inequality in $\ell^{q/p}$ to bound the previous expression,  
 by
 \begin{align*}& \Big(\sum_{j\ge 0} \Big(\sum_{n\in \bbZ} |a_{n+j-N} |^p 2^{-|n|\delta p} \Big)^{q/p}\Big)^{1/q}
 \\ &\lc 
 \Big(\sum_{j\ge 0} \sum_{n\in \bbZ} |a_{n+j-N} |^q 2^{-|n|q\frac\delta2 }\Big)^{1/q}
    \lc \Big(\sum_{m=1}^\infty|a_m|^q\Big)^{1/q}.
   \end{align*}
   This proves \eqref{upperbdforf}.
   
   \medskip
   
Finally, to establish the lower bound in Theorem \ref{locallowerboundsthm}, we simply chose  
   \[ a_m= \begin{cases} 1 \text{ if }  2^{-N+5} \fz_m\in Q
   \\
0 \text{ if }  2^{-N+5} \fz_m\notin Q.
\end{cases} 
\]
Since $\{\fz_m\}$ enumerates $\bbZ^d$ and $\# (2^{-N+5} \bbZ^d\cap Q)  \approx 2^{Nd} |Q|
$ we obtain \[\|f_N\|_{B^{d(\frac 1p-1)}_{p,q}} 
\lc (2^{Nd} |Q|)^{1/q}\]
from \eqref{upperbdforf}, and 
\[\|\bbE_N f_N\|_{B^{d(\frac 1p-1)}_{p,q}} \gc (2^{Nd} |Q|)^{1/p}\] 
from \eqref{lowerbdforENf}. This establishes the desired  lower bound for all $q\geq p$. \qed

\subsection{\it Proof of upper bounds in Theorem \ref{locallowerboundsthm} (ii)}
In what follows  let  $Q$ be a dyadic cube of side length $\ge 1$. We assume $\frac{d}{d+1}\leq p\le q\le 1$.

We use the global estimates
\eqref{PiN}, \eqref{SENPiN}  and examine the two expressions on the right hand side of \eqref{SENPiN} corresponding to the cases  $j\le N$ and $j\ge N$. The terms for $j\le N$ cause no problem. Namely, by 
Propositions \ref{jleNk>N}
and \ref{jleNkleN}
we have (for $p\le q$)
\[\Big(\sum_{k=0} 2^{k(\frac dp-d) r} \Big\|\sum_{j\le N} L_k\bbE_N ^\perp L_j\Lambda_j f\Big\|_p^r\Big)^{1/r}
\lc \|f\|_{B^{d/p-d}_{p,\infty}} \text{ if } \tfrac{d}{d+1}<p\le 1,
\]
and in the endpoint $q\ge p=\frac{d}{d+1}$  (when $d/p-d=1$) we have 
\begin{multline*} \Big(\sum_{k=0} 2^{k r} \Big\|\sum_{j\le N} L_k\bbE_N ^\perp L_j\Lambda_j f\Big\|_p^r\Big)^{1/r}\\
\lc 
\Big(\sum_{j=0}^N 2^{jp} \|\La_jf\|_p^p\Big)^{1/p}   
\lc N^{\frac 1p-\frac 1q}\|f\|_{B^{1}_{p,q}}, \quad  p=\tfrac{d}{d+1}
\end{multline*}
where we have applied H\"older's inequality.
 This global bound is far better than  what is need for the conclusion and this part satisfies the target upper bound in  \eqref{normequivalencelocal}.

Hence  it suffices to prove, for $f$ supported in $Q$, the following bound
\begin{multline}
\Big(\sum_{k=0}^\infty 2^{kd(\frac 1p-1)r }
\Big\| \sum_{j\ge N+1} L_k \bbE_N L_j\La_j f\Big\|_p ^r \Big)^{1/r} \\
\lc 
(2^{Nd}|Q|)^{\frac 1p-\frac 1q}
 \Big(\sum_{j=0}^\infty 2^{jd(\frac 1p-1)q}\|\La_j f\|_p^q\Big)^{1/q},
\end{multline}
for any $r>0$.
Notice that    Lemma \ref{faraway}  reduces matters to
       show the following inequalities.  
      \begin{multline}\label{kj>N-Besovqloc}
\Big(\sum_{k\ge N+1} 2^{kd(\frac 1p-1)r }
\Big\| \sum_{j\ge N+1} L_k \bbE_N L_j[\zeta_Q \La_j f]
\Big\|_p 
^r \Big)^{1/r} \\
\lc 
(2^{Nd}|Q|)^{\frac 1p-\frac 1q}\Big(\sum_{j=0}^\infty 2^{jd(\frac 1p-1)q}\|\La_j f\|_p^q\Big)^{1/q}
\end{multline}
and
\begin{multline} \label{kleNj>N-Besovqloc}
\Big(\sum_{k\le N} 2^{kd(\frac 1p-1)r }
\Big\| \sum_{j\ge N+1} L_k \bbE_N L_j[\zeta_Q \La_j f]\Big\|_p ^r \Big)^{1/r}\\
 \lc 
(2^{Nd}|Q|)^{\frac 1p-\frac 1q} \Big(\sum_{j=0}^\infty 2^{jd(\frac 1p-1)q}\|\La_j f\|_p^q\Big)^{1/q}.
\end{multline}

We first prove \eqref{kj>N-Besovqloc}.
Instead of using Proposition \ref{jk>N} 
 directly we shall use a modification of its proof in \cite[Proposition 2.1(i)]{gsu};
we first recall some notation  from that paper.

  We let $\sD_N$ be the collection of dyadic cubes of sidelength $2^{-N}$. 
 For $j>N$ we define 
 $\cU_{N,j}$ as in \eqref{UNj}, that is a $2^{-j-1}$-neighborhood of the set $\cup_{I\in\sD_N}\;\partial I$. 
 For $I\in \sD_N$ and $l>N$ we denote by 
$\sD_l[\partial I]$ the set of all $J\in \sD_{l}$ such that 
$\bar{J}\cap\partial I\not=\emptyset$.
Likewise, $\sD_N(I)$ denotes the collection of 
cubes $I'\in\sD_N$ with $\bar{I}\cap \bar{I'}\not=\emptyset$, that is the collection of 
neighboring cubes of $I$.

\medskip

We use the following result taken from \cite[Lemma 2.3]{gsu}.

\begin{lemma}  \label{supportlemma} 
(i) Let $k>N\ge 1$ and $G$ be locally integrable. Then
\Be L_k (\bbE_N G)(x)=0, \quad\text{ for all } x\in \cU_{N,k}^\complement=\bbR^d\setminus\cU_{N,k}\,.\Ee

(ii) Let $j>N\ge 1$, and $F$ locally integrable. 
\Be
\big|\SE_N(L_j F)\big|\lesssim 2^{(N-j)d}\sum_{I\in\sD_N} \sum_{J\in\sD_{j+1}[\partial I]}\|F\|_{L^\infty(J)}\,\bbone_I.
\label{ENjmax}
\Ee
\end{lemma}

 \begin{proof}[Proof of \eqref{kj>N-Besovqloc}]
     
Observe that $F_j:= \zeta_Q \La_j f$ and the functions $L_k \bbE_N L_j[F_j]$ are all supported 
in a fixed $C$-dilate of $Q$ (with say $C=10$). 
 By Lemma \ref{supportlemma}.i, $L_k\bbE_N[L_j F_j](x)=0$ if $x\in\cU_{N,k}^\complement$.
  We derive a pointwise estimate if  $x\in \cU_{N,k}\cap I$ for some $I\in\sD_N$. 
From \eqref{ENjmax} and the fact that $\supp\beta_k(x-\cdot)$ is contained in the union of all $I'\in\sD_N(I)$ we have
 \Beas
|L_k\bbE_N[L_j F_j](x)|& \leq & \int|\beta_k(x-y)|\,\big|\SE_N(L_j F_j)(y)\big|\,dy\\
& \lc &   2^{(N-j)d} \sum_{I'\in\sD_N(I)}\sum_{J\in\sD_{j+1}[\partial I']}
\|F_j\|_{L^\infty(J)}.
\Eeas

Let $Q_*$ be the above $C$-dilate of $Q$. Then 
using $|\cU_{N,k} \cap I | \approx 2^{-N(d-1)-k}$ we have
\begin{align*}
&\Big\| \sum_{j\ge N+1} L_k\bbE_N \big[ 
 L_jF_j \big]\Big\|_p
\\
&\lc\Big(\sum_{\substack{ I\in \sD_N\\ I\cap Q_* \not=\emptyset}}
|\cU_{N,k} \cap I |  \; \big| 2^{Nd} \sum_{j\ge N+1} 
 2^{-jd}\sum_{I'\in\sD_N(I)} \sum_{J\in\sD_{j+1}[\partial I']}\|F_j\|_{L^\infty(J)}\big|^p \Big)^{1/p}
 \\
 &\lc 2^{Nd} 2^{- (N(d-1)+k)/p}  
 \Big(\sum_{\substack{ I\in \sD_N\\ I\cap Q_* \not=\emptyset}}
 \big| \sum_{j\ge N+1} 
 2^{-jd} \sum_{J\in\sD_{j+1}[\partial I]}\|\La_j f\|_{L^\infty(J)}\big|^p \Big)^{1/p}.
\end{align*}
The $I$-sum in the last display contains $\approx |Q|2^{Nd}$ terms. Let $p\le  q\le 1$. Using  H\"older's inequality in this sum we see that 
\begin{multline} \Big\| \sum_{j>N} L_k\bbE_N \big[  L_jF_j \big]\Big\|_p \lc 
 (2^{Nd}|Q|)^{\frac 1p-\frac1q} 
 \times\\ 2^{Nd} 2^{- (N(d-1)+k)/p} 
\Big( \sum_{ I\in \sD_N}\Big|
  \sum_{j>N} 
 2^{-jd}\sum_{J\in\sD_{j+1}[\partial I]}\|\La_j f\|_{L^\infty(J)}\Big|^q\Big)^{1/q}.
 \label{auxiliaryline}
\end{multline}
Consider the maximal function
$\fM_j g(x)= \sup_{|h|_\infty \le 2^{-j+1} }|g(x+h) |$.
Then $\|\fM_j\La_j f\|_p\lc c_p \|\La_j f\|_p $ for all $p>0$.
Moreover, as in \cite[(22)]{gsu}, it holds
\[
\sup_{x\in J} \fM_j g(x)\;\lesssim\;\Big[\;\mint_{J^*}|\cM_j g(x+h)|^p\,dh\Big]^\frac1p,
\]
where $J^*$ is a $C'$-dilate of the cube $J\in\sD_{j+1}$. Therefore,
\begin{align*}
&\Big( \sum_{ I\in \sD_N}
\Big|
  \sum_{j>N} 
 2^{-jd}\sum_{J\in\sD_{j+1}[\partial I]}\|\La_j f\|_{L^\infty(J)}\Big|^q\Big)^{1/q}
\\
&\lc\Big( \sum_{ I\in \sD_N}\Big|
  \sum_{j>N} 
 2^{-jd}\sum_{J\in\sD_{j+1}[\partial I]}2^{jd/p} \|\cM_j [\La_j f]\|_{L^p(J^*)}\Big|^q\Big)^{1/q}.
 \end{align*} Using the embeddings $\ell^p\hookrightarrow\ell^1$ (for the $J$-sum) and 
$\ell^q\hookrightarrow\ell^1$ (for the $j$-sum), and in the second step $\ell^{p/q}\hookrightarrow\ell^1$ (for the $I$-sum), the above quantity is further estimated by 
 \begin{align*}
 &\Big( \sum_{ I\in \sD_N}
  \sum_{j>N} 
 2^{j d(\frac 1p-1)q} \Big( \sum_{J\in\sD_{j+1}[\partial I]}\|\cM_j [\La_j f]\|_{L^p(J^*)}^p \Big)^{q/p} \Big)^{1/q}
 \\
 &\lc\Big( 
  \sum_{j>N} 
 2^{j d(\frac 1p-1)q} \Big( \sum_{ I\in \sD_N}\sum_{J\in\sD_{j+1}[\partial I]}\|\cM_j [\La_j f]\|_{L^p(J^*)}^p \Big)^{q/p} \Big)^{1/q}
\\
 &\lc\Big( 
  \sum_{j>N} 
 2^{j d(\frac 1p-1)q}\|\cM_j [\La_j f]\|_{p}^q \Big)^{1/q}\lc \Big( 
  \sum_{j>N} 
 2^{j d(\frac 1p-1)q}\|\La_j f\|_{p}^q \Big)^{1/q}.
\end{align*}
Inserting this estimate into \eqref{auxiliaryline} we see that
\begin{align*} 
      &\Big(\sum_{k>N} 2^{kd(\frac 1p-1)r }
\Big\| \sum_{j>N} L_k \bbE_N L_j[\zeta_Q \La_j f]
\Big\|_p 
^r \Big)^{1/r} \\
&\lc (2^{Nd}|Q|)^{\frac 1p-\frac 1q} 
\Big( \sum_{k>N}  2^{(N-k) (d-\frac {d-1}p)r}\Big)^{1/r} 
\Big( 
  \sum_{j>N} 
 2^{j d(\frac 1p-1)q}\|\La_j f\|_{p}^q \Big)^{1/q},
\end{align*}
and since the $k$-sum is $O(1)$ in the  larger range $p> \frac{d-1}{d}$   we obtain
  \eqref{kj>N-Besovqloc} for $\frac{d}{d+1}\le p\le q\le 1$.
\end{proof}  
 
 \begin{proof}[Proof of \eqref{kleNj>N-Besovqloc}]

This case is simpler and can be obtained from the individual bounds of $\|L_k \bbE_N L_j[F_j]\|_p$
in Proposition \ref{localizedproposition}. Recall that $F_j=\zeta_Q\La_j f$ and $L_k \bbE_N L_j[F_j]$ are supported in a $C$-dilate of $Q$.

Let $k\le N$. Applying H\"older's inequality, the $q$-triangle inequality and  Proposition \ref{localizedproposition}.i  we now obtain
\begin{multline*}
2^{ks}\Big\| \sum_{j>N} L_k \bbE_N L_j[F_j]
\Big\|_p
\lc 2^{ks}\,|Q|^{\frac 1p-\frac 1q} 
\Big[\sum_{j>N} \big\|L_k \bbE_N L_j[F_j]
\big\|_q^q\Big]^\frac1q
\\ \lc |Q|^{\frac1p-\frac1q}
2^{k(s+d+1-\frac dq) } 2^{-N}
\Big(\sum_{j>N} 
2^{j(\frac dq-d)q} 
\|\La_j f
\|_q^q \Big)^{1/q}.
\end{multline*}
We now use the extension of Young's inequality
\Be
\|\La_j f\|_q\le 2^{j (\frac dp-\frac dq)} \|\La_jf\|_p;
\label{Young}\Ee
see e.g. \cite[2.7.1/3]{Tr83}.
As a result we obtain
\begin{multline*}2^{kd(\frac 1p-1)}\Big\| \sum_{j>N} L_k \bbE_N L_j[\zeta_Q \La_j f]
\Big\|_p\\ \lc 
2^{(k-N) (\frac dp-\frac dq+1) }
(2^{Nd}|Q|)^{\frac1p-\frac1q} \Big(\sum_{j>N} 2^{j(\frac dp-d)q} 
\|\La_j f
\|_p^q \Big)^{1/q}.\end{multline*}
Finally, we may sum over $k\le N$ using that $p\leq q$, 
and therefore obtain \eqref{kleNj>N-Besovqloc}. With this assertion, the  proof of Theorem \ref{locallowerboundsthm} is now complete.
\end{proof}

\section{A strongly admissible enumeration}\label{admissiblesection}
We give  explicit examples of strongly admissible  enumerations for $\sH_d $.
We define the family of cubes 
\[
\fQ_5=\Big\{\textstyle\prod_{i=1}^d [10\ka_i-5, 10\ka_i+5)\mid \ka\in \bbZ^d\Big\}.
\]
For $\ell=0,1,2,\dots$, let $\fQ_5(\ell)$ be a strictly  increasing collection  of finite families of  cubes from $\fQ_5$ such that 
for each cube in $\fQ_5(\ell)$ all its neighboring cubes
in  $\fQ_5$  belong to $\fQ_5({\ell+1})$, and such that $\fQ_5=\cup_{\ell} \fQ_5(\ell)$.

{\it Example. }
We may take $\fQ_5(\ell)$ to be family of all  $Q\in\fQ_5$ such that 
$Q\subset [-10 \ell-5, 10\ell+5)^d$. 

Let $\cA_0=[-5,5)^d$, and for $\ell\ge 1$ let $\cA_\ell$ be the union of cubes in $\fQ_5$ which belong to $\fQ_{5}(\ell) \setminus \fQ_5(\ell-1)$. 
For $\ell\ge 0$,  let $\sH_d (\ell,0)$ be the family of characteristic functions of 
 dyadic unit cubes contained in $\cA_\ell$.
 For $k\ge 1$, $\ell\ge 0$,  let $\sH(\ell,k)$ be the family  of Haar functions of mean value $0$ and Haar frequency $2^{k-1}$ with the property that the interior of their  support  is contained in $\cA_\ell$.
Clearly, $\sH_d =\cup_{\ell,k\geq0}\sH(\ell,k)$.

Let $N(\ell,k)=\# \sH_d (\ell,k)$.
We then have $N(0,0)=10^d$ and 
$$N(\ell,k)= N(\ell,0) 2^{(k-1)d}(2^d-1).$$  In the specific example above 
the sets $\cA_\ell$, $\ell\ge 1$, are corridors of width $10$, of the form
$[-10\ell-5, 10\ell+5)^d\setminus
[-10\ell+5, 10\ell-5)^d$ and we have 
$N(\ell,0)= 10^d ((2\ell+1)^d-(2\ell-1)^d)$.

We now define an admissible enumeration $\cU$ associated with this collection. 
Let $P(m)=\sum_{i=0}^{m} N(m-i,i)$,  for $m=0,1,2,\dots$, and let 
\Be
R(m)=\sum_{j=0}^mP(j),
\label{Rm}
\Ee
so that $R(m+1)-R(m)=P(m+1)$.
First, for $n=1,\dots, R(0)$ we enumerate the functions in  $\sH(0,0)$.
Next, for $n=R(m)+1,\dots, R(m+1)$ we enumerate the functions in $\cup_{i=0}^{m+1} \sH(m+1-i,i)$ as follows:
when $$R(m)+1\le n\le R(m)+ N(m+1,0)$$ we enumerate the functions in $\sH_d (m+1,0)$; subsequently, for 
each  $\nu=1,\dots, m+1$, when 
\[ R(m)+ \sum_{i=0}^{\nu-1} N(m+1-i,i)+1\le n\le 
R(m)+ \sum_{i=0}^{\nu} N(m+1-i,i)\]
we enumerate the functions in 
$\sH_d (m+1-\nu,\nu)$.

That is, the functions in $\sH(\ell',k')$ occur earlier than those in $\sH(\ell,k)$ if $\ell'+k'<\ell+k$.
Moreover, $\sH(\ell+1,k-1)$ also occurs earlier than $\sH(\ell,k)$.
Now, if $u_n$ and $u_{n'}$ are both supported in $I^{**}$, the five-fold dilate of a fixed unit cube $I$,
then their supports must be contained in cubes from $\cA(\ell)\cup\cA(\ell+1)$, for some smallest $\ell\geq0$.
Moreover, if $|\supp u_{n'}|\geq 2^d|\supp u_{n}|$, that is, $k(n')\leq k(n)-1$, then the above observations imply that $u_{n'}$ must occur before $u_{n}$.
Thus, the enumeration we just constructed for $\sH_d $  is strongly admissible  with $b=1$. 

\

In the next section it will be convenient to notice that, for the enumeration above, we have
\Be S_{R(m)} f= \bbE_{m-\ell}f, \quad\text{ if } \supp (f) \subset
\cA_\ell \; \text{ and $\ell\le m$}.
\label{SRm}
\Ee
In particular we have $S_{R(m)} f=\bbE_m f$ if $f$ is supported in $(-5,5)^d$.

\section{Failure of convergence for strongly admissible enumerations}
\label{failurestronglyadmsect}

In this section we prove the remaining negative results for the Schauder basis property, 
as stated in Theorem \ref{schauder-besov}; namely the cases
\Benu
\item[(a)] $\;s=\frac dp-d$, $\frac d{d+1}\leq p\leq 1$ and $0<q<p$
\item[(b)] $\;s=\frac 1p-1$, $1<p<\infty$ and $0<q\leq 1$.
\Eenu
We remark that in these cases the operators $\SE_N$ are uniformly bounded, by Theorem \ref{basic-sequence-besov}
(iii) and (vi), and local positive results hold by Theorem \ref{schauderB-bourdaud:d/p-d}.
We disprove the possibility that the admissible enumerations in \S\ref{admissiblesection}
may be \emph{global} Schauder bases in $B^s_{p,q}(\SR^d)$. It suffices to show that the corresponding
partial sum operators $S_R$ are not uniformly bounded.
 
\subsection{\it The case $0<p\leq 1$}

\begin{proposition}\label{admissiblecounterex}
Let $0<q<p\leq 1$. Then, for the strongly admissible enumerations defined in \S\ref{admissiblesection} 
 we have
\Be \label{opnorminadmissible}
\sup_{R\in \bbN} \sup\Big\{ \|S_R f\|_{B^{d(\frac 1p-1)}_{p,q}} :\,
\|f\|_{B^{d(\frac 1p-1)}_{p,q}} \le 1\Big\} =\infty.
\Ee
\end{proposition}

\begin{proof}
We shall use a similar notation as in \S\ref{lowbdslocaldp-1sect}.
Consider functions $g_l$ as defined in \eqref{gldef}. 
Fix $j \gg m$, and for $\ell\le m$ pick $\fz_\ell \in \bbZ^d$ so that the threefold dilate of the cube $\fz_\ell+[0,1)^d$ is contained in $\cA_\ell$. 
Define
\Be
f_{m,j} (x)= \sum_{\ell=1}^m g_j(x-\fz_\ell).
\label{fmj}
\Ee
Note that the summands $g_j(\cdot-\fz_\ell)$ have disjoint supports in $\cA_\ell$.
By \eqref{SRm}
\[
S_{R(m) } f_{m,j} =  \sum_{\ell=1}^m \bbE_{m-\ell} [g_j(\cdot-\fz_\ell)]=\sum_{\ell=1}^m h_{m-\ell}(\cdot-\fz_\ell)\,,
\]
where $h_N$ was defined in \eqref{hNdef}.

Let $\Psi_N$ be defined as in \eqref{PsiNdef}, so that by \eqref{Psicd} we have \[
\|\Psi_{N+1} * h_N\|_p\gc 2^{-Nd(\frac 1p-1)}.\]
Then
\Beas
 \|S_{R(m)} f_{m,j}\|_{B^{d(\frac 1p-1)}_{p,q}}
 &\ge &
  \Big(\sum_{N=1}^\infty 2^{N(\frac dp-d)q } \|\Psi_{N+1}* S_{R(m)} f_{m,j}\|_p^q\Big)^{1/q} 
	\Eeas
	\Beas
\quad   &=&
  \Big(\sum_{N=1}^\infty 2^{N(\frac dp-d)q }
  \Big(\sum_{\ell=1}^m \big\|\Psi_{N+1}*
  h_{m-\ell}(\cdot-\fz_\ell)\big\|_p^p\Big)^{q/p} \Big)^{1/q}\\
   & \geq &
  \Big(\sum_{N=1}^{m-1} 2^{N(\frac dp-d)q }
  \big\|\Psi_{N+1}* h_{N} \big\|_p^q\Big)^{1/q}\gc m^{1/q}.
  \Eeas
Similarly, using the inequality in \eqref{Ljgjdt}, that is
 \[
2^{kd(\frac 1p-1)}\|L_k g_j\|_p \lc 2^{-|j-k| \dt},
\]
for some $\dt>0$, we may conclude that
\Beas \|f_{m,j}\|_{B^{d(\frac 1p-1)}_{p,q}}&  \lc & \Big(\sum_{k=0}^\infty 2^{kd(\frac1p-1)q}
\|L_k(f_{m,j})\|_p^q\Big)^\frac1q \\
& = &  \Big(\sum_{k=0}^\infty 2^{kd(\frac1p-1)q}
\Big(\sum_{\ell=1}^m \|L_kg_j(\cdot-\fz_\ell)\|_p^p\Big)^\frac qp\Big)^\frac1q \\
& \lesssim &\Big(\sum_{k=0}^\infty 2^{-|j-k| \dt q}\Big)^\frac1q\,  m^{1/p}\,\lesssim\, m^{1/p}. \Eeas
   Hence, the left hand side in 
   \eqref{opnorminadmissible} is $\gc m^{1/q-1/p}$ which implies the assertion if $q<p$.
   \end{proof}
  
 \subsection{\it The case $1<p<\infty$} \label{S10.2}

We shall deduce this case from the previous one. First of all notice that Proposition \ref{admissiblecounterex} remains to be valid when $1<p<\infty$. Indeed, the condition on $p$ did not play any role in the proof.
In particular, if the dimension $d=1$ this implies \Be
\sup_{R\in \bbN}\; \big\|S_R f\|_{B^{\frac 1p-1}_{p,q}\to B^{\frac 1p-1}_{p,q} }  =\infty, \quad \mbox{when }\; 0<q\leq 1<p.
\label{SRp1}
\Ee
To establish the same result for $d\geq2$, we tensorize the previous example. 
Consider \[
F_{m,j}(x_1,x')=f_{m,j}(x_1)\,\chi(x'),
\]
where $f_{m,j}$ is the 1-dimensional function in \eqref{fmj},
and $\chi\in C^\infty_c((-2,2)^{d-1})$ with $\chi\equiv1$ in $[-1,1]^{d-1}$.
We claim that, for $s=1/p-1$ and $0<q\leq 1<p$, we have
\Be
\label{d1}
\|F_{m,j}\|_{B^s_{p,q}(\SR^d)} 
\lesssim m^{1/p}
\Ee 
and 
\Be
\label{d2}
\|S_{R(m,d)}(F_{m,j})\|_{B^s_{p,q}(\SR^d)} \gtrsim \|S_{R(m,1)}(f_{m,j})\|_{B^s_{p,q}(\SR)} \gtrsim m^{1/q}.
\Ee
Here $R(m,d)$ are the numbers in \eqref{Rm}, where we stress the dependence on the dimension. Notice that in either case they verify \eqref{SRm}.


To justify these inequalities, we construct a function $\Psi\in C^\infty_c(\SR^d)$ as in \eqref{Psitensor}, that is
\Be
\Psi(x)= 
\Delta^M[ \phi_0\otimes \varphi_0] (x)= \theta(x_1) \varphi_0(x') + \phi_0(x_1) \vartheta(x'),
\label{Psi3}
\Ee
for suitable $\phi_0$, $\varphi_0$, $\theta$, $\vartheta$ as in the paragraph preceding \eqref{Psitensor}. We let
\[
 \Psi_0(x)=\phi_0(x_1)\varphi_0(x'),\mand \Psi_k(x)=2^{kd}\Psi(2^kx), \;\;k\geq1.
\]
These functions meet the required hypothesis to have \[
\big\|g\big\|_{B^s_{p,q}(\SR^d)}\,\approx\,\Big(\sum_{k=0}^\infty2^{ksq}\big\|\Psi_k*g\big\|_{L^p(\SR^d)}^q\Big)^{\frac1q}\, .
\]
Moreover, if we define, for $k\geq1$,
\[
\phi_k(x_1)=2^k\theta(2^kx_1)\mand \varphi_k(x')=2^{(d-1)k}\vartheta(2^kx'),
\]
then the convolutions with $\phi_k$ (respectively $\varphi_k$), $k=0,1,2,\ldots$,
can be used to characterize the norms of $B^s_{p,q}$ in $\SR$ (respectively in $\SR^{d-1}$). Using this notation in \eqref{Psi3} we can now write
\Be
\Psi_k=\phi_k\otimes \varphi_{0,k}\,+\,\phi_{0,k}\otimes\varphi_k,
\label{PSIk}
\Ee
with $\phi_{0,k}(x_1)=2^k\phi_0(2^kx_1)$ and likewise for $\varphi_{0,k}$.

We now prove \eqref{d2}. First, using \eqref{SRm} one easily sees that
\[
S_{R(m,d)}(F_{m,j})\,=\,(S_{R(m,1)} f_{m,j})\otimes(\SE^{(d-1)}_m[\chi]).
\]
Moreover, we claim that
\Be
\Psi_k*(S_{R(m,d)}F_{m,j})(x_1, x')\,=\,\phi_k*(S_{R(m,1)} f_{m,j})(x_1),\quad x'\in(\tfrac14,\tfrac34)^{d-1}.
\label{claimPsiSR}
\Ee
Indeed, this is a direct consequence of \eqref{PSIk} and 
\[
\varphi_{0,k}*(\SE^{(d-1)}_m\chi)(x')=\int\varphi_{0,k}=1 \mand \varphi_{k}*(\SE^{(d-1)}_m\chi)(x')=\int\varphi_{k}=0.
\]
Then \eqref{claimPsiSR} implies the first inequality in \eqref{d2}, and from the 1-dimensional result one obtains the second inequality.

We now prove \eqref{d1}. If $k\geq1$ we can write
\Bea
\Psi_k*F_{m,j} & = & (\phi_k*f_{m,j})\otimes(\varphi_{0,k}* \chi)+(\phi_{0,k}*f_{m,j})\otimes(\varphi_k*\chi)\nonumber\\
& = & A_k+B_k\label{AkBk}
\Eea
(a similar formula holds for $k=0$). Then
\Be
\|A_k\|_p\lesssim \|\phi_k*f_{m,j}\|_p\mand \|B_k\|_p\leq \|\phi_{0,k}*f_{m,j}\|_p\,\|\varphi_k*\chi\|_p.
\label{ABk}
\Ee
From the previous calculation in one dimension we have
\[\|\phi_k*f_{m,j}\|_p\lc 2^{-|k-j|\delta} 2^{k(1-\frac 1p)} m^{1/p}.
\]

We estimate the term
\[
\|\phi_{0,k}*f_{m,j}\|_p=\Big(\sum_{\ell=1}^m\|\phi_{0,k}*g_j\|_p^p\Big)^\frac1p=m^\frac1p\, \|\phi_{0,k}*g_j\|_p.
\]
Now, if $k\geq j$ then
\[
\|\phi_{0,k}*g_j\|_p\leq\|\phi_{0,k}\|_1\,\|g_j\|_p\lesssim 2^{(1-\frac1p)j}\leq 2^{k(1-\frac1p)}.
\]
On the other hand, if $k<j$ then 
\[
\|\phi_{0,k}*g_j\|_p\leq \|\phi_{0,k}\|_p\,\|g_j\|_1\lesssim 2^{k(1-\frac1p)}.
\]
Thus, \[
\|B_k\|_p\lesssim \, m^{1/p}\, 2^{k(1-\frac1p)}\,\|\varphi_k*\chi\|_p,
\]
which can be inserted into \eqref{ABk}, and overall will imply
\[
\|F_{m,j}\|_{B^{\frac1p-1}_{p,q}}\;\lesssim \;\|f_{m,j}\|_{B^{\frac1p-1}_{p,q}}\;+\;m^{\frac1p}\,\|\chi\|_{B^{0}_{p,q}}\;\lesssim\;m^{\frac 1p}.
\]
This completes the proof of \eqref{d1}, and hence of \eqref{SRp1} for all $d>1$.

\section{Failure of unconditionality when $s=d/p-d$} 
\label{uncondfailurep<1}

Theorem \ref{schauder-besov} states that strongly admissible  enumerations of $\sH_d $ form a Schauder basis of
$B^{d/p-d}_{p,p}$ when $\frac{d}{d+1}< p\le 1$. We show that the stronger 
conclusion of unconditionality fails.
The argument will also apply to the Triebel-Lizorkin spaces $F^{d/p-d}_{p,q}$ and therefore we cover this case at the same time.

\begin{theorem}\label{failureuncond-d/p-d}
For every $N\geq1$, there is a collection $A(N)$ of Haar functions, all  supported in $[0,1]^d$, with $\# (A(N))\le  2^d N$, and such that the orthogonal projection  operators $P_{A(N)}$ satisfy the estimates
\begin{align*}
&\|P_{A(N)} \|_{B^{d(\frac 1p-1)}_{p,q}\to B^{d(\frac 1p-1)}_{p,q}} \gc N^{1/q},
\\
&\|P_{A(N)} \|_{F^{d(\frac 1p-1)}_{p,q}\to F^{d(\frac 1p-1)}_{p,q}} \gc N^{1/p}.
\end{align*}
\end{theorem}

We shall use the following well-known identity.

\begin{lemma}\label{haar_id} For  $N=1,2, \dots$, it holds
\Be
2^{Nd} \bbone_{I_{N,0}}\,=\,\bbone_{I_{0,0}}+ \sum_{k=0}^{N-1} 2^{kd}\sum_{\ep\in \Upsilon} h^{\ep}_{k,0}\,.
\label{haar2N}
\Ee
\end{lemma}
\begin{proof} 
The formula follows easily computing the Haar coefficients of the function on the left hand side of \eqref{haar2N}.
\end{proof}

Let $F_N(x)=2^{Nd}\bbone_{[0,2^{-N})^d}(x)$, and let $G_N$ be its odd extension
$G_N(x)= F_N(x)- F_N(-x)$. 
Consider the finite dimensional subspace \Be
A(N)= \text{span} \Big( \{\bbone_{[0,1)^d}\}\cup \bigcup_{k=0}^{N-1} \{ h^{\ep}_{k,0}: \ep\in \Upsilon\}\Big),
\label{ANspan}
\Ee
which has dimension $\dim A(N)=(2^d-1)N+1$. Let $P_{A(N)}$ be the orthogonal projection onto $A(N)$. Then, by Lemma \ref{haar_id}, $F_N\in A(N)$ and
$$P_{A(N)}\big( G_N\big)= F_N.$$
The failure of unconditionality follows now from
\begin{proposition}\label{P11} Let $\frac {d-1}{d}<p<\infty$, $q>0$. Then, for large $N$, 
\begin{subequations}
\begin{align}
\|G_N\|_{B^{\frac dp-d}_{p,q}} &\lc 1\,,
\label{GNupB}
\\
\|G_N\|_{F^{\frac dp-d}_{p,q}}  &\lc 1\,,
\label{GNupF}
\end{align}
\end{subequations}
and
\begin{subequations}
\begin{align}
&\|P_{A(N)} G_N\|_{B^{\frac dp-d}_{p,q}} \gc N^{1/q}\,,
\label{GNlowB}
\\
&\|P_{A(N)} G_N\|_{F^{\frac dp-d}_{p,q}} \gc N^{1/p}\,.
\label{GNlowF}
\end{align}
\end{subequations}
\end{proposition}

\begin{proof}
Since $\|\psi_0*G_N\|_p\lc 1$ for any $\psi_0\in\cS$, we only need to estimate the terms involving $\psi_k*G_N$, with $k\geq1$, in the $B^s_{p,q}$ or $F^s_{p,q}$ quasi-norms.
Assume that $\psi_k(x)= 2^{kd} \psi(2^kx)$, where $\psi\in C^\infty_c((-1,1)^d)$ is such that
$\psi(x)\ge 1$ for $x\in (-1/2,-1/8)^d$, and $\psi$ has sufficient vanishing moments (to characterize the involved B and F norms). 
For $k\geq1$, we  analyze $\psi_k* G_N$. Note, that $G_N$ is supported on $[-2^{-N},2^{-N}]^d$. Since 
$\int G_N(x) dx=0$ we have \[
|\psi_k*G_N(x)|\lc 2^{kd} 2^{k-N}\bbone_{[-2^{-k+1},2^{-k+1}]^d},\quad\mbox{ for }k\le N;
\]
 see \eqref{moment_gl}. 
Hence
\Be
2^{k(\frac dp-d)} \|\psi_k* G_N\|_p \lc 2^{k-N} , \quad k\le N.
\label{psiGN1}
\Ee
For $k>N$  let $D_{N}$ be the boundary of $I_N\cup -I_N$. Then 
$\psi_k*G_N$  is supported in a $C 2^{-k}$ neighborhood $\cN_{k,N}$ of $D_N$ and $\psi_k*G_N=O(2^{Nd})$ on
$\cN_{k,N}$. The measure of $\cN_{k,N}$ is $O(2^{-N(d-1)-k})$ and therefore we obtain
$$2^{k(\frac dp-d)} \|\psi_k* G_N\|_p \lc 2^{-(k-N)(d-\frac{d-1}{p})}, \quad k\ge N.$$
Since $p>\frac{d-1}{d}$ we can sum the estimates and obtain \eqref{GNupB}.

Similarly 
\begin{align*}
&\Big\|\Big(\sum_{k=1}^N |\psi_k*G_N|^q 2^{k(\frac dp-d)q}\Big)^{1/q}\Big\|_p
\\
&\lc \Big\|\Big(\sum_{k=1}^N |2^{kd} 2^{k-N} \bbone_{|x|\lc 2^{-k}} |^q 2^{k(\frac dp-d)q}\Big)^{1/q}\Big\|_p\lc 1
\end{align*}
and
\begin{align*}
&\Big\|\Big(\sum_{k=N+1}^\infty |\psi_k*G_N|^q 2^{k(\frac dp-d)q}\Big)^{1/q}\Big\|_p
\\
&\lc 2^{Nd} \Big\|\Big(\sum_{k=N+1}^\infty 2^{k(\frac dp-d)q} \bbone_{\cN_{k,N}} \Big)^{1/q}\Big\|_p
\\
&\lc 2^{Nd} \Big(\sum_{N\le l<\infty} \meas(\cN_{l,N})\Big( \sum_{N\le k\le l} 2^{k (\frac{d}{p}-d)q}\Big)^{p/q}\Big)^{1/p}
\\
&\lc 2^{Nd} \Big(\sum_{N\le l<\infty} 2^{-l-N(d-1)}\Big( \sum_{N\le k\le l} 2^{k (\frac dp-d)q}\Big)^{p/q}\Big)^{1/p}\lc 1.
\end{align*}
Observe that the last inequality requires a slightly different argument in each of the  cases 
$\frac{d-1}{d}<p<1$, $p=1$ and $p>1$; we leave details to the reader.
This proves \eqref{GNupF}.

We now include the lower bound for $P_{A(N)}G_N=F_N\equiv 2^{Nd}\bbone_{I_{N,0}}$.
Let $$\Omega_k = \Big(-\frac{3/8}{2^k}, -\frac{1/8}{2^k}\Big)^d 
.$$
Then, for $4\le k\le N-4$,
$$\psi_k*F_N(x)= \int 2^{kd} \psi(2^k (x-y)) 2^{Nd}\bbone_{I_{N,0}}(y) dy\, \ge\,  2^{kd},\quad \text{ for } x\in \Om_k,$$
due to $2^k(x-[0,2^{-N}]^d)\subset (-1/2,-1/8)^d$ and the assumptions on $\psi$. Hence $$2^{k(\frac dp-d)} \|\psi_k*F_N\|_p\gc 1, \quad 4\le k\le N-4,$$
 which implies \eqref{GNlowB}. Also
 \begin{align*} \|F_N\|_{F^{\frac dp-d}_{p,q}} 
 \gc \Big(\sum_{k=4}^{N-4} \int_{\Om_k} 2^{k(\frac dp-d)p} 2^{k dp}\Big)^{1/p}\gc N^{1/p}
 \end{align*}
and \eqref{GNlowF} follows.
\end{proof}

\section{Failure of unconditionality when $s=1/p-1$, $1<p<\infty$.}\label{uncondfailurep>1}
In dimension $d=1$ the failure of unconditionality of $\sH$ in $B^{\frac 1p-1}_{p,q}(\SR)$ is already contained in Proposition \ref{P11}.
 As happened in \S\ref{S10.2}, the argument for $d\geq2$ requires a slight  variation of the above.

We consider the finite dimensional space
\[
\cA(N):= \Span\Big\{ h\otimes \bbone_{[0,1]^{d-1}}\mid h\in A^{(1)}(N)\Big\},
\]
where $A^{(1)}(N)$ is the subspace defined in \eqref{ANspan} (when $d=1$).
Note that $\dim \cA(N)=\dim A^{(1)}(N)\approx N$.
We now have

\begin{theorem}\label{failureuncond-1/p-1}
Let $1<p<\infty$. 
Then 
\[ \|P_{\cA(N)} \|_{B^{\frac 1p-1}_{p,q}\to B^{\frac 1p-1}_{p,q}} \gc N^{1/q}.
\]
In particular, $\cH^d$ is not unconditional in $B^{\frac 1p-1}_{p,q}(\SR^d)$ for any $q>0$.
\end{theorem}
\begin{proof}
We keep the notation $F_N$ and $G_N$ for the 1-dimensional functions in Proposition \ref{P11}. We fix $\chi\in C^\infty_c((-1,2)^{d-1})$ with $\chi\equiv1$ in $[0,1]^{d-1}$, and define
\[
g_N(x) = G_N(x_1)\chi(x') \mand f_N(x)=F_N(x_1)\bbone_{[0,1]^{d-1}}(x').
\]
Observe that $f_N\in\cA(N)$ and
\[
P_{\cA(N)}(g_N)=f_N,
\] 
by our choice of $\chi$.
So, it suffices to show that, for large $N$,
\begin{align}
\|g_N\|_{B^{\frac 1p-1}_{p,q}}&\lc 1\,
\mand
\|f_N\|_{B^{\frac 1p-1}_{p,q}}\gc N^{1/q}\,.
\label{fNlo}
\end{align}
The  first assertion is proved as in \S\ref{S10.2}; namely, one constructs
functions $\Psi_k$ as in \eqref{PSIk} and observes that
\Beas
\Psi_k*g_N& = & (\phi_k*G_N)\otimes(\varphi_{0,k}*\chi)+(\phi_{0,k}*G_N)\otimes(\varphi_k*\chi)\nonumber\\
& = & A_k+B_k.
\Eeas
A similar proof as the one following \eqref{ABk} gives 
\[
\|A_k\|_p\lesssim \|\phi_k*G_N\|_p\mand \|B_k\|_p\lesssim 
2^{k(1-\frac1p)}\,\|\varphi_k*\chi\|_p.
\]
From here and the 1-dimensional results in \eqref{GNupB} it follows that
\[
\|g_N\|_{B^{\frac 1p-1}_{p,q}(\SR^d)}\lesssim \|G_N\|_{B^{\frac 1p-1}_{p,q}(\SR)}
+\|\chi\|_{B^{0}_{p,q}(\SR^{d-1})}\lesssim1.
\]
Likewise, to prove the second assertion in \eqref{fNlo} one uses\[
\Psi_k*f_N(x_1,x')=\phi_k*F_N(x_1),\quad x'\in(\tfrac14,\tfrac34)^{d-1}.
\]
This identity, as before, follows  from \eqref{PSIk} and the facts
\[
\varphi_{0,k}*\bbone_{[0,1]^{d-1}}(x')=\int\varphi_{0,k}=1 \mand \varphi_{k}*\bbone_{[0,1]^{d-1}}(x')=\int\varphi_{k}=0,
\]
because the supports of $\varphi_{0,k}(x-\cdot)$ and $\varphi_k(x-\cdot)$ are contained in $[0,1]^{d-1}$ for such values of $x'$. Thus,
\[
\|f_N\|_{B^{\frac 1p-1}_{p,q}(\SR^d)}\gtrsim \|F_N\|_{B^{\frac 1p-1}_{p,q}(\SR)}
\gtrsim N^{1/q}.
\]
This establishes \eqref{fNlo} and completes the proof of Theorem \ref{failureuncond-1/p-1}.
\end{proof}

\newpage

\begin{thebibliography}{10}

\bibitem{albiac} F. Albiac, J.L. Ansorena, P. Bern\'a, P. Wojtaszczyk.
Greedy approximation for biorthogonal systems in quasi-Banach spaces.
Preprint 2019, arXiv:1903.11651.

\bibitem{bourdaud} G. Bourdaud, Localisations des espaces de Besov. Studia Math. 90 (1988), 153-163.



\bibitem{franke} J. Franke, On the spaces $F^s_{p,q}$ of Triebel-Lizorkin type: Pointwise multipliers and spaces on domains. Math. Nachr. 125, 29--68 (1986).


\bibitem{gsu} G. Garrig\'os, A. Seeger, T. Ullrich, The Haar system as a Schauder basis in spaces of Hardy-Sobolev type, Journal of Fourier Analysis and Applications, 24(5) (2018), 1319--1339.

\bibitem{gsu-wavelet} \bysame, On uniform boundedness of dyadic averaging operators in spaces of Hardy-Sobolev type.
Analysis Math. {43} (2) (2017), 267--278.


\bibitem{gsu-endpt}  \bysame,
The Haar system
in Triebel-Lizorkin spaces: Endpoint results.  Preprint, 2019, arXiv:1907.03738.
\bibitem{oswaldold} P. Oswald, On inequalities for spline approximation and spline systems in the space $L_p, (0< p < 1)$, in {\em Approximation and Function Spaces}, Proc. Int. Conf. Gdansk 1979 (Z. Ciesielski, ed.), PWN Warszawa/North-Holland Amsterdam, pp. 531--552, 1981.

\bibitem{oswald} \bysame, Haar system as Schauder basis in Besov spaces: The limiting cases for $0<p\le 1$, arXiv e-print: arXiv:1808.08156.

\bibitem{Pe}J. Peetre, On spaces of Triebel-Lizorkin type.  Ark. Mat. 13 (1975),123--130.

\bibitem{RuSi96}T. Runst, W. Sickel, \newblock {\em Sobolev spaces of fractional order, {N}emytskij operators, and  nonlinear partial differential equations}, volume~3 of  de Gruyter Series  in Nonlinear Analysis and Applications. \newblock Walter de Gruyter \& Co., Berlin, 1996.



\bibitem{su} A. Seeger, T. Ullrich, Haar projection numbers and failure of unconditional convergence in Sobolev spaces. Math. Z. 285 (2017), 91 -- 119. 


\bibitem{sudet} \bysame. Lower bounds for Haar projections: Deterministic Examples. Constr. Appr. 46 (2017), 227--242. 

\bibitem{triebel73} H. Triebel. \"Uber die Existenz von Schauderbasen in Sobolev-Besov-R\"aumen. Isomorphiebeziehungen.  Studia Math. 46 (1973), 83--100.

\bibitem{triebel78} \bysame,  On Haar bases in Besov spaces. Serdica 4 (1978), no. 4, 330--343.

\bibitem{triebel_teubner} \bysame,  \emph{ Spaces of Besov-Hardy-Sobolev type.} Teubner, Leipzig, 1978.


\bibitem{Tr83} \bysame,  \emph{ Theory of function spaces.} Monographs in Mathematics, 78. Birkh\"auser Verlag, Basel, 1983.



\bibitem{triebel2} \bysame,  \emph{ Theory of function spaces II.} Monographs in Mathematics, 84. Birkh\"auser Verlag, Basel, 1992.





\bibitem{triebel-compl}
 \bysame,
Function spaces in Lipschitz domains and on Lipschitz manifolds. Characteristic functions as pointwise multipliers. 
Rev. Mat. Complut. 15 (2002), no. 2, 475--524. 

\bibitem{triebel-bases} \bysame,  \emph{Bases in function spaces, sampling, discrepancy, numerical integration.} EMS Tracts in Mathematics, 11. European Mathematical Society (EMS), Z\"urich, 2010.



\bibitem{ysy} W. Yun, W. Sickel, D. Yang. {Regularity Properties of the Haar System with Respect to Besov-type Spaces}.
Preprint arXiv:1898.09583.

\end{thebibliography}
\end{document}